\documentclass[11pt]{article}
\usepackage{bbm}
\usepackage[top=1in, bottom=1.5in]{geometry}
\usepackage{color}
\usepackage{comment}
\usepackage{amsfonts}
\usepackage{latexsym, amssymb, amsmath, amscd, amsthm, amsxtra}
\usepackage{mathtools}
\usepackage{enumerate}
\usepackage[all]{xy}
\usepackage{mathrsfs}
\usepackage{fancyhdr}
\usepackage{listings}
\usepackage{hyperref}
\usepackage{float}
\usepackage{algorithm}
\usepackage{algpseudocode}
\usepackage{framed}
\usepackage{lipsum}

\hypersetup{
	colorlinks   = true, %Colours links instead of ugly boxes
	urlcolor     = blue, %Colour for external hyperlinks
	linkcolor    = blue, %Colour of internal links
	citecolor   = red %Colour of citations
}

\newtheorem{theorem}{Theorem}[section]
\newtheorem{lemma}[theorem]{Lemma}

\newtheorem{proposition}[theorem]{Proposition}

\theoremstyle{definition}
\newtheorem{definition}[theorem]{Definition}

\usepackage{tocloft}
\usepackage{times}

\newcommand{\Pb}{\mathbb P}

\newcommand{\ER}{Erd\H{o}s-R\'enyi\ }
\newcommand{\bx}[2]{\mathbf{x}_#1^{(#2)}}

\newcommand{\HD}[1]{\textcolor{cyan}{HD:\ #1}}

\theoremstyle{definition}
\numberwithin{equation}{section}

\title{Sharp threshold for network recovery from \\voter model dynamics}

\author{Hang Du\thanks{Department of Mathematics, Massachusetts Institute of Technology} \and Seokmin Ha\footnotemark[1] \and Oriol Solé-Pi\footnotemark[1]}

\date{\today}

\begin{document}
	
\maketitle

\begin{abstract}
We investigate the problem of recovering a latent directed \ER graph $G^*\sim \mathcal G(n,p)$ from observations of discrete voter model trajectories on $G^*$, where $np$ grows polynomially in $n$. Given access to $M$ independent voter model trajectories evolving up to time $T$, we establish that $G^*$ can be recovered \emph{exactly} with probability at least $0.9$ by an \emph{efficient} algorithm, provided that
\[
M \cdot \min\{T, n\} \geq C n^2 p^2 \log n
\]
holds for a sufficiently large constant $C$. Here, $M\cdot \min\{T,n\}$ can be interpreted as the approximate number of effective update rounds being observed, since the voter model on $G^*$ typically reaches consensus after $\Theta(n)$ rounds, and no further information can be gained after this point. Furthermore, we prove an \emph{information-theoretic} lower bound showing that the above condition is tight up to a constant factor. Our results indicate that the recovery problem does not exhibit a statistical-computational gap.
\end{abstract}

\vskip1mm
\setcounter{tocdepth}{2} 
\tableofcontents

\section{Introduction and main results}

The \emph{voter model} is a classical stochastic process that describes the spread of opinions in a network through local interactions. Originally introduced in the study of interacting particle systems \cite{CS73, HL75} and statistical physics \cite{Kra92, FK96}, it has since found applications in social dynamics \cite{CS73, CFL09}, biological evolution \cite{VES08}, distributed computing \cite{HP01, BCNPP14}, and various other applied fields. In its discrete-time version, each vertex in a given graph holds one of two (or more) possible opinions, and at each time step, every vertex independently updates its opinion by adopting that of a randomly selected neighbor. Due to its simplicity and utility, the voter model is among the best-studied graph-related stochastic dynamical processes.

%Over time, this process leads to consensus, where all vertexs share the same opinion, provided the underlying network structure is well-connected. 

Since the behavior of the voter model is deeply intertwined with the underlying graph structure, a natural question arises: Can the latent graph be inferred by observing the voter model trajectories? In this paper, we explore this problem in the simplified setting where the underlying graph $G^*$ is directed and each vertex can only hold opinions in $\{\pm 1\}$. Let $G^*$ have the vertex set $[n]=\{1,\cdots,n\}$ and the edge set $\vec{E}(G^*)$. The voter model on $G^*$ evolves as follows: First, generate a random initialization $\mathbf{x}^{(0)}=(x^{(0)}_i)_{i\in [n]}$ drawn uniformly from $\{\pm 1\}^n$. Then, for any $t\ge 1$, generate $\mathbf{x}^{(t)}\in \{\pm 1\}^n$ by, independently for each $i \in [n]$, picking $j$ from the out-neighborhood of $i$ in $G^*$ (i.e., the set $\{j\in[n]\,:\, (i\to j) \in \vec{E}(G^*)\}$) uniformly at random and setting $x_i^{(t)}=x_j^{(t-1)}$. %Here, the out-neighborhood of $i\in[n]$ is defined as $\{j\in[n]\,:\, (i\to j) \in \vec{E}(G^*)\}$, and the elements of the out-neighborhood of $i$ are its out-neighbors.
%\HD{I deleted every ``out-neighbor'' so we don't need to define it.} \SH{Thank you! This looks good to me!}

Given two integers $M,T\ge 1$, let $(\mathbf{x}_m^{(0)},\cdots,\mathbf{x}_m^{(T)}),1\le m\le M$ be $M$ independent samples of the voter model trajectories evolving up to time $T$. Our goal is to exactly recover the underlying graph $G^*$ from these observations. Naturally, this task becomes easier as $M$ and $T$ increase since larger values mean we have access to more information. This raises the following question: How large should $M$ and $T$ be to ensure that $G^*$ can be recovered with probability close to $1$?

In this paper, we aim to answer this question in the setting where the underlying graph $G^*$ is a directed \ER graph drawn from $\mathcal G(n,p)$, meaning that each directed edge exists independently with probability $p$. We further assume that $p$ is neither too large nor too small, ensuring that the average degree of $G^*$ remains bounded by a constant time $n$ while still growing polynomially. More precisely, we assume that $p$ satisfies
\begin{equation}\label{eq-p-assumption}
p\le c_0,\quad np\ge n^\delta\,,
\end{equation}
where $c_0>0$ is a universal constant and $\delta>0$ is an arbitrary but fixed constant.

Our first result provides a condition on $M$ and $T$ under which $G^*$ can be recovered exactly with a probability close to $1$. 

\begin{theorem}\label{thm-main}
Let $(\bx{m}{0},\cdots,\bx{m}{T}), 1\le m\le M$ be $M$ independent voter model trajectories with evolution time $T$. Then, assuming \eqref{eq-p-assumption}, there exists $C>0$ depending only on $\delta$ such that if $M$ and $T$ satisfy
	\begin{equation}\label{eq-main-condition}
		M\cdot \min\{T,n\}\ge Cn^2p^2\log n\,,
	\end{equation}
then there is an estimator $\widehat{G}=\widehat{G}\big((\bx{m}{0},\cdots,\bx{m}{T}), 1\le m\le M\big)$ such that $\mathbb{P}[\widehat{G}=G^*]\ge 0.9$. (Here, the randomness is over $G^*\sim \mathcal G(n,p)$ and $(\bx{m}{0},\cdots,\bx{m}{T}), 1\le m\le M$.) Moreover, $\widehat{G}$ can be computed in $\widetilde{O}(n^4p^2)$ time. 
\end{theorem}

{We remark that by making some minor modifications to the arguments in the proof of Theorem~\ref{thm-main}, we can also deduce that $G^*$ can be almost-exactly recovered (resp. partially recovered)\footnote{Here, by almost-exact recovery (resp. partial recovery), we mean that there exists an estimator $\hat{G}$ containing $(1 + o(1))n^2p$ directed edges, of which a $1 - o(1)$ fraction (resp. a positive fraction) also belong to $G^*$.} efficiently with probability close to $1$, provided that $M\cdot \min\{T,n\}=\omega(n^2p^2)$ (resp. $M\cdot \min\{T,n\}=\Omega(n^2p^2)$). Furthermore, all the arguments used during the proof of the positive result extend almost verbatim to the undirected \ER case with discrete time updates.

Our next result provides an information-theoretic lower bound for the exact recovery task, indicating that condition \eqref{eq-main-condition} is optimal up to a multiplicative constant. 

\begin{theorem}\label{thm-negative}
With the same assumptions as above, there exists $c>0$ depending only on $\delta$ such that if $M$ and $T$ satisfy 
\begin{equation}\label{eq-condition-negative}
M\cdot \min\{T,n\}\le cn^2p^2\log n\,,
\end{equation}
then for any estimator $\widetilde{G}=\widetilde{G}\big((\mathbf{x}_m^{(0)},\cdots,\mathbf{x}_{m}^{(T)}),1\le m\le M\big)$ we have that $\mathbb{P}[\widetilde{G}=G^*]\le 0.1$.
\end{theorem}

Theorems~\ref{thm-main} and \ref{thm-negative} imply that the informational-theoretical threshold for the exact recovery task in our setting is given by the relation
\begin{equation}\label{eq-IT-threshold}
 M\cdot\min\{T,n\}\asymp n^2p^2\log n\,. 
\end{equation}
Furthermore, the last statement in Theorem~\ref{thm-main} indicates that, above this threshold, there exists an efficient recovery algorithm. Consequently, this problem does not exhibit a statistical-computational gap.

The fact that we are working in the setting of directed graphs is relevant to the proof of Theorem~\ref{thm-negative}; this is further discussed in Section~\ref{sec-negative}. Nonetheless, we expect the information-theoretic threshold for exact recovery to be given by \eqref{eq-IT-threshold} in the undirected case as well. %\OS{Should I mention here that we expect the threshold to still be the same? This is already mentioned in the "further questions" section}\HD{I think we can mention it here, one sentence would be enough. Do we really mention this in section 1.4? }\OS{I guess we mention implicitly in the first paragraph of Section 1.4, but it seems ok to mention it here explicitly} Unlike for the positive result, extending Theorem~\ref{thm-negative}
%to the setting of partial recovery might require a substantial amount of additional work.\HD{I'm not sure if we should mention this, as the generalization should only be more complicated but not essentially more difficult.}

We emphasize that the probabilities $0.9$ and $0.1$ in the statements of our main theorems \emph{cannot} be replaced with $1-o(1)$ and $o(1)$ in the most general setting. Specifically, consider the special case where $M=1$ and $n^2 p^2 \log n = \Theta(n)$. Here, the exact recoverability of $G^*$ strongly depends on the \emph{consensus time} (the first time at which all vertices adopt the same opinion), which exhibits fluctuations of order $n$ \cite[Theorem 3]{Oli13}. In this case, stronger probabilistic guarantees do not hold. Nevertheless, we believe that outside of this specific scenario (and its variants), the stronger claims with probabilities $1-o(1)$ and $o(1)$ should be valid.

Simulations, using both synthetic and real-world data, suggest that our algorithm from Theorem~\ref{thm-main} performs well even in situations where the underlying graph does not fit into our theoretical framework. The results of these simulations are presented in Section~\ref{subsec-algo-exp}.
\\
\noindent\textbf{Acknowledgement}. We warmly thank Kiril Bangachev and Guy Bresler for suggesting the problem studied in this paper. We are also grateful to Kiril Bangachev, Guy Bresler, and Nike Sun for providing many insightful comments on an earlier version of this manuscript.

\subsection{Backgrounds and related work}

\textbf{Voter model on finite graphs}. The voter model was first introduced on lattices $\mathbb{Z}^d$ by \cite{CS73,HL75} to capture particle interactions in Euclidean spaces, with the primary goal of understanding its stationary distributions-i.e., the long-time behavior of the system. Later, the need to model the spread of opinions in social networks motivated the study of the voter model on finite graphs. In the case of finite graphs, however, the only stationary distributions are usually trivial, so research has primarily focused on the consensus time $T_{\operatorname{cons}}$.
%\emph{consensus time} $T_{\operatorname{cons}}$, the first time at which all vertices adopt the same opinion.

The study of the voter model on finite graphs was initiated in \cite{Cox89} for the $d$-dimensional torus. Since then, the model has been analyzed on various types of graphs, particularly random graphs (see, e.g., \cite[Section 6.9]{Dur07}). A seminal work \cite{Oli13} shows that the consensus time $T_{\operatorname{cons}}$ scales with the expected meeting time of two independent random walks starting from uniformly chosen locations\footnote{In particular, for \ER graph $\mathcal G(n,p)$ with $np$ growing polynomially, the expected meeting time is of order $n$, implying that that the corresponding voter model consensus time is also typically of order $n$.}, and it exhibits certain mean-field properties, provided that the underlying random walk mixes rapidly (see \cite[Theorem 3]{Oli13} for a precise statement). This result applies to a broad class of random graph models, including classical \ER graphs with average degree at least $(\log n)^{1+\varepsilon}$, the giant component of sparse \ER graphs, and many small-world graphs. Under similar mixing conditions, \cite{CCC16} describes the scaling limit of the evolution of the voter model density (i.e., the fraction of vertices holding opinion $1$). Beyond the mean-field setting, \cite{FO22} investigates the voter model on subcritical scale-free random graphs, revealing an intriguing power-law scaling of $T_{\operatorname{cons}}$ in the number of vertices. Despite these advancements, the behavior of the voter model on graphs outside the mean-field regime remains largely unexplored.

On the other hand, the study of the voter model beyond the consensus time has gained more attention recently. For instance, \cite{HLYZ22} establishes certain mean-field properties of the voter model's evolution long before consensus (usually referred to as the Big Bang regime) for graphs satisfying specific ``transience-like'' conditions. Additionally, \cite{Capa24} examines the evolution of discordant edges—edges whose endpoints hold different opinions—in the voter model on sparse random directed graphs. We hope that the ideas and technical results presented in this paper will provide further tools for studying these emerging aspects of voter model analysis.
\\
\noindent\textbf{Learning from dynamics}. Our work also contributes to the growing body of research on \emph{learning from dynamics} \cite[\dots]{NS12,ACKP13,BGS14,HC19,BLMT24,MS24,GMM24}, particularly in the context of learning network structure from dynamical processes \cite{NS12,ACKP13,HC19,MS24}. While classical machine learning theory typically focuses on i.i.d. data samples, the problem of learning from data with inherent dynamical structure has recently gained significant attention. A key motivation for this research direction is that, in many scenarios, a continuous stream of structured data points is more accessible—either easier to generate or more readily available—than i.i.d. data samples. Prior work has largely focused on learning from dynamics governed by either simple update rules (e.g., Glauber dynamics \cite{BGS14,BLMT24,GMM24}) or easily accessible observations (e.g., epidemic models \cite{NS12,ACKP13,HC19}). Since the voter model exhibits both of these advantages, it is a natural setting for studying structure learning problems. The parameter learning task using voter model dynamics has been explored in \cite{YSOKM11}, where the update rule follows unknown probability weights but the network structure is assumed to be known. However, to the best of our knowledge, the present work is the first to address the problem of learning network structure from voter model trajectories.

Another striking phenomenon that shows up when learning from dynamics is that dynamically structured data can sometimes allow one to overcome well-known computational lower bounds that apply to i.i.d. samples (see, e.g., \cite{GMM24}). In our setting, however, we demonstrate that the full trajectory of the voter model dynamics conveys roughly the same amount of information as a bunch of independently drawn samples (where each of the observed dynamics has evolution time $T=1$). We will elaborate further on this point in Section~\ref{subsec-intuition}.

\subsection{Heuristics and challenges}\label{subsec-intuition}
In this subsection, we explain the heuristic behind the threshold~\eqref{eq-IT-threshold} and the key challenges that must be overcome to make it rigorous.

We first introduce some notations. Given a graph $G^*$ on $[n]$, we use $\operatorname{N}_i=\operatorname{N}_i(G^*)$ to denote the out-neighborhood of vertex $i$ in $G^*$, and we write $d_i=|\operatorname{N}_i|$. For $1\le m\le M$, $0\le t\le T_m$, and $i\in [n]$, we let $x_{m,i}^{(t)}\in \{\pm1\}$ be the opinion of vertex $i$ at time $t$ in the $m$-th trajectory. Hereafter we use w.h.p. as a shorthand of ``with high probability'', meaning with probability $1-o(1)$ as $n\to \infty$. \\
\noindent\textbf{Interpretation of threshold~\eqref{eq-IT-threshold}}. We now interpret the threshold \eqref{eq-IT-threshold}. As a starting point, consider the case $T=1$, where we observe $M$ independent vector pairs $(\mathbf{x}_m^{(0)},\mathbf{x}_m^{(1)}),1\le m\le M$, with $\mathbf{x}_m^{(0)}$ drawn uniformly from $\{\pm 1\}^m$ and $\mathbf{x}_m^{(1)}$ generated according to the voter model update rule. For each $i\in [n]$, our goal is to infer the set  $\operatorname{N}_i=\big\{j\in [n]:(i\to j)\in \vec{E}(G^*)\big\}$. We make the following simple observation: For any $1\le m\le M$ and $j\in [n]$,
\begin{equation}
    \mathbb{E}[x_{m,i}^{(1)}x_{m,j}^{(0)}]=\frac{1}{d_i}\sum_{u\in \operatorname{N}_i}\mathbb{E}[x_{m,u}^{(0)}x_{m,j}^{(0)}]=\begin{cases}
        \frac{1}{d_i}\,,\quad&\text{if }j\in \operatorname{N}_i\,,\\
        0\,,\quad&\text{if }j\notin \operatorname{N}_i\,.
    \end{cases}
\end{equation}
Inspired by this, a natural approach to recovering $\operatorname{N}_i$ is to consider the classifiers
\begin{equation}\label{eq-classifiers}
\mathcal S^\dagger_{i\to j}:=\sum_{m=1}^M x_{m,i}^{(1)}x_{m,j}^{(0)}\,, \quad j\in [n]\,,
\end{equation}
and then apply an appropriate thresholding procedure.

We analyze this approach more carefully. It is straightforward to check that for $j\in \operatorname{N}_i$ the classifier $\mathcal S^\dagger_{i\to j}$ has mean $\frac{M}{d_i}$ and variance of order $M$, whereas for $j\in \operatorname{N}_i^c$ it has mean $0$ and variance $M$. Furthermore, one can show that $\mathcal S^\dagger_{i\to j}$ exhibits Gaussian-type fluctuations. Thus, provided that
\begin{equation}\label{eq-IT-threshold-T=1}
\frac{M}{d_i}\ge \sqrt{CM\log n}\ \iff \ M\ge C\log n\cdot d_i^2
\end{equation}
holds for some large enough constant $C$, the union bound yields that w.h.p.,
\begin{equation}\label{eq-thresholding}
\min_{j\in \operatorname{N}_i}\{\mathcal S^\dagger_{i\to j}\}\ge \frac{M}{d_i}-\frac{1}{2}\sqrt{CM\log n}>\frac{1}{2}\sqrt{CM\log n}\ge \max_{j\in \operatorname{N}_i^c}\{\mathcal S^\dagger_{i\to j}\}\,.
\end{equation}
Thus, whenever \eqref{eq-IT-threshold-T=1} holds, the approach of thresholding over $\mathcal S_{i\to j},j\in [n]$ succeeds in recovering $\operatorname{N}_i$ w.h.p.. Moreover, for $G^*\sim \mathcal G(n,p)$ with $p$ satisfying \eqref{eq-p-assumption}, we have that $d_i=(1+o(1))np$ holds w.h.p. and thus \eqref{eq-IT-threshold-T=1} corresponds to \eqref{eq-IT-threshold} with $T=1$. 

For the general case, we have $M\cdot T$ observable pairs $(\mathbf{x}_m^{(t)},\mathbf{x}_m^{(t+1)})$, and once again one can use these pairs to construct classifiers as in \eqref{eq-classifiers}. However, the voter dynamics might lead to consensus, rendering some pairs uninformative. Indeed, as suggested by \cite[Theorem~3]{Oli13}, the consensus time is typically of order $n$ for $G^*\sim \mathcal G(n,p)$, so one should expect the number of \emph{effective} observable pairs (those with $\mathbf{x}_m^{(t)}\neq \pm \mathbf{1}$) to typically be $\Theta(M\cdot \min\{T,n\})$. Thinking idealistically, if all the effective pairs behave like independent samples distributed as in the case $t=0$, 
then threshold \eqref{eq-IT-threshold} emerges naturally. \\ 
\noindent\textbf{Key challenges in the analysis}.
Here we discuss the main challenges we face in analyzing the above heuristics. 

The previous discussion suggests a natural approach to proving the positive result (Theorem~\ref{thm-main}). To recover the outer neighborhood $\operatorname{N}_i$ of $i$, we consider the classifiers  
\begin{equation}\label{eq-S^dagger}
\mathcal{S}^\ddagger_{i\to j}:=\sum_{m=1}^M\sum_{t=0}^{\min\{T,n\}-1}x_{m,i}^{(t+1)}x_{m,j}^{(t)}\,,\quad j\in [n]\,,
\end{equation}
and then use some threshold to distinguish between those vertices $j$ that belong to $\operatorname{N}_i$ and those that do not. Experimental results suggest that these classifiers are highly effective, even for graphs beyond our assumptions\footnote{E.g., the random configuration model and some real-world networks; see Subsection~\ref{subsec-algo-exp} for details.}, yet its theoretical analysis remains challenging. 

A key conceptual issue is that, while the intuition behind why threshold~\eqref{eq-IT-threshold} is correct relies on the effective pairs $(x_{m}^{(t+1)},x_{m}^{(t)})$ behaving approximately as independent samples with $x_m^{(t)}$ being a uniform random vector in $\{\pm 1\}^n$, these pairs are in fact strongly correlated, and $x_m^{(t)}$ becomes increasingly biased as $t$ grows. Another technical obstacle lies in analyzing the thresholding process: Unlike the classifiers in \eqref{eq-classifiers}, $\mathcal{S}^\ddagger_{i\to j}$ is not well-concentrated enough, making it impossible to establish a deterministic threshold for classification, as was done previously in \eqref{eq-thresholding}.

The proof of the negative result (Theorem~\ref{thm-negative}) presents even greater challenges. Conceptually, while the voter model may resemble an i.i.d. process during the early stages of its evolution, its dynamics become highly skewed as it approaches consensus. This naturally raises the question of whether this final phase provides significantly stronger structural information. Theorem~\ref{thm-negative} suggests an essentially negative answer. Establishing this result requires a deep understanding of the entire evolution of the voter model dynamics, making the analysis particularly challenging.

More details on how we tackle these key difficulties are discussed in the next subsection.

\subsection{Technical overview}\label{subsec-technical-overview}
We now turn to a more in-depth discussion of the strategies followed in our proof, highlighting the main novel ingredients.

We begin with the proof of Theorem~\ref{thm-main}. For ease of presentation, we analyze another (closely related) set of classifiers $\{\mathcal S_{i\to j},i,j\in [n]\}$, but the same argument extends straightforwardly via the union bound to the classifiers $\mathcal S^{\ddagger}_{i\to j},i,j\in [n]$. Each $\mathcal S_{i\to j}$ consists of a sub-sum of \eqref{eq-S^dagger} in which any two terms sharing the same index $m$ differ in their time indices by at least $t_*$, a large constant depending on $\delta$ (see Section~\ref{sec-positive} for the precise definitions). Intuitively, the $t_*$ time gap produces a decoupling effect, as we explain below. Using a standard duality between coalescing random walks and the voter model (which we recall in Section~\ref{sec:backtracking}), one can relate the voter model dynamics to coalescing random walks on $G^*$. Meanwhile, for $G^*\sim \mathcal G(n,p)$, it holds w.h.p. that a random walk on $G^*$ mixes well within $O(1)$ steps (see Section~\ref{subsec-admissible-graph}). Hence, at a heuristic level, any two terms $x_{i,m}^{(t_k+1)}x_{j,m}^{(t_k)},k=1,2$ with $|t_1-t_2|$ exceeding a large constant are effectively independent. 

The above heuristic is formalized through our approach to resolving the aforementioned thresholding issue. Instead of proposing a deterministic threshold, we show that for each $i\in [n]$, the set $\{\mathcal{S}_{i\to j}, j\in [n]\}$ naturally clusters into two groups depending on whether $j$ belongs to $\operatorname{N}_i$ (see \eqref{eq-clustering}). This clustering is proved by decomposing each $\mathcal{S}_{i\to j}$ into two components: A martingale part and a drift part. The martingale part is well-concentrated around zero, while the drift part is analyzed using duality to coalescing random walks.
Remarkably, by choosing $t_*$ sufficiently large, the drift part can be uniformly approximated for all pairs $(i,j) \in [n]^2$ in a way that distinguishes whether $j \in \operatorname{N}_i$. Ultimately, the problem reduces to showing that there exist sufficiently many vectors $x_m^{(t)} \in \{\pm 1\}^n$ that are not “extremely biased” toward $1$ or $-1$ (see Lemma~\ref{lem-effective-observations-lower-bound} below). This condition is much weaker than requiring $x_m^{(t)}$ to behave like a uniform random vector, and we prove the final statement using tools from martingale theory. To the best of our knowledge, this provides a novel framework for analyzing voter model dynamics on mean-field networks, and we believe it has broader applicability to other settings. The detailed proof is presented in Section~\ref{sec-positive}.

The proof of Theorem~\ref{thm-negative} is considerably more intricate on a technical level.
We begin with a standard argument that reduces our task to showing that the maximum likelihood estimator (MLE)
fails to recover $G^*$ with probability close to $1$.
To that end, we first derive the log-likelihood $\mathcal{L}(G)$ for a general directed graph $G$ based on the observations
$(\mathbf{x}_m^{(0)}, \cdots, \mathbf{x}_m^{(T)})$ for $1 \le m \le M$.
We then argue that, with probability close to $1$, there exists a ``flipping graph'' $G$ of $G^*$---obtained by removing an edge
$(i \to j)$ in $\vec{E}(G^*)$ and adding an edge $(i \to k)$ not in $\vec{E}(G^*)$---such that $\mathcal{L}(G) > \mathcal{L}(G^*)$.

The main challenge in our analysis stems from the intricate form of $\mathcal{L}(G)$, which necessitates tracking the evolution of $\mathcal{S}_{m,i}^{(t)} = \sum_{j \in \operatorname{N}_i} x_{m,j}^{(t)}$, $0 \le t \le T_m$, for all $1 \le m \le M$ and $i \in [n]$ \emph{simultaneously}. To address this, we first establish a novel ``from global to local'' lemma (Lemma~\ref{lem-FGTL}), which demonstrates that each $\mathcal{S}_{m,i}^{(t)}$ is essentially governed by the single aggregate variable $S_m^{(t)} = \sum_{i \in [n]} x_{m,i}^{(t)}$. Next, we show that for each $m$, the sequence $S_m^{(t)}$, $0 \le t \le T$, behaves analogously to a certain random walk on $\{-n, \cdots, n\}$ (Lemma~\ref{lem-RW}). This correspondence enables us to leverage techniques from random walk analysis---including martingale theory and multi-scale analysis---to gain a deeper understanding of the behavior of $S_m^{(t)}$ and, consequently, $\mathcal{S}_{m,i}^{(t)}$. Notably, while our random walk formulation bears resemblance to the convergence results to Wright--Fisher diffusion obtained in \cite{CCC16}, their results are not directly applicable here, as our primary objective is to analyze the behavior of $S_m^{(t)}$ ``near the edge''. Essentially, all the aforementioned lemmas rely only on certain regularity and mean-field properties of $G^*$, and so we expect them to have further applications in other voter-model related studies. The details are provided in Section~\ref{sec-negative}.

\subsection{Extensions and further questions}
Although we work in a simplified setting where (i) the graph is directed, (ii) update times are discrete, and (iii) vertices are not allowed to select themselves when updating their labels, we believe that each of these assumptions can be relaxed and that our proof strategy should carry over to several other variants of our setting. For example, one could consider another standard version of the voter model in which the underlying graph is undirected, and each vertex adopts the label of a randomly chosen neighbor after waiting for a random time governed by an exponential clock. We expect analogues of Theorems~\ref{thm-main} and \ref{thm-negative} to hold in this setting as well. Moreover, as suggested by the discussion following Theorem~\ref{thm-main}, we believe that the threshold for almost-exact recovery (resp. partial recovery) is given by $M\cdot \min\{T,n\}=\omega(n^2p^2)$ (resp. $M\cdot\min\{T,n\}=\Omega(n^2p^2)$).
%As the reader might expect after reading the discussion following Theorem~\ref{thm-main}, \HD{we also believe that the threshold at which almost-exact recovery (resp. partial recovery) becomes feasible is given by $M\cdot \min\{T,n\}=\omega(n^2p^2)$ (resp. $M\cdot\min\{T,n\}=\Omega(n^2p^2)$).}

On the other hand, we crucially rely on the fact that $G^*$ is a random graph in several key aspects of our analysis. For instance, in the proof of Theorem~\ref{thm-main} we use the fact that the random walk on $G^*$ mixes in $O(1)$-time, while in the proof of Theorem~\ref{thm-negative} we additionally employ several regularity properties of random graphs. We believe these ingredients are, in some sense, necessary for the information-theoretical threshold to be given by \eqref{eq-IT-threshold}. Thus, an important question remains: Under what conditions can a general graph be recovered by observing the voter model dynamics? This problem is likely to be quite challenging. 

For example, we conclude from Theorem~\ref{thm-main} that if $p$ satisfies assumption~\eqref{eq-p-assumption} and $np^2\log n\ll 1$, then it is possible to recover $G^*\sim \mathcal G(n,p)$ with probability close to $1$ from a single trajectory of the voter model, provided that the evolution time is sufficiently large. Additionally, from the discussion in Section~\ref{subsec-intuition}, we know that for any graph with maximal degree $\Delta$, there is an efficient algorithm that can recover $G^*$ with probability close to $1$ after observing $O(\Delta^2\log n)$ many voter trajectories with evolution time $T=1$. A more specific but still intriguing question is whether for a connected sparse graph $G^*$ (e.g., one with maximal degree $O(1)$) it is possible to recover $G^*$ from $O(1)$ independent voter model trajectories, even allowing infinite evolution time\footnote{Simulations suggest that this could be true, see Table~\ref{table1} and Table~\ref{table3} in Section~\ref{subsec-algo-exp}.}.

Yet another seemingly interesting problem consists of understanding if there is any advantage to be gained from selecting the initialization vectors $\mathbf{x}_m^{(0)}$ ($1\leq m\leq M$) either deterministically or according to a non-uniform probability distribution. One could even consider an adaptive setting in which $\mathbf{x}_m^{(0)}$ is chosen based on the first $m-1$ observed trajectories of voter model dynamics, each evolving up to time $T$. A natural question is whether the information-theoretic threshold remains governed by \eqref{eq-main-condition} in these more flexible initialization schemes. We leave these questions as future directions to study.

\section{Preliminaries}\label{sec-preliminary}

In this section, we provide some important preliminaries for the proof of our main results. In Section~\ref{sec:backtracking}, we recall a classic observation that the voter model is dual to coalescing random walks on $G^*$, which will play an essential role in later proof. Section~\ref{subsec-admissible-graph} collects several desirable properties of random graphs $G^*\sim \mathcal G(n,p)$ that will become useful later.

\subsection{Duality to coalescing random walks}\label{sec:backtracking}
We recall an alternative (and very useful) description of the voter model dynamics, referred to as the \textit{duality to coalescing random walks}. %(we remark that this is a common tool in the study of voter-like models, see e.g. \cite{Cap24}). 

For each vertex $i\in [n]$ and each $t\ge 1$, we select an arrow $a_i^t$ pointing from $i$ to an element of $\operatorname{N}_i$ chosen uniformly at random, so that all these arrows $a_i^t$ are chosen independently. Define the \textit{backward path} starting at $i\in [n]$ and time $1\le t\le T$ as the path $\mathcal P_i^t=(i_0,\cdots,i_t)$ on $G$ that starts at $i_0=i$ and then traverses the arrows $a_{i_0}^t,\cdots,a_{i_{t-1}}^1$ in sequence (the endpoint of arrow $a_{i_k}^{t-k}$ is $i_{k-1}$). Let $\mathbf{x}^{(0)}=(x_{i}^{(0)})_{i\in [n]}$ be chosen uniformly from $\{\pm 1\}^n$ and, for all $i\in [n]$ and $t\ge 1$, define $x_i^{(t)}=x_{i_t}^{(0)}$ (where $i_t$ is the end point of the path $\mathcal P_i^t$). Note that the tuple $(\mathbf{x}^{(0)},\cdots,\mathbf{x}^{(T)})$ defined in this way has the same distribution as the voter model trajectories evolving up to time $T$. Thus, we may think of the dynamics as being encoded by the random arrows $a_i^t$, $i\in [n]$, $1\le t\le T$ and the random initialization $\mathbf{x}^{(0)}\in \{\pm 1\}^n$. See Figure~\ref{fig:backtracking-RW} for an illustration.
\begin{figure}
    \centering
    \includegraphics[width=0.7\linewidth]{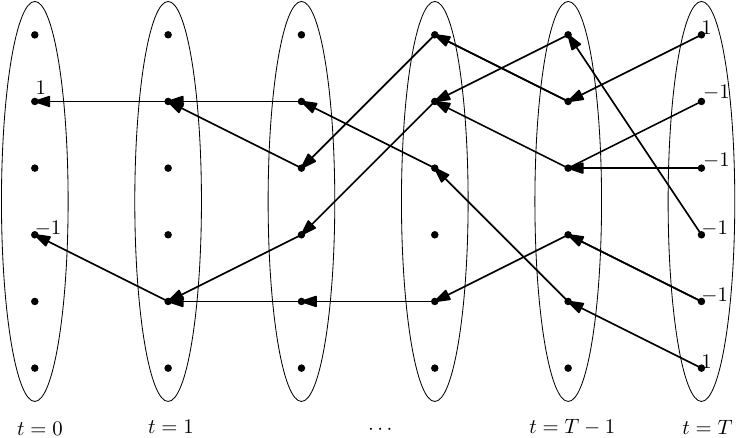}
    \caption{An illustration of the backtracing random walk description}
    \label{fig:backtracking-RW}
\end{figure}

We provide a direct consequence of the duality of coalescing random walks that will become useful later. First, note that for any $i,j\in [n]$, if the two paths $\mathcal P_i^t=(i_0,\cdots,i_t)$ and $\mathcal P_j^t=(j_0,\cdots,j_t)$ meet at some time $t'\le t$, then they coalesce with each other. More precisely, if there exists $t'\le t$ such that $i_{t'}=j_{t'}$, then the paths remain identical all the way up to time $t$ as they traverse the same arrows on and after time $t'$. As a result, if this is the case, then $x_i^{(t)}=x_j^{(t)}$. Otherwise, $x_i^{(t)}=x_{i_t}^{(0)}$ and $x_{j}^{(t)}=x_{j_t}^{(0)}$ for $i_t\neq j_t$. Since $\mathbf{x}^{(0)}$ is a uniform vector in $\{\pm 1\}^n$, we deduce that \begin{equation}\label{eq-p_ij^t}
\mathbb{E}[x_i^{(t)}x_j^{(t)}]=\mathbb P[\mathcal P_i^t\text{ coalesces with }\mathcal P_j^t]\triangleq p_{i,j}^{(t)}\,,\ \forall i,j\in [n],t\ge 1\,.
\end{equation}
%\SH{throughout the paper, we are using both $\triangleq$ and $:=$. Are we distinguishing the meaning of these two notations?}\HD{My intention was to use $A:=B$ to denote we define $A$ as $B$, and to use $A\triangleq B$ to indicate that we denote $B$ as $A$.}\SH{two notations are used correctly in the paper to me}
Since, before coalescing, the paths $\mathcal P_i^t$ and $\mathcal P_j^t$ are independent random walks on $G^*$, $p_{i,j}^{(t)}$ defined as above is simply the probability that two independent random walks starting at $i$ and $j$ meet within $t$ steps. Identity \eqref{eq-p_ij^t} will play a crucial role in our proof.

\subsection{Properties of random graphs}\label{subsec-admissible-graph}
We now collect some desirable properties that a random graph $G^*\sim \mathcal G(n,p)$ satisfies w.h.p., including approximate regularity, fast mixing, and an estimate on the consensus time of voter model dynamics on $G^*$. We refer to graphs possessing these properties as \emph{admissible graphs}. 

\begin{definition}\label{def-admissible}
We say that a directed graph $G^*$ on $[n]$ is \emph{admissible} if the following hold:\\
 \noindent (i) If $d$ is either the in-degree or out-degree  of any vertex in $G^*$, then %\SH{should we define these as well?}\HD{I think these are standard conventions} \SH{ok!}
\begin{equation}\label{eq-degree}
|d-np|\le \sqrt{10np\log n}\,.
\end{equation}
Additionally, for any $i\neq j\in [n]$, if $np^2\ge (\log n)^4$, then 
\begin{equation}\label{eq-overlap-1}
\big||\operatorname{N}_i\cap \operatorname{N}_j|-np^2\big|\le \sqrt{10np^2\log n}\,,
\end{equation}
and if $np^2\le (\log n)^4$, then 
\begin{equation}\label{eq-overlap-2}
|\operatorname{N}_i\cap \operatorname{N}_j|\le 4(\log n)^4\,.
\end{equation}
 \noindent (ii) For any $i\in [n]$, let $\{X_{t,i}\}_{t\ge 0}$ %\SH{in our paper, we are using $()_{t\ge 0}$ for random walk, and using $\{\}_{t\geq 0}$ elsewhere. Do you wanna use one expression or leave it as now? I am not sure about the preferred notation in the field.}\HD{I prefer to change everywhere to $\{\}_{t\ge 0}$. Changed} 
 be a random walk on $G^*$ starting at $i$, and for $t\ge 1$, let $\pi_{t,i}$ denote the law of $X_{t,i}$ (which is a probability distribution on $[n]$). Then, as $t\to\infty$, $\pi_{t,i}$ converges to the (unique) stationary distribution $\pi$ on $[n]$, which satisfies $\pi(i)=1/n+o(1/n)$ for each $i\in [n]$. Moreover, for any integer $k\ge 1$, 
there exists a constant $c=c(k,\delta)$ such that 
\[
\operatorname{TV}(\pi_{t,i},\pi)\leq n^{-k}\,,\ \forall\ t\geq c_{k,\delta},i\in [n]\,.
\] 
 \noindent (iii) There exists $C^*>0$ such that the consensus time of the voter model on $G^*$,  defined as 
 \begin{equation}\label{eq-T-cons}
T_{\operatorname{cons}}:=\min\{t\ge 0:x_i^{(t)}=x_j^{(t)},\forall i,j\in [n]\}\,,
\end{equation}
 satisfies $\mathbb{E}[T_{\operatorname{cons}}]\le C^*n$ (where the randomness is over the voter model trajectories). 
 
 \noindent (iv) For all partitions $[n]=X\sqcup Y$ of the vertex set, the number of triples $i,j,j'$ of elements of $[n]$ such that $j\in X$, $j'\in Y$ and both $(i,j)$ and $(i,j')$ belong to $\vec{E}(G^*)$ is at least $np^2|X|\cdot|Y|/10^{12}$.%\OS{should this be iii or iv?}\HD{I think (iv) is good}
\end{definition}

If $G^*\sim \mathcal G(n,p)$, then $G^*$ is admissible w.h.p.. The first two items are standard properties of random graphs with an average degree growing polynomially in $n$. Moreover, given that the mixing time of a random walk on $G^*$ is $O(1)$ w.h.p., \cite[Theorem 3]{Oli13} implies that the third item holds w.h.p. as well. The last item follows from a standard computation for random graphs using the union bound, and can be thought of as a quantitative version of the statement that $G^*$ is a very good expander. The detailed arguments can be found in Appendix~\ref{appendix-admissible}. Hereafter, unless otherwise specified, we assume that $G^*$ is a deterministic admissible graph.

\section{Proof of Theorem~\ref{thm-main}}\label{sec-positive}
This section contains the proof of Theorem~\ref{thm-main}. We fix an admissible $G^*$ and propose an estimator
$\widehat{G}=\widehat{G}\big(\{(\mathbf{x}_m^{(0)},\cdots,\mathbf{x}_m^{(T)}),1\le m\le M\}\big)$ computed by an $\widetilde{O}(n^4 p^2)$-time algorithm. We show that for a sufficiently large
$C=C(\delta)$, if \eqref{eq-main-condition} holds, then $\widehat{G} = G^*$ with probability at
least $0.92$ (with respect to the randomness coming from the voter model dynamics). Since $G^* \sim \mathcal{G}(n,p)$ is admissible w.h.p., this implies
Theorem~\ref{thm-main}.

We first introduce some notation. Recall Property-(ii) in the definition of admissibility. Write $t_*=c(2,\delta)=O(1)$ and $t_s=st_*$ for $s=0,\cdots,T_*=\lfloor (\min\{T,n\}-1)/t_*\rfloor$. The case $T\le t_*$ can be addressed as in Section~\ref{subsec-intuition} (we simply use the information up to $T=1$ and lose only a constant factor $t_*$). In what follows, we assume that $T\ge t_*+1$ and hence $T_*\ge \frac{1}{2t^*}\min\{T,n\}$. For any two vertices $i,j\in [n]$, we define 
\begin{equation}\label{eq-testing-statistics}
\mathcal S_{i\rightarrow j}:=\sum_{m=1}^{M}\sum_{s=1}^{T_*}x_{m,i}^{(t_s+1)}x_{m,j}^{(t_s)}\,.
\end{equation}

A set $A \subseteq \mathbb{R}$ is called \emph{$2$-clustering} if it can be partitioned into two nonempty subsets whose largest internal gaps are smaller than the smallest gap between elements from different subsets. Given any 2-clustering set, there exists a unique partition with the above properties, and there is a $\widetilde{O}(|A|)$-time algorithm (which we refer to as \emph{$2$-CLUS}) that finds it. Moreover, we may assume that the first subset always contains the larger elements of $A$.

\subsection{The recovery algorithm}
We propose the following simple algorithm.

\begin{algorithm}[H]\label{alg-learning}
	\caption{Recovering $G^*$} 
	\begin{algorithmic}[1]
 \State \textbf{Input}: $(\bx{m}{0},\cdots,\bx{m}{T}),1\le m\le M$.
 \State \textbf{Output}: a graph $\widehat{G}$ on $[n]$.
		\For {$i,j\in [n]$}
	\State Compute $S_{i\rightarrow j}$ defined as in \eqref{eq-testing-statistics}.
		\EndFor
        \For {$i\in[n]$}
    \State Set $A_i=\{\mathcal S_{i\rightarrow j},j\in [n]\}$.
    \State Run $2$-CLUS on $A_i$ to get $A_i=A_{i,1}\sqcup A_{i,2}$.
        \EndFor
  \State Construct $\widehat{G}$ on $[n]$ such that $(i,j)\in \vec{E}(\widehat{G})$ if and only if $\mathcal S_{i,j}\in  A_{i,1}$.
  \State \textbf{Return}: $\widehat{G}$.
	\end{algorithmic} 
\end{algorithm}
\noindent It is clear that the running time of Algorithm~\ref{alg-learning} is $\widetilde{O}(n^4p^2)$ (computing each $\mathcal S_{i\to j}$ takes $\widetilde{O}((np)^2)$ time and there are $O(n^2)$ of them). We will show that when $C$ is large enough and $M\cdot \min\{T,n\}\ge Cn^2p^2\log n$, the output $\widehat{G}$ of Algorithm~\ref{alg-learning} satisfies $\widehat{G}=G^*$ with probability at least $0.92$. 

Before delving into the detailed analysis of Algorithm~\ref{alg-learning}, we briefly outline the proof ideas. Our goal is to show that with probability at least $0.92$, for each $i\in [n]$, 
\begin{equation}\label{eq-clustering}
\min_{j\in \operatorname{N}_i,j'\in \operatorname{N}_i^c}|\mathcal S_{i\to j}-\mathcal S_{i'\to j'}|>\max\left\{\max_{j,j'\in \operatorname{N}_i}|\mathcal S_{i\to j}-\mathcal S_{i\to j}|,\max_{j,j'\in \operatorname{N}_i^c}|\mathcal S_{i\to j}-\mathcal S_{i'\to j'}|\right\}\,.
\end{equation}
We decompose each $\mathcal S_{i\to j}$ as the sum of $\mathcal M_{i\to j}$ and $\mathcal A_{i\to j}$, where $\mathcal M_{i\to j}$ can be viewed as a martingale with respect to an appropriate filtration, and $\mathcal A_{i\to j}$ corresponds to the drift term. Using a standard martingale concentration inequality, we show that w.h.p. all $\mathcal M_{i\to j}$ concentrate tightly around $0$ (see Lemma~\ref{lem-martingale-concentration}), and then the problem reduces to showing that the set $\{\mathcal A_{i,j},j\in [n]\}$ exhibits a clustering property similar to \eqref{eq-clustering}. 

Using the duality to coalescing random walks, we can express each $\mathcal A_{i\to j}$ as
\[
\mathcal A_{i\to j}=MT_*p_{i\to j}+(1-p_{i\to j})\mathcal J_{i\to j}\,,
\]
where $p_{i\to j}:=\frac{1}{d_i}\sum_{u\in \operatorname{N}_i}p_{u,j}^{(t_*)},i,j\in [n]$ are deterministic quantities, and $\mathcal J_{i\to j},i,j\in [n]$ are random variables. We will then make the two following crucial observations:
\begin{itemize}
    \item For every $i\in [n]$, the set of deterministic quantities $\{p_{i\to j},j\in [n]\}$ exhibits the following desirable clustering property:
    \[
    \min_{j\in \operatorname{N}_i,j'\notin\operatorname{N}_i^c}|p_{i\to j}-p_{i\to j'}|>\max\left\{\max_{j,j'\in \operatorname{N}_i}|p_{i\to j}-p_{i\to j'}|,\max_{j,j\in \operatorname{N}_i^c}|p_{i\to j}-p_{i\to j'}|\right\}\,.
    \]
    \item There exists a random variable $\mathcal J$ independent of $i,j\in [n]$ such that $\mathcal J_{i\to j},i,j\in [n]$ are all very close to $\mathcal J$, and so
    \[
\mathcal A_{i\to j}\approx MT_*p_{i\to j}+(1-p_{i\to j})\mathcal J=(MT_*-\mathcal J)p_{i\to j}+\mathcal J\,.
\]
\end{itemize}
    Roughly speaking, the first of the two observations holds because, for $j\in \operatorname{N}_i$, the term $p_{j,j}^{(t_*)}=1$ appears in the sum $\sum_{u\in \operatorname{N}_i}p_{u,j}^{(t_*)}$, and this term is significantly larger than other summands $p_{u,j}^{(t_*)}$. The second item essentially arises from the fact that a random walk on $G^*$ mixes very well after $t_*$ steps (so the effect of the starting point gets diluted). 

Utilizing a quantitative version of the above observations (see \eqref{eq-R=J} and \eqref{eq-difference-is-large-in-differnet-group} below), the clustering property follows upon showing that $MT_*-\mathcal J$ satisfies an appropriate lower bound with probability at least $0.95$. Finally, the last task is carried out via a stochastic comparison enabled by the martingale-like properties of $\mathcal J$ (see Lemma~\ref{lem-effective-observations-lower-bound}). The detailed arguments are provided in the next subsection, with the proof of several technical lemmas deferred to the appendix.

\subsection{Analysis of the recovery algorithm}\label{subsec-algo-analysis}

Define a filtration $\{\mathcal F_{k}\}_{k=1}^{M\cdot T_*}$ as follows: for $k=rT_*+s$, $0\le r\le M-1,1\le s\le T_*$, let
\[
\mathcal F_{k}:=\sigma\left(\mathbf{x}_{m}^{(t)},1\le m\le r,0\le t\le T_m; \mathbf{x}_{r+1}^{(t)},0\le t \le t_{s-1}\right)\,.
\]
Define $\mathcal A_{m,i\rightarrow j}^{(t_s)}:=\mathbb{E}[x_{m,i}^{(t_s+1)}x_{m,j}^{(t_s)}\ \mid \mathcal F_{(m-1)T_*+s}]$ and $\mathcal M_{m,i\rightarrow j}^{(t_s)}:=x_{m,i}^{(t_s+1)}x_{m,j}^{(t_s)}-\mathcal A_{m,i\rightarrow j}^{(t_s)}$. Moreover, let $\prec$ denote the lexicographical order on the set $\mathcal I=\{(r,s):1\le r\le M,1\le s\le T_*\}$. For each pair $(r,s)\in \mathcal I$, let
\[
\mathcal{M}_{i\rightarrow j}^{(rT_*+s)}:=\sum_{(r',s')\preceq (r,s)}\mathcal M_{r',i\rightarrow j}^{(t_{s'})}\,,\quad
\mathcal A_{i\rightarrow j}^{(rT_*+s)}:=\sum_{(r',s')\preceq (r,s)}\mathcal A_{r',i\rightarrow j}^{(t_{s'})}\,.
\] For simplicity, we abbreviate $\mathcal M_{i\to j}^{(M\cdot T_*)}\triangleq \mathcal M_{i\to j}$ and $\mathcal A_{i\to j}^{(M\cdot T_*)}\triangleq \mathcal A_{i\to j}$.
This way, $S_{i\rightarrow j}=\mathcal{M}_{i\rightarrow j}+ \mathcal{A}_{i\rightarrow j}$ and the process $\{\mathcal M_{i\rightarrow j}^{(k)}\}_{k=1}^{M\cdot T_*}$ is a martingale with respect to $\{\mathcal F_k\}_{k=1}^{M\cdot T_*}$. 

We first handle the terms $\mathcal M_{i\to j}$ and show that, w.h.p., all of them are well-bounded. Recall that $c_0>0$ is the universal constant in assumption~\eqref{eq-p-assumption} (the specific choice of $c_0$ will become clear near the end of this section).
\begin{lemma}\label{lem-martingale-concentration}
	\label{prop-martingale-concentration}
For $C=C(\delta)$ large enough, it holds w.h.p. that 
	\begin{equation}\label{eq-martingale-concetnration}
		|\mathcal M_{i\to j}|<\frac{50c_0MT_*}{np}\,,\forall i,j\in [n]\,.
	\end{equation}
\end{lemma}
Lemma~\ref{lem-martingale-concentration} follows from a straightforward application of Azuma's inequality together with the union bound; the details can be found in Appendix~\ref{appendix-positive}. We now turn our attention to the drift term $\mathcal A_{i\to j}$, which poses the major difficulty. By the Markov property of the voter dynamics, we have for any $1\le m\le M,i,j\in [n]$ and $1\le s\le T_*$ that
\[
\mathcal A_{m,i\rightarrow j}^{(t_s)}=\mathbb{E}\left[x_{m,i}^{(t_s+1)}x_{m,j}^{(t_s)}\ \mid \mathbf{x}_{m}^{(t_{s-1})}\right]=\frac{1}{d_i}\sum_{u\in \operatorname{N}_i}\mathbb E\left[x_{m,u}^{(t_s)}x_{m,j}^{(t_s)}\ \mid \mathbf{x}_{m}^{(t_{s-1})}\right]\,.
\] 
Recall the definition of $p_{i,j}^{(t)}$ for $i,j\in [n]$ and $t\ge 0$ in \eqref{eq-p_ij^t}. Using the duality to coalescing random walks, the above can be rewritten as
\begin{equation}
	\begin{aligned}
		{\mathcal A}_{m,i\rightarrow j}^{(t_s)}= \frac{1}{d_i}\sum_{u\in\operatorname{N}_i}p_{u,j}^{(t_*)}+\mathcal R_{m,i\rightarrow j}^{(t_s)}\,,
	\end{aligned}
\end{equation}
where 
\[
{\mathcal R}_{m,i\rightarrow j}^{(t_s)}:=\frac{1}{d_i}\sum_{u\in\operatorname{N}_i}(1-p_{u,j}^{(t_0)})\mathbb{E}_{(p,q)\sim \mu_{u,j}}[x_{m,p}^{(t_{s-1})}x_{m,q}^{(t_{s-1})}]\,,
\]
and $\mu_{u,j}$ is the joint probability distribution of the time $t_*$ locations of two independent random walks starting at $u,j$ conditioned on the event that they do not meet before or on time $t_*$.

We use the next lemma to handle the complicated measure $\mu_{u,j}$. 

%We claim that for any $u\neq j$, the total-variational distance between $\mu_{u,j}$ and $\pi^{\otimes 2}$ is at most $\frac{1}{2np}$. \HD{The $o(\cdot)$ seems only hold for $p=o(1)$?}

\begin{lemma}\label{lem-TV-distance}
    Under assumption \eqref{eq-p-assumption}, we have that for any $u\neq j\in [n]$, $\operatorname{TV}(\mu_{n,j},\pi^{\otimes 2})\le \frac{50c_0}{np}\,.
    $
\end{lemma}

Lemma~\ref{lem-TV-distance} is a consequence of admissibility, and its proof is deferred to Appendix~\ref{appendix-positive}.
As a consequence of Lemma~\ref{lem-TV-distance}, if we define $\mathcal J_{m}^{(t_{s-1})}:=\mathbb{E}_{(p,q)\sim \pi^{\otimes2}}[x_{m,p}^{(t_{s-1})}x_{m,q}^{(t_{s-1})}]$ then 
\begin{equation}\label{eq-R=J}
{\mathcal R}_{m,i\rightarrow j}^{(t_s)}=\frac{1}{d_i}\sum_{u\in\operatorname{N}_i}(1-p_{u,j}^{(t_*)})\mathcal J_{m}^{(t_{s-1})}+\operatorname{err}_{m,i\to j}^{(t_{s-1})}\,,
\end{equation}
where $\operatorname{err}_{m,i\to j}^{(t_{s-1})}$ has absolute value no more than $\frac{50c_0}{np}$. 
%Here, crucially, $\mathcal J_m^{(t_{s-1})}$ is independent of $i$ and $j$.
Hence, 
\[
\mathcal A_{m,i\to j}^{(t_s)}=\frac{1-\mathcal J_m^{(t_s-1)}}{d_i}\sum_{u\in \operatorname{N}_i}p_{u,j}^{(t_*)}+\mathcal J_{m}^{(t_{s-1})}+\operatorname{err}_{m,i\to j}^{(t_{s-1})}\,.
\]

The key intuition here is that for any $i,j,j'$, if $j,j'$ both belong to $\operatorname{N}_i$ or both belong to $\operatorname{N}_i^c$, then $A_{m,i\to j}^{(t_{s-1})}$ and $A_{m,i\to j'}^{(t_{s-1})}$ are approximately equal. However, if $j\in \operatorname{N}_i$ while $j'\in \operatorname{N}_i^c$, then $\mathcal A_{m,i\to j}^{(t_{s-1})}$ is considerably larger than $\mathcal A_{m,i\to j'}^{(t_{s-1})}$ as there is the term $p_{j,j}^{(t_*)}=1$ that is significantly larger than other terms $p_{u,j}^{(t_*)}$ with $u\neq j$. To make this precise, we introduce the following lemma.

\begin{lemma}\label{lem-concentration-p}
    For any pairs $(u,v)$ and $(u',v')$ with $u\neq v$ and $u'\neq v'$, $|p_{u,v}^{(t_*)}-p_{u',v'}^{(t_*)}|= o\left(\frac{1}{np}\right)$. %\OS{we might need to modify the proof for large p}
\end{lemma}

Lemma~\ref{lem-concentration-p} is also a consequence of admissibility, and its proof can once again be found in Appendix~\ref{appendix-positive}.
By Lemma~\ref{lem-concentration-p}, for any $i,j,j'$ such that $j,j'\in \operatorname{N}_i$ or $j,j'\in \operatorname{N}_i^c$, we have
\begin{equation}\label{eq-difference-is-small-in-the-same-group}
    \left|\frac{1}{d_i}\sum_{u\in \operatorname{N}_i}p_{u,j}^{(t_*)}-\frac{1}{d_i}\sum_{u\in \operatorname{N}_i}p_{u,j'}^{(t_*)}\right|=o\left(\frac 1{np}\right)\,.
\end{equation}
Moreover, since $\max_{u\neq v}p_{u,v}^{(t_*)}=o(1)$ (this can be deduced upon inspecting the proof of Lemma~\ref{lem-TV-distance}, but can also be proven in a straightforward manner), for any $i,j,j'$ such that $j\in \operatorname{N}_i$ and $j'\in \operatorname{N}_i^c$,
\begin{equation}\label{eq-difference-is-large-in-differnet-group}
\frac{1}{d_i}\sum_{u\in \operatorname{N}_i}p_{u,j}^{(t_*)}\ge \frac{1}{2np}+\frac{1}{d_i}\sum_{u\in \operatorname{N}_i}p_{u,j'}^{(t_*)}\,.
\end{equation}

In light of \eqref{eq-difference-is-small-in-the-same-group} and \eqref{eq-difference-is-large-in-differnet-group}, for any $i,j,j'$ such that $j,j'\in \operatorname{N}_i$ or $j,j'\in \operatorname{N}_i^c$, and any $1\le m\le M$ and $0\le s\le T_*$, 
\[
|\mathcal A_{m,i\to j}^{{(t_s)}}-\mathcal A_{m,i\to j'}^{(t_s)}|\le o\left(\frac 1{np}\right)+\operatorname{err}_{m,i\to j}^{(t_{s-1})}+\operatorname{err}_{m,i\to j'}^{(t_{s-1})}\le \frac{200c_0}{np}\,.
\]
Summing over $m$ and $s$, we get
\begin{equation}\label{eq-200c_0}
|\mathcal A_{i,j}-\mathcal A_{i,j'}|\le \frac{200c_0MT_*}{np}\,.
\end{equation}
Similarly, for any $i$ and any $j\in \operatorname{N}_i,j'\in \operatorname{N}_i^c$, 
\begin{equation}\label{eq-gap}
    \mathcal A_{i,j}-\mathcal A_{i,j'}\ge \frac{1}{2np}\sum_{m=1}^M\sum_{s=1}^{T_*}(1-\mathcal J_{m}^{(t_{s-1})})-\frac{100c_0MT_*}{np}\,.
\end{equation}

We now proceed to lower-bound the sum in \eqref{eq-gap}. Recall that
\[
\mathcal J_m^{(t_{s-1})}=\mathbb{E}_{(p,q)\sim \pi^{\otimes 2}}[x_{m,p}^{(t_{s-1})}x_{m,q}^{(t_{s-1})}]=\left(\mathbb{E}_{p\sim \pi}[x_{m,p}^{(t_{s-1})}]\right)^2\,.
\]
We write $W_m^{(t)}=\mathbb{E}_{p\sim \pi}[x_{m,p}^{(t)}]=\sum_{i\in[n]}\pi(i)x_{m,p}^{(t)}$ and then define $\widetilde{T}_m$ as follows: If there exists $t\le \min\{T,n\}$ such that $|W_m^{(t)}|>\frac 12$, let $\widetilde{T}_m$ be the minimal such $t$; otherwise, set $\widetilde {T}_m=\min\{T,n\}$. The following lemma provides a lower bound on the sum of these times $\widetilde T_m$.
 
\begin{lemma}\label{lem-effective-observations-lower-bound}
    There exists a universal constant $\zeta>0$ such that, for any $M\ge 1$, it holds with probability at least $0.95$ that 
    \begin{equation}\label{eq-effective-observations-lower-bound}
    \sum_{m=1}^M \widetilde{T}_m\ge \zeta\cdot M\cdot \min\{T,n\}\,.
    \end{equation}
\end{lemma}
Intuitively, Lemma~\ref{lem-effective-observations-lower-bound} follows from the convergence result of \cite{CCC16}, which suggests that for each $1\le m\le M$, $\widetilde{T}_m$ is typically of order $\Theta(\min\{T,n\})$. We provide a self-contained proof in Appendix~\ref{appendix-positive}.
In the following, we assume that Lemma~\ref{lem-effective-observations-lower-bound} holds and complete the analysis of Algorithm~\ref{alg-learning}. Given \eqref{eq-effective-observations-lower-bound}, we have
\begin{align*}
\sum_{m=1}^{M}\sum_{s=0}^{T_*}(1-\mathcal J_{m}^{(t_{s-1})})\ge&\  \frac34\#\left\{1\le m\le M,0\le s\le T_*-1:|W_m^{(t_s)}|\le \frac12\right\}\\
\ge&\ \frac{3}{4}\sum_{m=1}^M{\frac{\widetilde{T}_m}{t_*}}\ge \frac{3\zeta M\cdot\min\{T,n\}}{4t_*}\ge \frac{3\zeta MT_*}{4}\,.
\end{align*}

We pick the universal constant $c_0$ such that $c_0\le \zeta/1500$. This way, with probability at least $0.95$, for any $i\in [n]$, $j\in \operatorname{N}_i$, and $j'\in \operatorname{N}_i^c$, it will be the case that 
\begin{equation}\label{eq-3.11}
\mathcal A_{i\to j}-\mathcal A_{i\to j}\ge \frac{3\zeta MT_*}{8np}-\frac{100c_0MT_*}{np}>\frac{400c_0MT_*}{np}\,. 
\end{equation}
Recall that by Lemma~\ref{lem-martingale-concentration}, \eqref{eq-martingale-concetnration} holds w.h.p.. Combining \eqref{eq-martingale-concetnration} with \eqref{eq-200c_0} and \eqref{eq-3.11}, we conclude that, with probability at least $0.95-o(1)\ge 0.92$, for any $i,j,j'$, if $j\in \operatorname{N}_i$ and $j\in\operatorname{N}^c$, then
\[
\mathcal S_{i\to j}-\mathcal S_{i\to j'}>\frac{400c_0MT_*}{np}-\frac{50c_0MT_*}{np}-\frac{50c_0MT_*}{np}=\frac{300c_0MT_*}{np}\,,
\]
and if $j,j'\in \operatorname{N}_i$ or $j,j'\in \operatorname{N}_i^c$, then
\[
|\mathcal S_{i\to j}-\mathcal S_{i\to j'}|<\frac{200c_0MT_*}{np}+\frac{50c_0MT_*}{np}+\frac{50c_0MT_*}{np}=\frac{300c_0MT_*}{np}\,.
\]
This shows that, with probability at least $0.92$, the clustering property \eqref{eq-clustering} holds for every $i\in [n]$, thereby proving Theorem~\ref{thm-main}.

\subsection{Experimental results on synthetic and real data}\label{subsec-algo-exp}

To evaluate the effectiveness of the classifiers $\mathcal S_{i\to j}^\ddagger,i,j\in [n]$ as proposed in \eqref{eq-S^dagger}, we test the performance of Algorithm~\ref{alg-learning} with $\mathcal S_{i\to j}$ replaced by $\mathcal S^\ddagger_{i\to j}$ (i.e., $t_*=1$). 

We first evaluated the algorithm on randomly generated directed graphs on $[n]$ with $n$ ranging from $50$ to $3000$. Each vertex’s out-degree was sampled uniformly from $\{3,4,5\}$, and its out-neighborhoods were selected independently at random. We simulated the voter model trajectories up to the consensus time $T_{\text{cons}}$, beyond which no further information is gained. Using a single trajectory ($M=1$), we tested whether the algorithm could recover the graph structure. Given the graphs’ sparsity, accuracy may be misleading (e.g., an empty-graph prediction would result in a high accuracy). To address this, we use the $F_1$ score—the harmonic mean of precision and recall—as the metric for performance evaluation. Each experiment was repeated 100 times for statistical robustness. Results are summarized in Table~\ref{table1}.

\begin{comment}
% accuracy result
\centering
\resizebox{\textwidth}{!}{
\begin{tabular}{c|ccccccccc }
$n$                                                             & 50    & 100    & 200    & 300    & 400    & 500    & 1000    & 2000    & 3000    \\ \hline
$|\vec{E}|$                                                             & 194   & 392    & 798    & 1204   & 1598   & 1991   & 3998    & 7996    & 11928   \\
\begin{tabular}[c]{@{}c@{}} Consensus\\ time\end{tabular}      & 73.42 & 124.09 & 282.62 & 373.53 & 515.68 & 663.09 & 1378.95 & 2880.21 & 4318.84 \\
\begin{tabular}[c]{@{}c@{}}Mean\\ accuracy (\%)\end{tabular}    & 44.45 & 52.60  & 71.44  & 79.24  & 84.06  & 89.41  & 96.17   & 99.46   & 99.82   \\
\begin{tabular}[c]{@{}c@{}}Median \\ accuracy (\%)\end{tabular} & 47.60 & 55.83  & 78.92  & 86.57  & 91.76  & 94.30  & 99.35   & 100     & 100   
\end{tabular}
}
\end{comment}
\begin{table}[ht]
\centering
\resizebox{\textwidth}{!}{
\begin{tabular}{c|ccccccccc }
$n$                                                             & 50    & 100    & 200    & 300    & 400    & 500    & 1000    & 2000    & 3000    \\ \hline
$|\vec{E}|$                                                             & 194   & 392    & 798    & 1204   & 1598   & 1991   & 3998    & 7996    & 11928   \\
Consensus time    & 73.42 & 124.09 & 282.62 & 373.53 & 515.68 & 663.09 & 1378.95 & 2880.21 & 4318.84 \\
 Mean  $F_1$ score    & 0.2155 & 0.1599 & 0.2023 & 0.2255 & 0.2964 & 0.3468 & 0.5676 & 0.7806 & 0.8799 \\
Median $F_1$ score  & 0.1857 & 0.1077 & 0.1113 & 0.1155 & 0.1438 & 0.1924 & 0.6285 & 0.9955 & 0.9999 \\
Baseline $F_1$ score    & 0.1440 & 0.0754 & 0.0391 & 0.0264 & 0.0198 & 0.0158 & 0.0080 & 0.0040 & 0.0026
\end{tabular}
}
\caption{Performance of Algorithm~\ref{alg-learning} on randomly generated directed graphs. We simulated a single trajectory of the voter model ($M=1$) until consensus. The observed consensus time scales linearly with the number of vertices $n$. As $n$ increases, the algorithm’s performance improves, achieving near-perfect recovery for large graphs despite having access to only one realization. The baseline $F_1$ score corresponds to that of an empty-graph prediction. Consensus times are rounded to two decimal places, and $F_1$ scores to four.}
\label{table1}
\end{table}
% As shown in Table~\ref{table1}, both the mean and median $F_1$ scores tend to increase with $n$, achieving near-perfect reconstruction for large graphs. 
In Table~\ref{table1}, we also report the baseline $F_1$ score, corresponding to the $F_1$ score of an empty-graph prediction, given by $\frac{2r}{r+1}$, where $r = \frac{|\vec{E}|}{n^2}$ represents the fraction of 1's in the ground-truth adjacency matrix. Our algorithm consistently outperforms this baseline for all $n$. While the baseline $F_1$ score decreases as $n$ grows due to increasing sparsity, our algorithm's $F_1$ score generally improves, with the median reaching nearly 1 for $n \geq 2000$, indicating near-perfect reconstruction. Additionally, the consensus time scales approximately linearly with $n$.

Next, we fixed a graph with $n=3000$ vertices that was generated by the same process as above and examined how the performance of Algorithm~\ref{alg-learning} varies for different values of $M$ and $T$. As in the previous experiment, the graph was constructed with each vertex having an out-degree between $3$ and $5$,  resulting in a total of $|\vec{E}|=12,005$ edges. We varied $M$ from $1$ to $300$ and $T$ from $10$ to $4000$, testing each $(M,T)$ pair on $100$ independent voter model trajectories using our algorithm.
%\SH{I want to reduce this gap due to the table. One way is to split the next paragraph into two so that some part comes here.}

\begin{comment}
% accuracy result
\centering
\begin{tabular}{c|cccccccc}
$(M,T)$ & 10    & 30      & 50      & 100     & 200     & 500     & 1000    & 2000    \\ \hline
1       & 0.36  & 11.26   & 23.87   & 49.77   & 78.65   & 97.35   & 99.79   & 100     \\
3       & 10.22 & 44.39   & 67.36   & 90.72   & 98.67   & 100     & $-$ & $-$ \\
5       & 21.56 & 67.06   & 86.00   & 97.53   & 99.89   & 100     & $-$ & $-$ \\
10      & 45.97 & 90.51   & 97.62   & 99.89   & 100     & $-$ & $-$ & $-$ \\
50      & 97.00 & 100     & $-$     & $-$ & $-$ & $-$ & $-$ & $-$ \\
100     & 99.82 & 100     & $-$ & $-$ & $-$ & $-$ & $-$ & $-$ \\
200     & 100   & $-$ & $-$ & $-$ & $-$ & $-$ & $-$ & $-$
\end{tabular}
\end{comment}

\begin{table}[H]
\centering
\resizebox{\textwidth}{!}{
\begin{tabular}{c|ccccccccccc}
$(M,T)$ & 10    & 30      & 50      & 100     & 200     & 500     & 1000    & 2000 & 3000 & 4000    \\ \hline
1       & 0.0027 & 0.0027   & 0.0028   & 0.0034   & 0.0070   & 0.0657   & 0.5410   & 0.9966     & 0.9990 & 1 \\
3       & 0.0027 & 0.0032   & 0.0047   & 0.0163   & 0.1289   & 0.9932    & 1 & $-$  & $-$ & $-$\\
5       & 0.0028 & 0.0047   & 0.0105   & 0.0735   & 0.6390   & 0.9999     & 1 & $-$  & $-$ & $-$\\
10      & 0.0033 & 0.0158   & 0.0734   & 0.6946   & 0.9994     & 1 & $-$ & $-$  & $-$ & $-$\\
50      & 0.0569 & 0.9947     & 0.9999     & 1 & $-$ & $-$ & $-$ & $-$  & $-$ & $-$\\
100     & 0.5771 & 1    & $-$ & $-$ & $-$ & $-$ & $-$ & $-$  & $-$ & $-$\\
200     & 0.9991   & 1 & $-$ & $-$ & $-$ & $-$ & $-$ & $-$  & $-$ & $-$\\
300     & 1   & $-$ & $-$ & $-$ & $-$ & $-$ & $-$ & $-$ & $-$ & $-$ \\
\end{tabular}}
\caption{Median $F_1$ score of Algorithm~\ref{alg-learning} across 100 trials for different $(M, T)$ pairs on a randomly generated directed graph with $n=3000$ vertices and $|\vec{E}|=12,005$ edges. The baseline $F_1$ score for this graph is 0.00266, and our algorithm outperforms the naive empty-graph prediction even for $(M, T) = (1,10)$. For fixed $M$, once the median $F_1$ score reaches 1, further increasing $T$ adds only redundant information, so those values are omitted. All numerical values are rounded to four decimal places.
}
\label{table2}
\end{table}

Table~\ref{table2} presents the median $F_1$ score over 100 trials for various $(M,T)$ pairs. Across all cases, our algorithm outperformed the naive empty-graph prediction (baseline $F_1$ score: 0.00266). The results confirm that accuracy largely depends on the product $M \cdot T$, as predicted by Theorems~\ref{thm-main} and \ref{thm-negative}. For instance, the median $F_1$ score falls within $0.055$–$0.075$ for $(M,T) = (1,500), (5,100), (10,50)$, and $(50,10)$, indicating similar performance when $M \cdot T$ is fixed. When $M \cdot T$ exceeds 3000, the algorithm achieves perfect recovery in all tests. Intriguingly, the $F_1$ score does not increase gradually with $M \cdot T$; it remains below 0.1 for $M \cdot T \leq 500$, exceeds 0.5 for $M \cdot T \geq 1000$, and surpasses 0.99 for $M \cdot T \geq 1500$. %This reflects the information gap in Theorems~\ref{thm-main} and \ref{thm-negative}. 
Once the median $F_1$ score reaches 1 for a given $M$, further increasing $T$ provides no additional benefit, so these cases are omitted from the table for clarity. These results highlight the robustness and scalability of our approach, demonstrating its potential for real-world applications in causal inference and network reconstruction.

To further assess the practical effectiveness of our algorithm, we applied it to a real-world dataset: the Twitter Interaction Network for the 117th US Congress (\cite{fink2023centrality}). This dataset, collected via Twitter's API, represents directed interactions among members of the US House and Senate, forming a network with $n=475$ vertices and $|\vec{E}|=25,571$ edges. We varied $M$ from 1 to 1000 and simulated voter dynamics until consensus (since the graph is connected, consensus is always reached).

%and we observed finite consensus times.

\begin{comment}
\begin{table}[h] 
\begin{tabular}{c|cccccccc}
$M$     & 1      & 5      & 10     & 50     & 100    & 300    & 500    & 1000   \\ \hline
\begin{tabular}[c]{@{}c@{}} Consensus\\ time\end{tabular}   & 318.19 & 346.23 & 383.92 & 379.46 & 370.51 & 369.86 & 365.82 & 366.14 \\
\begin{tabular}[c]{@{}c@{}}Mean\\ accuracy (\%)\end{tabular} & 39.12  & 59.25  & 67.09  & 80.51  & 86.73  & 94.90  & 97.14  & 98.59 
\end{tabular}
\caption{Performance of Algorithm~\ref{alg-learning} on the data from \cite{fink2023centrality}. The network consists of $n=475$ vertices and $|\vec{E}|=25,571$ directed edges, representing interactions among members of Congress. The consensus time of the graph was around 350. The algorithm achieves high accuracy even for relatively small values of $M$. All numerical values are rounded to two decimal places.}
\label{table3}
\end{table}
\end{comment}

\begin{table}[h] 
\resizebox{\textwidth}{!}{
\begin{tabular}{c|cccccccccc}
$M$              & 1      & 5      & 10     & 50     & 100    & 200    & 300    & 400    & 500    & 1000   \\ \hline
Mean consensus time   & 391.55 & 360.89 & 366.81 & 371.43 & 364.68 & 366.33 & 366.39 & 369.19 & 367.90 & 367.21 \\
Mean $F_1$ score & 0.1142 & 0.1297 & 0.1442 & 0.2639 & 0.3905 & 0.5872 & 0.7146 & 0.7748 & 0.8185 & 0.8957
\end{tabular}}
\caption{Performance of Algorithm~\ref{alg-learning} on the Twitter Interaction Network of the 117th US Congress (\cite{fink2023centrality}), which consists of $n=475$ vertices and $|\vec{E}|=25,571$ directed edges. The table reports the mean consensus time and mean $F_1$ score as $M$ varies. The baseline $F_1$ score for this network is 0.1112, and our algorithm outperforms it even at $M=1$. As $M$ increases, performance improves consistently, with the $F_1$ score exceeding 0.7 at $M=300$. The consensus time remains around 370 across different $M$ values. Mean consensus times are rounded to two decimal places, and $F_1$ scores to four.
}
\label{table3}
\end{table}

As shown in Table~\ref{table3}, the mean $F_1$ score increases with $M$, mirroring the trend observed for synthetic data. Even at $M=1$, the algorithm surpasses the baseline $F_1$ score of 0.1112, indicating that it effectively captures meaningful network structure. For moderate values like $M=300$, the $F_1$ score exceeds 0.7, demonstrating strong alignment with the true network.

These results show that our method generalizes well beyond synthetic graphs, effectively recovering relationships in real-world networks and reinforcing its applicability to broader social network analysis tasks.

\section{Proof of Theorem~\ref{thm-negative}}\label{sec-negative}
We now turn to prove Theorem~\ref{thm-negative}. We first reduce the result to a claim regarding the maximal-likelihood estimator (MLE) via a standard argument. Write $\mathbf{X}=\big\{(\mathbf{x}^{(0)}_m,\cdots,\mathbf{{x}}_m^{(T)}),1\le m\le M\big\}$ for simplicity. For any estimator $\widetilde{G}=\widetilde{G}(\mathbf{X})$, 
our goal is to show that
\[
\mathbb{E}_{G^*\sim \mathcal G(n,p)}\mathbb{P}_{\mathbf{X}\mid G^*}[\widetilde{G}(\mathbf X)=G^*]\le 0.1\,.
\]
Denoting  by $\mu(\mathbf{X})$ the posterior measure on $G^*$ given $\mathbf{X}$, we can rewrite the above expression as
\[
\mathbb{E}_{\mathbf X}\mathbb{P}_{\operatorname{G\sim \mu(\mathbf{X})}}[G=\widetilde{G}(\mathbf{X})]\le \mathbb{E}_{\mathbf{X}}\mathbb{P}_{G\sim \mu(\mathbf X)}[G=\operatorname{MLE}(\mathbf{X})]=\mathbb{E}_{G^*\sim \mathcal G(n,p)}\mathbb{P}_{\mathbf{X}\mid G^*}[\operatorname{MLE}(\mathbf{X})=G^*]\,.
\]
Here, $\operatorname{MLE}(\mathbf{X})$ denotes the maximal-likelihood estimator, and the inequality holds by definition. Since $G^*\sim \mathcal G(n,p)$ is admissible w.h.p., it suffices to show that for any admissible graph $G^*$, $\mathbb{P}_{\mathbf{X}\mid G^*}[\operatorname{MLE}(\mathbf{X})=G^*]\le 0.08$, provided that condition \eqref{eq-condition-negative} holds with $c=c(\delta)$ small enough.

Recall that for a graph $G$ on $[n]$, we use $\operatorname{N}_i(G)$ to denote the out-neighborhood of vertex $i \in [n]$. We also write $\operatorname{N}_i=\operatorname{N}_i(G^*)$ and $d_i=|\operatorname{N}_i|$ for simplicity. 
%\SH{We have already introduced this notation at the beginning of section 1.2. Should we delete here? or maintain for the readers?}\HD{Good catch, let's recall the notations here.} 
Additionally, for $1\le m\le M$, $0\le t\le T$, and $i\in [n]$, we let $x_{m,i}^{(t)}\in \{\pm1\}$ be the opinion of vertex $i$ at time $t$ in the $m$-th trajectory. For any subset $\operatorname{N}$ of $[n]$, we denote
\[
S_{m}^{(t)}(\operatorname{N}):=\sum_{j\in \operatorname{N}}x_{m,j}^{(t)}\,.
\]
Given the observations $\mathbf{X}$, we define $T_m$ as follows: If there exists $0\le t\le T$ such that $x_{i,m}^{(t)}=x_{j,m}^{(t)}\ $ for  all $ i,j\in [n]$, then let $T_m$ be the minimal such $t$; otherwise, set $T_m=T$. For any $G$, the likelihood of $\mathbf X$ given that $G^*=G$, is
\begin{align*}
  %\Pb[(\bx{m}{0},\bx{m}{1},\cdots,\bx{m}{T}), 1\le m\le M\mid G^*=G]=
  &\ p^{\sum_{i\in [n]}d_i(G)}(1-p)^{\sum_{i\in [n]}(n-1-d_i(G))}\prod_{m=1}^M\prod_{t=0}^{T_m-1}\prod_{i\in [n]}\left(\frac{1}{2}+\frac{x_{m,i}^{(t+1)}\cdot S_m^{(t)}(\operatorname{N}_i(G))}{2d_i(G)}\right)\\
  \propto&\ \left(\frac{p}{1-p}\right)^{\sum_{i\in [n]}d_i(G)}\prod_{m=1}^M\prod_{t=0}^{T_m-1}\prod_{i\in [n]}\left(1+\frac{x_{m,i}^{(t+1)}\cdot S_m^{(t)}(\operatorname{N}_i(G))}{d_i(G)}\right)\,.
\end{align*}
Therefore, the MLE is given by the maximizer of the score function about $G$, defined as
\[
\mathcal L(G):=\log\left(\frac{p}{1-p}\right)\sum_{i\in[n]}d_i(G)+\sum_{m=1}^M\sum_{t=0}^{T_m-1}\sum_{i\in [n]}\log\left(1+\frac{x_{m,i}^{(t+1)}\cdot S_m^{(t)}(\operatorname{N}_i(G))}{d_i(G)}\right)\,.
\]

Now, consider a triple $(a,b,c)$ with three distinct vertices $a,b,c\in [n]$ such that $(a,b)\in \vec{E}(G^*)$ while $(a,c)\notin \vec{E}(G^*)$. Define the \emph{flipping graph} $G_{abc}$ as the graph obtained from $G^*$ by removing the directed edge $(a,b)$ and adding a directed edge $(a,c)$. We will argue that w.h.p. there exists a flipping graph $G_{abc}$ such that $\mathcal L(G_{abc})>\mathcal L(G^*)$, thus demonstrating the failure of the MLE. A key observation is that this modification only alters the out-neighborhood of $a$, leaving all other vertex neighborhoods unchanged. In contrast, performing a similar local modification to an undirected graph would affect at least two vertex neighborhoods, which could introduce additional complications to the analysis.

For any flipping graph $G_{abc}$, it holds $d_i(G_{abc})=d_i,\forall i\in [n]$. 
%Thus, by definition, 
%\[
%\mathcal L(G_{abc})=\log\left(\frac{p}{1-p}\right)\sum_{i\in [n]}d_i+\sum_{m=1}^M\sum_{t=0}^{T_m-1}\sum_{i\in [n]}\log\left(1+\frac{x_{m,i}^{(t+1)}\cdot S_m^{(t)}(\operatorname{N}_i(G_{abc}))}{d_i}\right)\,.
%\] \SH{I think we can skip this eq and directly go to the difference}\HD{We can do this in the AOS version}
Moreover, we have $\operatorname{N}_i(G_{abc})=\operatorname{N}_i(G^*)$ for all $i\neq a$ and  $\operatorname{N}_a(G_{abc})=(\operatorname{N}_a\setminus\{b\})\cup \{c\}$. Thus, $\mathcal L(G_{abc})-\mathcal L(G^*)$ equals
\begin{align*}
    %& \sum_{m=1}^M\sum_{t=0}^{T-1}\left[\sum_{i\in [n]}\log\left(1+\frac{x_{m,i}^{(t+1)}\cdot S_m^{(t)}(\operatorname{N}_i(G_{abc}))}{d_i}\right)-\sum_{i\in [n]}\log\left(1+\frac{x_{m,i}^{(t+1)}\cdot S_m^{(t)}(\operatorname{N}_i(G^*))}{d_i}\right)\right]\\
    &\sum_{m=1}^M\sum_{t=0}^{T_m-1}\Bigg[\log\left(1+\frac{x_{m,a}^{(t+1)}\cdot S_m^{(t)}(\operatorname{N}_a(G^*))}{d_a}\right)-\log\left(1+\frac{x_{m,a}^{(t+1)}\cdot S_m^{(t)}(\operatorname{N}_a(G_{abc}))}{d_a}\right)\Bigg]\\
    =&%\sum_{m=1}^M\sum_{t=0}^{T_m-1}\log\left(\frac{d_a+x_{m,a}^{(t+1)}\cdot S_m^{(t)}(\operatorname{N}_a(G_{abc}))}{d_a+x_{m,a}^{(t+1)}\cdot S_m^{(t)}(\operatorname{N}_a(G^*))}  \right) 
    \sum_{m=1}^M\sum_{t=0}^{T_m-1}\log\left(1+\frac{x_{m,a}^{(t+1)} (x_{m,b}^{(t)}-x_{m,c}^{(t)})}{d_a+x_{m,a}^{(t+1)}\cdot S_m^{(t)}(\operatorname{N}_a(G^*))}  \right)   \,.
\end{align*}

In what follows we further abbreviate $S_m^{(t)}(\operatorname{N}_a(G^*))$ as $S_{m,a}^{(t)}$, and we write
\begin{equation}\label{eq-def-X}
X_{m,a,b,c}^{(t)}:=\frac{x_{m,a}^{(t+1)} (x_{m,b}^{(t)}-x_{m,c}^{(t)})}{d_a+x_{m,a}^{(t+1)}\cdot S_{m,a}^{(t)}}\,.
\end{equation}
OS{We say that a triple $(a,b,c)$ is \emph{bad} if there exist $1\le m\le M$ and $1\le t\le T$ such that $X_{m,a,b,c}^{(t)}=-1$; otherwise, we say the tripe is }\emph{good}\footnote{A bad triple corresponds to a flipping graph that does not lie in the support of the posterior measure of $G^*$ given $\mathbf X$.}. Note that $X_{m,a,b,c}^{(t)}=-1$ occurs only when $x_{m,a}^{(t+1)}S_{m,a}^{(t)}=-d_a+2$ and $x_{m,a}^{(t+1)} (x_{m,b}^{(t)}-x_{m,c}^{(t)})=-2$. Thus, if a tripe $(a,b,c)$ is good, it holds that
\[
-\frac 12\le X_{m,a,b,c}^{(t)}\le 1\,,\forall 1\le m\le M,1\le t\le T_m-1\,.
\]

Let $\gamma:=\min\{\sqrt{\zeta\delta/5000},\zeta/1000,\delta/10\}$ (where $\zeta$  is the universal constant in Lemma~\ref{lem-effective-observations-lower-bound}). We randomly select $n^\gamma$ triples $\{(a_i,b_i,c_i),1\le i\le n^\gamma\}$ in the following way. For each $1\le i\le n^\gamma$, set $a_i=i$ and then, uniformly at random, choose $b_i$ and $c_i$ such that $(a_i,b_i)\in \vec{E}(G^*)$ and $(a_i,c_i)\notin \vec{E}(G^*)$. Here, the choices of $b_i,c_i,1\le i\le n^\gamma$ are independent.

Recall the constant $C^*=C^*(\delta)$ from Property-(iii) of admissibility. Define 
\begin{equation}\label{eq-C0}
C_0=C_0(\delta):=16\max\{100C^*,1\}+10^6\,.
\end{equation}
We assume $M\cdot \min\{T,n\}=cn^2p^2\log n$ for 
\begin{equation}\label{eq-choice-c}
c=c(\delta)<\gamma^2/(100C_0^2).
\end{equation}
We claim that, in order to prove Theorem~\ref{thm-negative}, it suffices to show the following proposition.
\begin{proposition}\label{eq-prop-three-properties}
With probability at least $0.92$, the following three items hold (here, the randomness comes from the voter model trajectories as well as the choices of the triples as above): % $(a_i,b_i,c_i),1\le i\le n^\gamma$):
\\
\noindent (i) For each $1\le i\le n^\gamma$, $a_i,b_i,c_i$ are distinct and each triple $(a_i,b_i,c_i)$ is good.\\
\noindent (ii) For each $1\le i\le n^\gamma$, we have that
\begin{equation}\label{eq-quadratic-term}
\sum_{m=1}^M\sum_{t=0}^{T_m-1} \big(X_{m,a_i,b_i,c_i}^{(t)}\big)^2\le C_0c\log n\,.
\end{equation}
\\
\noindent (iii) There exists $1\le i^*\le n^\gamma$ such that
\begin{equation}
\sum_{m=1}^M\sum_{t=0}^{T_m-1} X_{m,a_i,b_i,c_i}^{(t)} \ge \gamma\sqrt{c}\log n\,.
\end{equation}
\end{proposition}

Assuming this proposition, we can prove Theorem~\ref{thm-negative} as follows.
\begin{proof}[Proof of Theorem~\ref{thm-negative} assuming Proposition~\ref{eq-prop-three-properties}]
    It suffices to show that with probability at least $0.92$, there exists $i$ such that $\mathcal L(G_{a_ib_ic_i})>\mathcal L(G^*)$. By Proposition~\ref{eq-prop-three-properties}, we have with probability at least $0.92$ that Items (i)-(iii) all hold. Assume this is the case and fix $1\le i^*\le n^\gamma$ as in Item (iii). We claim that $\mathcal L(G_{a_{i^*}b_{i^*}c_{i^*}})>\mathcal L(G^*)$. Since $(a_{i^*},b_{i^*},c_{i^*})$ is good, we have
    \[
    -\frac 12\le X_{m,a_{i^*},b_{i^*},c_{i^*}}^{(t)} \le 1\,,\forall 1\le m\le M,0\le t\le T_m-1\,.
    \]
    Therefore, using the fact that $\log(1+x)\ge x-10x^2,\forall -1/2\le x\le 1$, we conclude that
    \begin{align*}
    \mathcal L(G_{a_{i^*}b_{i^*}c_{i^*}})-\mathcal L(G^*)=&\ \sum_{m=1}^M\sum_{t=0}^{T_m-1}\log\left(1+X_{m,a_{i^*},b_{i^*},c_{i^*}}^{(t)}\right)\\
    \ge&\ \sum_{m=1}^M\sum_{t=0 }^{T_m-1}X_{m,a_{i^*},b_{i^*},c_{i^*}}^{(t)}-10\sum_{m=1}^M\sum_{t=0}^{T_m-1}\big(X_{m,a_{i^*},b_{i^*},c_{i^*}}^{(t)}\big)^2\\
    \ge&\ \gamma \sqrt c\log n-10C_0c\log n\,,
    \end{align*}
    which is positive due to our choice of $c$. This finishes the proof.
\end{proof}

The remainder of this section is devoted to proving Proposition~\ref{eq-prop-three-properties}. In Section~\ref{subsec-technical-imputs}, we first establish several important properties of the voter model dynamics. Then, building on these results, Section~\ref{subsec-proof-of-three-properties} addresses all three items in Proposition~\ref{eq-prop-three-properties} using tools varying from stochastic domination to the conditional first- and second-moment methods. To maintain a smooth flow of presentation, we defer some technical proofs to the appendix.

\subsection{Key lemmas regarding voter model dynamics}\label{subsec-technical-imputs}

In what follows, we fix an admissible graph $G^*$. %\OS{The proof of $\mathcal G_{\operatorname{RW}}$ requires the randomness of $G^*$ again, admissibility is not quite enough. We could add an extra property to the definition of admissibility, or just rewrite this paragraph}\HD{If that property is not too weird, I prefer to add it into admissibility} 
We use $\mathbb{P}$ to denote the probability measure governing the randomness over the voter model dynamics. We say that an event $\mathcal G$ occurs ``with overwhelming probability'' (w.o.p.) if $\mathbb P[\mathcal G]\ge 1-n^{-C}$ for any constant $C$ provided that $n$ is large enough. %When we focus on a fixed $m$, we will drop the subscript $m$ to lighten the notations.

Below, we present three key results that contribute to the understanding of the voter model dynamics on admissible graphs. Since the proofs of these statements are mostly disconnected from the main arguments in the rest of the paper, we defer them to Appendix~\ref{appendix-proof-of-key-lemmas} and include only a brief sketch of each here. We begin with a lemma which tells us that $S_m^{(t)}=\sum_{i\in [n]}x_{i,m}^{(t)}$ does not vary too drastically over $t$. 

\begin{definition}[Moderate increment]
    Define $\mathcal G_{\operatorname{MI}}$ as the event that for any $1\le m\le M$, $|S_m^{(0)}|\le n/\log n$,
    and for any $0\le t\le T_m-1$, 
    \begin{equation}
        \label{eq-moderate-increment}
            |S_m^{(t+1)}-S_m^{(t)}|\le 2(\log n)^2(n-|S_m^{(t)}|)^{1/2}\,.
    \end{equation}
\end{definition}
%\HD{This one looks good}

%\begin{definition}[Moderate increment]  
 %   Let $\xi_n=(\log n)^{-1/4}$. Define $\mathcal G_{\operatorname{MI}}$ as the event that for any $1\le m\le M$, $|S_m^{(0)}|\le \xi_n\cdot n$, 
 %   and for any $0\le t\le T_m-1$, 
 %   \begin{equation}
  %      \label{eq-moderate-increment}
   %         |S_m^{(t+1)}-S_m^{(t)}|\le \xi_n\cdot \max\{{n-|S_m^{(t)}|},2(\log n)^2\}\,.
   % \end{equation}
    %\OS{Maybe we should add that when $|S_m|\le n/2$ then $|S_m^{(t+1)}-S_m^{(t)}|\le \widetilde O(\sqrt{n})$. This follows immediately from the same proof.}\HD{Can we just replace $(n-|S_m^{(t)}|)$ by $(\log n)^2(n-|S_m^{(t)}|)^{1/2}$? Would it convenient to write? I don't like the $\xi_n$ notation here actually...}\OS{Yeah let's do that, I'll keep both versions for now}
%\end{definition}
\begin{lemma}\label{lem-MI}
    $\mathcal G_{\operatorname{MI}}$ happens w.o.p..
\end{lemma}

Roughly speaking, Lemma~\ref{lem-MI} is proved via the union bound: For each fixed pair $(m,t)$, we show that \eqref{eq-moderate-increment} happens w.o.p.. To this end, we note that conditioned on any realization of $x_{i,m}^{(t)}$ for $i\in [n]$, $S_{m}^{(t+1)}$ is a sum of conditionally independent random $\pm 1$ variables with $\mathbb{E}[S_{m}^{(t+1)}]$ close to $S_{m}^{(t)}$. In light of this, it follows from a standard fact about the concentration of the sum of independent random variables that \eqref{eq-moderate-increment} holds w.o.p.. The detailed proof can be found in Appendix~\ref{appendix-lem-MI}.

Next, we provide a lemma that, at each time step, translates global information about the aggregate number of opinions in the graph into local information for every neighborhood. 

\begin{definition} [From global to local]
    Define $\mathcal G_{\operatorname{FGTL}}$ as the following event: For any $1\le m\le M$ and $t\ge 0$, it holds that
    \begin{equation}\label{eq-from-local-to-global-event}
    \begin{aligned}
\frac{p}{10}(n-|S_m^{(t)}|)\mathbf{1}\{n-|S_m^{(t)}|\ge p^{-1}(\log n)^2\}\le &\ \min_{i\in [n]}\{d_i-|S_{i,m}^{(t)}|\}\\
\le \max_{i\in [n]}\{d_i-|S_{i,m}^{(t)}|\}\le &\ \max\{10(\log n)^2,10p(n-|S_m^{(t)}|)\}\,.
\end{aligned}
    \end{equation}
\end{definition}

\begin{lemma}\label{lem-FGTL}
     $\mathcal G_{\operatorname{FGTL}}$ happens w.o.p.. 
\end{lemma}

The central idea in the proof is to condition on the realization of $x_{m,i}^{(t-L)}$, where $L$ is such that the random walk on $G^*$ started at any vertex mixes well within $L$ steps. Intuitively, after this conditioning, one should expect the signs at time $t$ of any two vertices to behave almost like independent random variables. Thus, the distribution of signs within each neighborhood should resemble the distribution on the entire vertex set. Making this precise requires some work, but the proof can be carried out via an inductive argument that relies once again on a standard concentration result for a weighted sum of independent random variables. The proof is provided in Appendix~\ref{appendix-lem-FGTL}.

%Although the proof of Lemma~\ref{lem-FGTL} follows in a similar spirit as Lemma~\ref{lem-MI}, where the essential input is concentration of (weighted) sum of independent random variables. 

\begin{comment}
\begin{lemma}\label{lem-expectation}\HD{Not sure if this is useful anymore. I will move the argument to where it is used.}
    For any $m,t$ it holds that
    \[
    \mathbb{E}\left[\frac{(x_{j}^{(t)}-x_k^{(t)})^2}{\max\{np,d_i^2-|S_i^{(t)}|^2\}}\right]\le (np)^{-2}(\log n)^{3}\,.
    \]
\end{lemma}
\begin{proof}
    We conditioned on $\mathcal G_{\operatorname{MI}}\cap \mathcal G_{FGTL}$ happens and $n-|S^{(t)}|=\Delta$. If $p\Delta\ge (\log n)^2$, from $\mathcal G_{FGTL}$ we know $$\frac{p\Delta}{10}\le d_j-|S_j^{(t)}|,d_k-|S_k^{(t)}|\le 10p\Delta\,,$$
    and so $p_j^{(t)},p_k^{(t)}\le \frac{20\Delta}{np}$. By $\mathcal G_{\operatorname{MI}}$ we have $n-|S^{(t-1)}|\ge 0.5\Delta$. Therefore the expected value (conditionally on $n-|S^{(t)}|=\Delta$) is bounded by $\frac{1}{np\cdot 0.5\Delta}\cdot \frac{40\Delta}{np}=80(np)^{-2}$. For the case $p\Delta\le (\log n)^2$, we have $p_j^{(t)},p_k^{(t)}\le \frac{10(\log n)^2}{np}$, and so the conditional expectation is bounded by $\frac{1}{np}\cdot \frac{20(\log n)^2}{np}=20(np)^{-2}(\log n)^2$. In conclusion, we have the expectation is bounded by \[80(np)^{-2}+20(np)^{-2}(\log n)^2+\exp(-\omega(\log n))\le (np)^{-2}(\log n)^3\,.\qedhere\]
\end{proof}
\end{comment}

In light of Lemma~\ref{lem-FGTL}, the trajectory of $\{S_{m}^{(t)}\}_{t\ge 0}$ captures the behavior of all trajectories of $\{S_{m,i}^{(t)}\}_{t\ge 0},i\in [n]$. Lastly, the lemma below addresses some important properties of the evolution of $\{S_m^{(t)}\}_{t\ge 0}$. %\SH{In Def 4.6, it is $\{S_{m}^{(t)}\}_{t=1}^T$, while we used $\{S_m^{(t)}\}_{t\ge 0}$ in this paragraph. Which notation do you guys prefer? We also have $\{S_{m}^{(t)}\}_{t\ge 0}$ and $\{W_{m}^{(t)}\}_{t\ge 0}$}\HD{Let's use $\{S_m^{(t)}\}_{t\ge 0}$} \SH{This issue looks resolved now to me}

\begin{definition}[$\{S_{m}^{(t)}\}_{t\geq 0}$ behaves like a random walk]
    Define $\mathcal G_{\operatorname{RW}}$ as the event that the following two properties hold:\\
    \noindent (i) For any $1\le m\le M$ and $0\le t\le \min\{T,n/(\log n)^{10}\}$, it holds that $|S_{m}^{(t)}|\le \frac n2$. \\
    \noindent (ii) For any $1\le \Delta\le n/(\log n)^{10}$, $\#\{1\le m\le M,0\le t\le T:n-\Delta\le |S_m^{(t)}|<n\}$ is bounded from above by $M\cdot \max\{\Delta (\log n)^5,(\log n)^{10}\}$. %\OS{The log exponents got worse. I could make then better (probably 3 and 5, respectively) but it seems like this is enough.}
\end{definition}

\begin{lemma}\label{lem-RW}
    $\mathcal G_{\operatorname{RW}}$ happens w.o.p..%\HD{Here we may need $\mathcal G_{\operatorname{RW}}$ happens w.o.p., is this true?}\OS{It should be true w.o.p. I'll make all the required changes to the proof this Saturday. Is probability $1-n^{-10}$ enough?} \HD{Yeah that would be enough, but ideally we can show w.o.p. so we can unify the convention.}
\end{lemma} 
The intuition behind this lemma is that for each $1\le m\le M$, $S_m^{(t)}$ behaves like a random walk on $\{-n,\cdots,n\}$ such that, given $S_{m}^{(t)}$, the next step increment $S_{m}^{(t+1)}-S_m^{(t)}$ has expectation close to $0$ and variance of order $n-|S_m^{(t)}|$.
The proof relies on analyzing the martingales $\{W_m^{(t)}\}_{t\geq0}$, $1\le m\le M$, where $W_m^{(t)}=\mathbb E_{p\sim\pi}[x_{m,p}^{(t)}]$. We have already used this martingale during the proof of Theorem~\ref{thm-main} (see Lemma~\ref{lem-effective-observations-lower-bound}). Once again, we exploit the fact that, up to a small error, $W_m^{(t)}$ acts as a normalized version $S_m^{(t)}$. That Property-(i) in the definition of $\mathcal G_{\operatorname{RW}}$ holds w.h.p. follows from a conditional version of Azuma's inequality applied to this martingale and the fact that, at each time step, $W_m^{(t)}$ is unlikely to change by a large amount by Lemma~\ref{lem-MI}. Showing that Property-(ii) holds w.o.p. requires more delicate arguments. At a high level, the key is to exploit the fact that if, at some time $t$, we have that $|W_m^{(t)}|$ is close to $1$, then there is a good chance that $|W_m^{(t+1)}|,|W_m^{(t+2)}|,\cdots$ reaches $1$ without ever going too far below $|W_m^{(t)}|$. To make this precise, we consider a dyadic partition of the interval $[-1,1]$ into smaller subintervals and show that the index of the subinterval in which the martingale finds itself at time $t$ evolves as a biased random walk. At some point, it is also important to argue that the martingale is unlikely to remain within a single subinterval for a long time. This can be achieved by providing a lower bound on $\mathbb E[(W_m^{(t+1)}-W_m^{(t)})^2]$ using Property-(iv) of admissibility. The detailed proof can be found in Appendix~\ref{appendix-lem-RW}.

%As mentioned earlier, the complete proof of the three lemmas above can be found in Appendix~\ref{appendix-proof-of-key-lemmas}.\OS{We can do this later, but we should set up links to the appendix so that the reader doesn't have to scroll through 20 pages every time.}\HD{Agreed.}

%\begin{definition}[(Consensus is attained after linearly many steps)]Define $\mathcal G_{\operatorname{CON-K}}$ as the event that $G^*$ satisfies the following property:With any fixed $m$, w.o.p. over the randomness of the voter model dynamics we have that $x_{m,i}^{(t)}=x_{m,j}^{(t)}$ for all $i,j\in [n]$ and all $t\geq K.$ 
%\end{definition}

\subsection{Proof of Proposition~\ref{eq-prop-three-properties}}\label{subsec-proof-of-three-properties}
Here, we provide the proof of Proposition~\ref{eq-prop-three-properties}. We will show that Item-(i) holds w.h.p., Item-(ii) holds with probability at least $0.99-o(1)$, and Item-(iii) holds with probability $0.95-o(1)$. Then, by the union bound, all three items hold simultaneously with probability at least $0.94-o(1)>0.92$, as desired.

\subsubsection{Item-(i)}
We now show that Item-(i) in Proposition~\ref{eq-prop-three-properties} happens w.h.p.. First, since $\gamma\le \delta/10$, it follows from the union bound that w.h.p. $a_i,b_i,c_i$ are distinct for each $1\le i\le n^\gamma$. Next, we show that all the triples $(a_i,b_i,c_i)$ are good w.h.p..

We fix a triple $(a_i,b_i,c_i)\triangleq (a,b,c)$. From Lemmas~\ref{lem-FGTL} and \ref{lem-RW}, by losing a super-polynomially small probability, we may assume that $\mathcal G_{\operatorname{FGTL}}\cap \mathcal G_{\operatorname{RW}}$ holds. If $T\le n/(\log n)^{10}$, then Property-(i) in $\mathcal G_{\operatorname{RW}}$ implies that $|S_m^{(t)}| \leq\frac{n}{2}$ for all values of $1\le m\le M$ and $0\le t\le T$. By the first inequality in \eqref{eq-from-local-to-global-event}, we conclude that $d_a-|S_{m,a}^{(t)}|\ge np/20$ for any $1\le m\le M$ and $0\le t\le T$. This implies that all triples are good. In what follows, we assume that $T> n/(\log n)^{10}$ and hence $M=\widetilde{O}(np^2)$ by condition~\eqref{eq-main-condition}.

By Property-(ii) of $\mathcal G_{\operatorname{RW}}$, $\#\{1\le m\le M,0\le t\le T:0<n-|S_m^{(t)}|\le 10p^{-1}(\log n)^2\}$ is bounded by
$$M\times\widetilde{O}(p^{-1})=\widetilde{O}(np)\,.$$ 
Thus, from $\mathcal G_{\operatorname{FGTL}}$ we conclude that $\#\{m,t:d_a-|S_{m,a}^{(t)}|=2\}$ is also upper-bounded by $ \widetilde{O}(np)\le n^{-\delta/4}\cdot cn^2p^2\log n$. Additionally, for any $1\le i\le n^\gamma$, $1\le m\le M$, and $0\le t\le T$ such that $d_a-|S_{m,a}^{(t)}|=2$, in order for $(m,t)$ to act as a witness that $(a,b,c)$ is a bad triple, it must also be the case that $x_{m,a}^{(t+1)}S_{m,a}^{(t)}<0$ and $x_{m,b}^{(t)}\neq x_{m,c}^{(t)}$. We claim that these two additional constraints are satisfied with conditional probability at most $O\left(\frac{(\log n)^2}{(np)^2}\right)$. 

To show this, we first observe that \[\mathbb{P} \left[x_{m,a}^{(t+1)}S_{m,a}^{(t)}<0\ \mid\ d_a-|S_{m,a}^{(t)}|=2\right]= \frac{1}{d_a}=O\left(\frac{1}{np}\right)\, . \] 
Moreover, given that $\mathcal G_{\operatorname{FGTL}}$ occurs, the assumption that $d_a-|S_{m,a}^{(t)}|=2$ also implies $n-S_n^{(t)}\leq p^{-1}(\log n)^2$. Since $b$ and $c$ were chosen uniformly at random from $\operatorname{N}_a(G^*)$ and $[n]\backslash(\{a\}\cup\operatorname{N}_a(G^*))$, independently of the voter model dynamics, we have \[\mathbb{P} \left[x_{m,b}^{(t)}\neq x_{m,c}^{(t)}\ \mid\ d_a-|S_{m,a}^{(t)}|=2,\ x_{m,a}^{(t+1)}S_{m,a}^{(t)}<0\right]\leq O\left(\frac{(\log n)^2}{np}\right)\,,\]
verifying the claim.

Combining these observations and using the union bound, we get that the expected number of bad triples conditioned on $\mathcal G_{\operatorname{FGTL}}\cap\mathcal G_{\operatorname{RW}}$ is upper-bounded by 
\[n^{\gamma}\cdot  n^{-\delta/4}\cdot cn^2p^2\log n\cdot O\left(\frac{(\log n)^2}{(np)^2}\right)\leq \widetilde{O}(n^{-\delta/8})\,.\] 
A direct application of Markov's inequality completes the proof.

%\HD{Recall that we assumed $M\cdot \min\{T,n\}=cn^2p^2\log n$. If $T\le n^{0.9}$, then property (i) in $\mathcal G_{\operatorname{RW}}$ implies that any triple is good. Hence it suffices to consider the case $T\ge n^{0.9}$, and hence $np\ge n^{0.4}$. By property-(ii) of $\mathcal G_{\operatorname{RW}}$, $\#\{1\le m\le M,0\le t\le T:n-|S_m^{(t)}|\le 2p^{-1}\}$ is at most $M\cdot p^{-1}\log n$. Thus under $\mathcal G_{\operatorname{FGTL}}\cap \mathcal G_{\operatorname{RW}}$ it holds for any $a\in [n]$, $\#\{m,t:d_a-|S_{m,a}^{(t)}|=2\}\le M\cdot p^{-1}\log n\le M\cdot n^{0.6}\log n\le n^{-0.2}\cdot (cn^2p^2\log n)$. If $d_a-|S_{m,a}^{(t)}|=2$, in order for $x_{m,a}^{(t+1)}S_{m,a}^{(t)}<0$, one lose a probability factor $O(\frac{1}{np}$. Additionally, in order for $x_{m,b}^{(t)}\neq x_{m,c}^{(t)}$, one further lose a probability factor $O(\frac{\log n}{np})$ (the randomness comes from the choices of $b,c$). A Markov inequality then yields Property (i) holds w.h.p..}

\subsubsection{Item-(ii)}

We now show that Item-(ii) happens with probability at least $0.99-o(1)$. For $1\le m\le M$, recall that $T_m=\min\{T,T_{\operatorname{coal}}(m)\}$ is the number of effective observations in the $m$-th voter model trajectory. First, we consider the event $\mathcal {G}$ that $T_1,\cdots,T_m$ satisfy
\begin{equation}\label{eq-sum-Tm}
\sum_{m=1}^M T_m\le \max\{100C^*,1\}\cdot cn^2p^2\log n\,,
\end{equation}
where $C^*$ is the constant in Property-(iii) of admissibility. 
We claim that $\mathcal G$ happens with a probability of at least $0.99$. Indeed, if $T\le n$, then using the trivial fact that $T_m\le T$ for all $m$ and condition \eqref{eq-condition-negative} (which states that $M\cdot T\le cn^2p^2\log n$), we see that \eqref{eq-sum-Tm} holds deterministically. If $T\ge n$, then condition \eqref{eq-condition-negative} becomes $M\le cnp^2\log n$. Since $T_m$ is bounded by the coalescence time, it follows from Property-(iii) of admissibility that
\[
\mathbb{E}\left[\sum_{m=1}^M T_m\right]\le M\cdot C^*n\le C^*\cdot cn^2p^2\log n\,.
\]
Markov's inequality then yields that \eqref{eq-sum-Tm} fails with probability at most $0.01$, verifying the claim. 

%Now assume \eqref{eq-sum-Tm} holds, we show that for each $1\le i\le n^\gamma$, \eqref{eq-quadratic-term} holds w.o.p.. The result then follows by taking the union bound. In what follows, by losing a super-polynomially small probability, we also assume $\mathcal G_{\operatorname{MI}}\cap \mathcal G_{\operatorname{FGTL}}\cap \mathcal G_{\operatorname{RW}}$ happens. 

For any triple $(a,b,c)$, we define
\[
\mathcal T_{a,b,c}:= \sum_{m=1}^M\sum_{t=0}^{T_m-1} \big(X_{m,a,b,c}^{(t)}\big)^2\,.
\]
Our goal is to upper-bound the maximum of $\mathcal T_{a_i,b_i,c_i}$ over $1\le i\le n^\gamma$. We have that for any $x\ge 0$
\begin{align*}
\mathbb{P}\Big[\max_{1\le i\le n^\gamma}\mathcal T_{a_i,b_i,c_i}\ge x\Big]
\le &\ \mathbb{P}\big[(\mathcal G\cap \mathcal G_{\operatorname{MI}}\cap \mathcal G_{\operatorname{FGTL}}\cap \mathcal G_{\operatorname{RW}})^c\big]\\+&\ \mathbb{P}\Big[\max_{1\le i\le n^\gamma}{\mathcal T}_{a_i,b_i,c_i}\ge x,\mathcal G\cap \mathcal G_{\operatorname{MI}}\cap  \mathcal G_{\operatorname{FGTL}}\cap  \mathcal G_{\operatorname{RW}}\Big]\\
\le&\ 0.01+o(1)+\sum_{1\le i\le n^\gamma}\mathbb{P}\big[{\mathcal T}_{a_i,b_i,c_i}\ge x,\mathcal G\cap \mathcal G_{\operatorname{MI}}\cap \mathcal G_{\operatorname{FGTL}}\cap \mathcal G_{\operatorname{RW}}\big]\,,
\end{align*}
where the last inequality follows from the union bound. In light of this, to show that with probability at least $0.96$ the maximum over $\mathcal T_{a_i,b_i,c_i}$ is no more than $C_0c\log n$, it suffices to show that for each triple $(a,b,c)$ the probability
$\mathbb{P}\big[{\mathcal T}_{a,b,c}\ge C_0c\log n]$
is super-polynomially small. 

Toward this end, we fix a triple $(a,b,c)$ and define
\[
\widetilde{\mathcal T}_{a,b,c}:= \sum_{m=1}^M\sum_{t=0}^{T_m-1} \left(\frac{16}{(np)^2}+\frac{4Y_{m,t}}{\big(d_{a}-|S_{m,a}^{(t)}|\big)^2}\right)\,,
\]
where $Y_{m,t},1\le m\le M,0\le t\le T_m-1$ are independent Bernoulli variables with parameters $p_{m,0}=1$ and $$
p_{m,t}:= \min\Big\{\frac{10^5\max\{p(n-|S_m^{(t-1)}|),(\log n)^2\}^2}{(np)^2},1\Big\}\,,\forall 1\le t\le T\,.
$$
We have the following stochastic domination relation between $\mathcal T_{a,b,c}$ and $\widetilde{\mathcal T}_{a,b,c}$.
\begin{lemma}\label{lem-stochastic-domination}
   There is a coupling between $\mathcal T_{a,b,c}$ and $\widetilde{\mathcal T}_{a,b,c}$ such that almost surely
   $$\mathcal T_{a,b,c}\le \widetilde{\mathcal T}_{a,b,c}+MT\cdot\mathbf{1}\{(\mathcal G_{\operatorname{MI}}\cap \mathcal G_{\operatorname{FGTL}})^c\}\,.$$
\end{lemma}

To develop an intuition for Lemma~\ref{lem-stochastic-domination}, recall the definition
\[
X_{m,a,b,c}^{(t)}=\frac{x_{m,a}^{(t+1)} (x_{m,b}^{(t)}-x_{m,c}^{(t)})}{d_{a}+x_{m,a}^{(t+1)}\cdot S_{m,a}^{(t)}}\,.
\]
This leads to
\begin{equation}\label{eq-3-cases}
\big(X_{m,a,b,c}^{(t)}\big)^2=\begin{cases}
    \frac{4}{(d_a+|S_{m,a}^{(t)}|)^2}\,,\quad&\text{ if }x_{m,a}^{(t+1)}S_{m,a}^{(t)}\ge 0\text{ and }x_{m,b}^{(t)}\neq x_{m,c}^{(t)}\,,\\
    \frac{4\cdot \mathbf{1}\{|S_{m,a}^{(t)}|<d_a\}}{(d_a-|S_{m,a}^{(t)}|)^2}\,,\quad&\text{ if $x_{m,a}^{(t+1)}S_{m,a}^{(t)}<0$ and $x_{m,b}^{(t)}\neq x_{m,c}^{(t)}$}\,,\\
    0\,,\quad&\text{ otherwise}\,.
\end{cases}
\end{equation}
(Here, we use the convention that $\frac00=0$.) For the first and the third cases, since $d_a\ge np/2$, we obtain
\[
\big(X_{m,a,b,c}^{(t)}\big)^2\le \frac{16}{(np)^2}\,.
\]
To handle the second case, we fix $m$ and $t$ and condition on any realization of $\{x_{m,i}^{(t')}:i\in [n],t'<t\}\cup \{x_{m,a}^{(t)}\}$ that does not violate $\mathcal G_{\operatorname{MI}}\cap \mathcal G_{\operatorname{FGTL}}$. Now, it suffices to prove that both $\mathcal G_{\operatorname{MI}}\cap \mathcal G_{\operatorname{FGTL}}$ and the second case in \eqref{eq-3-cases} happen simultaneously with conditional probability at most $p_{m,t}$. The details can be found in Appendix~\ref{appendix-lem-stochastic-domination}. 

With this domination relation in hand, since $\mathcal G_{\operatorname{MI}}\cap \mathcal G_{\operatorname{FGTL}}$ happens w.o.p., it remains to show that the probability
\[
\mathbb{P}\big[\widetilde{\mathcal T}_{a,b,c}\ge C_0c\log n,\mathcal G\cap \mathcal G_{\operatorname{MI}}\cap \mathcal G_{\operatorname{FGTL}}\cap \mathcal G_{\operatorname{RW}}\big]
\]
is super-polynomially small. Note that if we fix the realization of all random variables $x_{m,i}^{(t)}$ for $1\le m\le M,0\le t\le T_m$, and $i\in [n]$, then $\widetilde{\mathcal T}_{a,b,c}$ becomes a weighted sum of independent Bernoulli variables plus the deterministic constant 
\[
\frac{16}{(np)^2}\sum_{m=1}^{M}T_m\,.
\]
We show that, provided that the realization of $\{x_{m, i}^{(t)}\}$ satisfies $\mathcal G\cap \mathcal G_{\operatorname{MI}}\cap \mathcal G_{\operatorname{FGTL}}\cap \mathcal G_{\operatorname{RW}}$, then $\widetilde{\mathcal T}_{a,b,c}$ does not exceed $C_0c\log n$ w.o.p. (where the randomness is now only over the Bernoulli variables). 

First, $\mathcal G$ implies that the deterministic constant is at most $16\max\{100C^*,1\}\cdot c\log n$. Recalling our choice of $C_0$ in \eqref{eq-C0}, it suffices to prove that given $\mathcal G_{\operatorname{MI}}\cap \mathcal G_{\operatorname{FGTL}}\cap \mathcal G_{\operatorname{RW}}$, w.o.p. the weighted sum satisfies
\begin{equation}\label{eq-900c}
\sum_{m=1}^M\sum_{t=0}^{T_m-1}\frac{Y_{m,t}}{\big(d_{a}-|S_{m,a}^{(t)}|\big)^2}\le 10^6c\log n\,.
\end{equation}
The idea is to first argue that under these three typical events, the expected value of the weighted sum is far below $10^6c\log n$. Then, we apply concentration inequalities for weighted sums of independent Bernoulli variables to conclude the result. For technical reasons, we decompose the sum in \eqref{eq-900c} into two parts and analyze them separately. 

Fixing the triple $(a,b,c)$, define $\mathcal P_{\operatorname{good}}$ as the set of pairs $(m,t)$ with $1\le m\le M$ and $0\le t\le T_m-1$ such that $d_a-|S_{a,m}^{(t)}|\ge (\log n)^2$. Let $\mathcal P_{\operatorname{bad}}$ be the remaining pairs. Let $\mathcal S_{\operatorname{good}}$ and $\mathcal S_{\operatorname{bad}}$ be the sums of good and bad terms in \eqref{eq-900c}, respectively. 

For a good pair $(m,t)$, since $d_a-|S_{m,a}^{(t)}|\ge (\log n)^2$, by $\mathcal G_{\operatorname{MI}}\cap \mathcal G_{\operatorname{FGTL}}$ we have $p_{m,t}\le \frac{10^{5}(d_a-|S_{m,a}^{(t)}|)^2}{(np)^2}$. Hence,
\[
\mathbb{E}\left[\frac{Y_{m,t}}{\big(d_a-|S_{m,a}^{(t)}|\big)^2}\right]\le \frac{10^5}{(np)^2}\,,
\]
and thus the expected value of $\mathcal S_{\operatorname{good}}$ is upper-bounded by $10^5c\log n$. 

\begin{comment}
    Using Lemma~\ref{lem-weighted-Bernoulli-concentration} with $\lambda^*=(\log n)^{-4}$, $\theta=1/(10\lambda^*)$ and $M=100c\log n$, we get that the sum over good pairs exceeds its expectation plus $100c\log n$ happens with probability $\exp(O(\log n)-\Omega((\log n)^5))$, which is super-polynomially small. Therefore, w.o.p. the sum of good pairs is bounded by $500c\log n$.
\end{comment}

For any bad pair $(m,t)$, $\mathcal G_{\operatorname{FGTL}}$ ensures that $n-|S_{m}^{(t)}|\le 10p^{-1}(\log n)^2\le \widetilde{O}(n^{1-\delta})$. Consequently, from $\mathcal G_{\operatorname{RW}}$, we get that $\mathcal P_{\operatorname{bad}}=\varnothing$ unless $T\ge n/(\log n)^{10}$. Now, assuming $T\ge n/(\log n)^{10}$, as argued in the proof of Item-(i), $\mathcal G_{\operatorname{RW}}$ guarantees that the size of $\mathcal P_{\operatorname{bad}}$ is at most $\widetilde{O}(np)$. Additionally, for bad pairs, we have $p_{m,t}\le \frac{10^5(\log n)^2}{(np)^2}$. Therefore, the expected value of $\mathcal S_{\operatorname{bad}}$ is upper-bounded by 
\[
\widetilde{O}(np)\times \frac{100(\log n)^2}{(np)^2}\le \widetilde{O}((np)^{-1})\le n^{-\delta/2}\,.
\]

The next lemma shows that w.o.p., $\mathcal S_{\operatorname{good}}$ and $\mathcal S_{\operatorname{bad}}$ do not exceed their expectations by more than $10^5c\log n$.
\begin{lemma}\label{lem-Sgood-Sbad}
    Given $\{x_{m,i}^{(t)}\}$ satisfying $\mathcal G\cap \mathcal G_{\operatorname{RW}}$, it holds that both 
    \[
    \mathbb{P}[\mathcal S_{\operatorname{good}}-\mathbb{E}[\mathcal S_{\operatorname{good}}]\ge 10^5c\log n] \quad \text{and} \quad  \mathbb{P}[\mathcal S_{\operatorname{bad}}-\mathbb{E}[\mathcal S_{\operatorname{bad}}]\ge 10^5c\log n]
    \]
    are super-polynomially small (here, the randomness is over $Y_{m,t},1\le m\le M,0\le t\le T_m-1$). 
\end{lemma}

This follows from standard concentration inequalities, and the details can be found in Appendix~\ref{appendix-lem-Sgood-Sbad}. We have shown that under $\mathcal G\cap \mathcal G_{\operatorname{MI}}\cap \mathcal G_{\operatorname{FGTL}}\cap \mathcal G_{\operatorname{RW}}$, $\widetilde{\mathcal T}_{a,b,c}\le C_0c\log n$ happens w.o.p.. Building on previous arguments, this proves that Item-(ii) holds with probability at least $0.99-o(1)$, as desired.

\subsubsection{Item-(iii)}

We now show that Item-(iii) holds with probability $0.95-o(1)$. For simplicity, define
\[
\mathcal S_{a_i,b_i,c_i}:=\sum_{m=1}^M\sum_{t=0}^{T_m-1}X_{m,a_i,b_i,c_i}^{(t)}\,,\quad \forall 1\le i\le n^\gamma\,.
\]
Our goal is to establish that, w.h.p., 
\[
\max_{1\le i\le n^\gamma}\mathcal S_{a_{i},b_{i},c_{i}}\ge \gamma\sqrt c\log n\,.
\]

Before delving into the proof, let us outline the strategy that we will follow. Our starting point is the second-moment method: Writing
\[
X:=\sum_{i=1}^{n^\gamma}\mathbf{1}\{\mathcal S_{a_i,b_i,c_i}\ge \gamma\sqrt c\log n\}\,,
\]
we have from the Paley-Zygmund inequality,
\begin{align*}
\mathbb{P}\Big[\max_{1\le i\le n^\gamma}\mathcal S_{a_i,b_i,c_i}
\ge\gamma\sqrt c\log n\Big]=&\ \mathbb{P}[X>0]\ge \frac{\big(\mathbb{E}[X]\big)^2}{\mathbb{E}[X^2]}\\=&\ \frac{\sum_{i,j=1}^{n^\gamma}\mathbb{P}[\mathcal S_{a_i,b_i,c_i}\ge \gamma\sqrt c\log n]\mathbb{P}[\mathcal S_{a_j,b_j,c_j}\ge \gamma\sqrt c\log n]}{\sum_{i,j=1}^{n^\gamma}\mathbb{P}[\mathcal S_{a_i,b_i,c_i}\ge \gamma\sqrt c\log n,\mathcal S_{a_j,b_j,c_j}\ge \gamma\sqrt c\log n]}\,.
\end{align*}
However, a direct application of the above inequality only yields an $o(1)$ lower bound because the fluctuation of $X$ is too large, making the inequality too loose. To address this, we introduce a suitable filtration $\mathcal{F}_{-\mathcal{I}}$ (see \eqref{eq-F_{-I}} below) and apply a conditional second-moment method. The key intuition is that conditioning on $\mathcal{F}_{-\mathcal{I}}$ significantly reduces the fluctuation of $X$, allowing the second-moment method to yield an effective lower bound.

Nevertheless, several technical steps are required to facilitate the analysis. First, we define good realizations of $\mathcal{F}_{-\mathcal{I}}$ and show that a realization is good with probability close to 1, allowing us to restrict our attention to these cases. We then decompose each sum $\mathcal{S}_{a_i,b_i,c_i}$ into two components: $S(i)$, which contains the ``bulk'' terms, and $\overline{S}(i)$, the ``edge'' terms (see \eqref{eq-S(i)} below), and analyze them separately. Intuitively, $\overline{S}(i)$ is small since it consists of only a few terms. We establish that w.h.p., $\max_{1\le i\le n^\gamma}|\overline{S}(i)| \leq \gamma\sqrt{c} \log n$ (Proposition~\ref{prop-ovelrine-S(i)}). Next, applying the second-moment method, we show that w.h.p., $\max_{1\le i\le n^\gamma}S(i) \geq 2\gamma\sqrt{c} \log n$ (Proposition~\ref{prop-S(i)}), completing the proof. We remark that the key reason for separating bulk and edge terms is to obtain a strong quantitative central limit theorem for $S(i)$ (Lemma~\ref{lem-quantitative-estimate}), which is crucial for the second-moment method analysis.

We now provide the proof details. To prepare for the presentation, we fix the set of indices $\mathcal I=\{a_i,b_i,c_i,1\le i\le n^\gamma\}$. For any subset $I\subseteq \mathcal I$, we define
\begin{equation}\label{eq-F_{-I}}
    \mathcal F_{-I}:=\sigma\left\{x_{m,j}^{(t)}:0\le m\le M,0\le t\le T_m,j\in [n]\setminus I\right\}\,.
\end{equation}
Moreover, for $1\le m\le M$ and $0\le t\le T_m$, we write
\[
S_{m,-\mathcal I}^{(t)}:= \sum_{j\in [n]\setminus \mathcal I}x_{m,j}^{(t)}\,,\quad S_{m,a_i,-\mathcal I}^{(t)}:= \sum_{j\in \operatorname{N}_i\setminus \mathcal I}x_{m,j}^{(t)}\,,i\in [n]\,.
\]
We also define the indicators
\[
I_{m}^{(t)}:=\mathbf{1}\{n-|S_{m,-\mathcal I}^{(t)}|\ge n^{1-\delta/20}\},\,\,\quad
\overline{I}_m^{(t)}:=1-I_m^{(t)}\,.
\]
Clearly, $S_{m,-\mathcal I}^{(t)},S_{m,a_i,-\mathcal I}^{(t)}$, and $I_{m}^{(t)}$ are all measurable with respect to $\mathcal F_{-\mathcal I}$. 

For $1\le i\le n^\gamma$, recall that we denote $X_{m,a_i,b_i,c_i}^{(t)}= \frac{x_{m,a_i}^{(t+1)}(x_{m,b_i}^{(t)}-x_{m,c_i}^{(t)})}{d_{a_i}+x_{m,a_i}^{(t+1)}\cdot S_{m,a_i}^{(t)}}$. Let
\begin{equation}\label{eq-S(i)}
S(i):=\sum_{m=1}^M\sum_{t=0}^{T_m-1}I_{m}^{(t)}X_{m,a_i,b_i,c_i}^{(t)},\quad\overline{S}(i):=\sum_{m=1}^M\sum_{t=0}^{T_m-1}\overline{I}_{m}^{(t)}X_{m,a_i,b_i,c_i}^{(t)}\,.
\end{equation}
%Recall that our goal is to show that there exists $1\le i\le n^\gamma$ such that $S(i)+\overline{S}(i)\ge \gamma \sqrt c\log n$. 
Additionally, for each $1\le i\le n^\gamma$, we write 
$$U_m^{(t)}(i):=\begin{cases}\frac{(x_{m,b_i}^{(t)}-x_{m,c_i}^{(t)})^2}{d_{a_i}^2-|S_{m,a_i}^{(t)}|^2}\,,\quad&\text{if }|S_{m,a_i}^{(t)}|\neq d_{a_i}\,,\\
\frac{(x_{m,b_i}^{(t)}-x_{m,c_i}^{(t)})^2}{2d_{a_i}}\,,\quad&\text{if $|S_{m,a_i}^{(t)}|=d_{a_i}$}\,.\end{cases}$$ 
Note that by taking the expectation of $x_{m,a_i}^{(t+1)}$, we have $$\mathbb{E}\Big[\big(X_{m,a_i,b_i,c_i}^{(t)}\big)^2\mid \mathcal F_{-\mathcal I}\Big]=\mathbb{E}[U_{m}^{(t)}(i)\mid \mathcal F_{-\mathcal I}]\,,\ \forall 1\le m\le M,0\le t\le T_m\,.$$ 

We proceed to define good realizations of $\mathcal F_{-\mathcal I}$. 
\begin{definition}
    We call a realization of $\mathcal F_{-\mathcal I}$ good, if the following properties hold: 
        \\
    \noindent (a) %Denoting $S_{m,i,-\mathcal I}^{(t)}:= \sum_{i\rightarrow j,j\notin \bigcup\{a_k,b_k,c_k\}}x_{m,j}^{(t)}$, 
    For any $1\le i\le n^\gamma$,
    \begin{equation}\label{eq-good-realization-1}
        \mathbb{P}\left[\sum_{m=1}^M\sum_{t=0}^{T_m-1}\overline{I}_m^{(t)}U_{m}^{(t)}(i)\ge \gamma^2c\log n\mid \mathcal F_{-\mathcal I}\right]\le n^{-\delta/10}\,.
    \end{equation}
    \\
    \noindent (b) For at least half of indices $i\in [n^\gamma]$, we have that \begin{equation}\label{eq-good-realization-2}
\mathbb{E}\left[\sum_{m=1}^M\sum_{t=0}^{T_m-1}I_m^{(t)}U_m^{(t)}(i)\mid \mathcal F_{-\mathcal I}\right]\ge 10^{-3}\zeta c\log n\,.
    \end{equation}\\
(c) It holds that
\begin{equation}\label{eq-good-realization-3}
    \mathbb{E}\left[\sum_{1\le i\neq j\le n^\gamma}\Bigg(\sum_{m,t}I_m^{(t)}\frac{x_{m,a_j}^{(t)}(x_{m,b_i}^{(t)}-x_{m,c_i}^{(t)})}{(d_{a_i}+x_{m,a_i}^{(t+1)}S_{m,a_r,-\mathcal I}^{(t)})^2}\Bigg)^2\mid \mathcal F_{-\mathcal I}\right]\le n^{-\delta/5}\,.
\end{equation}
\end{definition}

%We will show that a realization of $\mathcal F_{-\mathcal I}$ is typically good, and for good realizations, it happens with high conditional probability that $\max_{1\le i\le n^\gamma}|\overline{S}(i)|=o(\log n)$, and $\max_{1\le i\le n^\gamma}S(i)\ge 2\gamma\sqrt c\log n$, as stated in the next three propositions. Upon showing these propositions, the proof of item-(iii) holds w.h.p. follows. 

The next proposition states that good realizations are indeed typical; its proof is deferred to Appendix~\ref{appendix-prop-good-realization-typical}. 
\begin{proposition}\label{prop-good-realization-typical}
    A realization of $\mathcal F_{-\mathcal I}$ is good with probability at least $0.95-o(1)$.
\end{proposition}

In what follows, we fix a good realization $\omega$ of $\mathcal F_{-\mathcal I}$ and condition on it. To understand why the properties in Proposition~\ref{prop-good-realization-typical} are important, we make the following observations. Since $I_m^{(t)}$ and $\overline{I}_m^{(t)}$ are now deterministic quantities, we can write
\[
S(i)=\sum_{I_m^{(t)}=1}X_{m,a_i,b_i,c_i}^{(t)}\,,\quad \overline{S}(i)=\sum_{\overline{I}_m^{(t)}=1}X_{m,a_i,b_i,c_i}^{(t)}\,.
\]
%\OS{added := in both places} \SH{ob didn't we already define $S(i)$ and $\overline{S}(i)$ in (4.12)?}\OS{right! reverted the cchanges}
We further condition on a realization $\widetilde{\omega}$ of $\mathcal F_{-\{a_i\}}$ (recall \eqref{eq-F_{-I}}) that is compatible with $\omega$, and we denote the conditional probability measure given $\widetilde{\omega}$ as $\widetilde{\mathbb P}$. Then, all the random variables $X_{m,a_i,b_i,c_i}^{(t)}$ are independent under $\widetilde{\mathbb P}$. Additionally, by a straightforward calculation, we have
\[
\widetilde{\mathbb{E}}[X_{m,a_i,b_i,c_i}^{(t)}]=0\,,\quad \operatorname{Var}_{\widetilde{\mathbb P}}[X_{m,a_i,b_i,c_i}^{(t)}]=U_m^{(t)}(i)\,.
\]
Therefore, $S(i)$ (resp. $\overline{S}(i)$) is a sum of independent random variables with mean $0$ and variance $\sum_{I_{m}^{(t)}=1}U_m^{(t)}(i)$ (resp. $\sum_{\overline{I}_m^{(t)}=1}U_m^{(t)}(i)$). 
Given that properties-(a) and (b) of good realizations hold, we expect that for a typical $\widetilde{\omega}$, all the $\overline{S}(i)$'s will have variance no more than $\gamma^2c\log n$, while at least half of the $S(i)$'s will have variance at least $10^{-3}\zeta c\log n$. By the central limit theorem, it is natural to think of $S(i)$ and $\overline{S}(i)$ as Gaussian variables. In this way, we expect that w.h.p. under $\widetilde{\mathbb P}$,
\[
\max_{1\le i\le n^\gamma}|\overline{S}(i)|\le \sqrt{3\gamma^2c\log n\cdot \log n^\gamma}<\gamma\sqrt c\log n\,.\quad
\]
Additionally, if we assume nice independence over $S(i),1\le i\le n^\gamma$, then w.h.p. under $\widetilde{\mathbb P}$,
\[
\max_{1\le i\le n^\gamma}|S(i)|\ge \sqrt{10^{-3}\zeta c\log n\cdot \log n^\gamma}>2\gamma\sqrt{c}\log n\,.
\]
The next two propositions make the above heuristic rigorous. 
%The next proposition gives a (w.h.p.) uniform upper-bound on $\overline{S}(i),1\le i\le n^\gamma$. 

\begin{proposition}\label{prop-ovelrine-S(i)}
    Conditioned on any good realization of $\mathcal F_{-\mathcal I}$, it holds w.h.p. that for any $1\le i\le n^\gamma$, $|\overline{S}(i)|\le \gamma\sqrt c\log n$.
\end{proposition}

\begin{proposition}\label{prop-S(i)}
    Conditioned on any good realization of $\mathcal F_{-\mathcal I}$, it holds w.h.p. that there exists $1\le i\le n^\gamma$ such that $S(i)\ge 2\gamma\sqrt c\log n$. 
\end{proposition}

As hinted above, Proposition~\ref{prop-ovelrine-S(i)} can be proved by applying martingale concentration inequalities and then taking the union bound; the details can be found in Appendix~\ref{appendix-overline-S(i)-small}. The proof of Proposition~\ref{prop-S(i)} is less straightforward, so we dedicate the remainder of this section to it. The challenging part here is that, beyond the marginal distributions of each $S(i)$, Proposition~\ref{prop-S(i)} further requires certain weak correlation properties within $S(i)$. On a more technical level, we use the second-moment method to prove Proposition~\ref{prop-S(i)}, and we will see that \eqref{eq-good-realization-3} in Property-(c) of good realizations plays a key role in this decorrelation phenomenon. 

\begin{comment}
The proof of Proposition~\ref{prop-S(i)} contains several technical steps, which we discuss in order. Fix a good realization of $\mathcal F_{-\mathcal I}$. It is easy to check that deterministically,
\[
\left|\sum_{m,t}I_m^{(t)}U_m^{(t)}(i)-\sum_{m,t}I_m^{(t)}\frac{(x_{m,b_i}^{(t)}-x_{m,c_i}^{(t)})^2}{d_{a_i}^2-|S_{m,a_i,-\mathcal I}|^2}\right|\le \widetilde{O}\Big(\frac{n^\gamma}{(np)^2}\Big)=o(1)\,.
\]
Therefore, provided that $\gamma$ is small enough, then under any good realization of $\mathcal F_{-\mathcal I}$, there are at least half of the indices $1\le i\le n^\gamma$ such that 
\[
\sigma_i^2:=\sum_{m,t}I_m^{(t)}\frac{(x_{m,b_i}^{(t)}-x_{m,c_i}^{(t)})^2}{d_{a_i}^2-|S_{m,a_i,-\mathcal I}|^2}\ge 10^{-3}\zeta c\log n-o(1)\,.
\]
We denote $I$ as the set of such indices, then $|I|\ge n^\gamma/2$, then deterministically it holds
\[
\widetilde{\sigma}_i^2:=\sum_{m,t}I_m^{(t)}U_{m}^{(t)}(i)\ge 10^{-3}\zeta c\log n-o(1)\,,\forall i\in I\,.
\]
\end{comment}

Fix a good realization $\omega$ of $\mathcal F_{-\mathcal I}$. Let $I$ be the subset of $\{1,\cdots,n^\gamma\}$ such that \eqref{eq-good-realization-2} holds. In what follows, we denote by $\overline{\mathbb{P}}$ the conditional probability measure given the realization $\omega$ and abbreviate $\chi:=2\gamma\sqrt c\log n$. 
From the second-moment method, we conclude
\begin{align}\label{eq-second-moment-method}
\overline{\mathbb{P}}\Big[\max_{1\le i\le n^\gamma}S(i)\ge \chi\Big]\ge \overline{\mathbb{P}}\Big[\max_{ i\in I}S(i)\ge \chi\Big]\ge \frac{\sum_{i,j\in I}\overline{\mathbb{P}}[S(i)\ge \chi,S(j)\ge \chi]}{\sum_{i,j\in I}\overline{\mathbb{P}}[S(i)\ge \chi]\overline{\mathbb{P}}[S(j)\ge \chi]}\,.
\end{align}
We show that the final expression is $1-o(1)$. 

For $i\in I$, define 
\begin{equation}\label{eq-sigma_i^2}
\sigma_i^2:=\sum_{I_m^{(t)}=1}\frac{(x_{m,b_i}^{(t)}-x_{m,c_i}^{(t)})^2}{d_{a_i}^2-|S_{m,a_i,-\mathcal I}|^2}\,,
\end{equation}
which is measurable w.r.t. $\mathcal F_{-\mathcal I}$. It is straightforward to check that, deterministically,
\begin{equation}\label{eq-deterministic-relation}
\begin{aligned}
&   \ \left|\sum_{I_m^{(t)}=1}U_m^{(t)}(i)-\sum_{I_m^{(t)}=1}\frac{(x_{m,b_i}^{(t)}-x_{m,c_i}^{(t)})^2}{d_{a_i}^2-|S_{m,a_i,-\mathcal I}|^2}\right|\\
=&\ \left|\sum_{I_m^{(t)}=1}(x_{m,b_i}^{(t)}-x_{m,c_i}^{(t)})^2\cdot \frac{(S_{m,a_i}^{(t)}+S_{m,a_i,-\mathcal I}^{(t)})\cdot\sum_{j\in \operatorname{N}_i\cap \mathcal I}x_{m,j}^{(t)}}{(d_{a_i}^2-|S_{m,a_i}^{(t)}|^2)(d_{a_i}^2-|S_{m,a_i,-\mathcal I}|^2)}\right|\\
\le&\ \widetilde{O}(n^2p^2)\times \widetilde{O}\left(\frac{np\cdot n^\gamma}{(np)^4\cdot n^{-\delta/10}}\right)= \widetilde{O}\Big(\frac{n^{\gamma+\delta/10}}{np}\Big)=o(1)\,.
\end{aligned}
\end{equation}
We conclude from \eqref{eq-good-realization-2} that $\sigma_i^2$ satisfies $\sigma_i^2\ge 10^{-3}\zeta c\log n-o(1)$ for all $i\in I$. Moreover, \eqref{eq-deterministic-relation} indicates that, given any realization $\widetilde{\omega}$ of $\mathcal F_{-\{a_i\}}$ that is compatible with $\omega$, the variance of $S(i)$ is roughly $\sigma_i^2$. Hence, $S(i)$ should behave like a centered Gaussian with variance $\sigma_i^2$. More precisely, we have the following lemma.

%The next lemma suggests that $S(i)$ is very close to a Gaussian variable with mean $0$ and variance $\sigma_i^2$.

\begin{lemma}\label{lem-quantitative-estimate}
Given a good realization $\omega$ of $\mathcal F_{-\mathcal I}$, let $I$ and $\sigma_i^2$ be defined as above. Then, for any $i\in I$ and any compatible realization $\widetilde{\omega}$ of $\mathcal F_{-\{a_i\}}$, it holds for any $x\in [\chi-1,\chi+1]$ that
\[
\widetilde{\mathbb{P}}[S(i)\ge x]=(1+o(1))\Phi(x/\sigma_i)\,,
\]
where $\widetilde{\mathbb P}=\mathbb{P}[\cdot\mid \widetilde{\omega}]$ and $\Phi$ is the Gaussian tail function. Thus, $\overline{\mathbb P}[S(i)\ge x]=(1+o(1))\Phi(x/\sigma_i)$.  
\end{lemma}
The proof of this lemma relies on a quantitative central limit theorem for sums of independent random variables; see Appendix~\ref{appendix-lem-quantitative-estimate} for details. 

Given Lemma~\ref{lem-quantitative-estimate}, we obtain that for $\sigma_i^2\ge 10^{-3}\zeta c\log n-o(1)$,
\[
\Phi(\chi/\sigma_i)\ge \Phi\big(\chi/\sqrt{10^{-3}\zeta c\log n-o(1)}\big)\ge n^{-500\gamma^2/\zeta-o(1)}\ge \max\{n^{-\gamma/2},n^{-\delta/10}\}\,,
\]
where the second inequality follows from a standard Gaussian tail estimate, and the last inequality is due to our choice of $\gamma$. 
Thus, we have 
$$\overline{\mathbb P}[S(i)\ge \chi]=\overline{\mathbb E}\Big[\mathbb{P}\big[S(i)\ge \chi\mid \mathcal F_{-\{a_i\}}\big]\Big]=(1+o(1))\Phi(\chi/\sigma_i)\gg n^{-\gamma}\,.$$
Therefore, the denominator in \eqref{eq-second-moment-method} goes to $\infty$ as $n\to \infty$ (as $|I|\ge n^\gamma/2$ by Property-(b) of good realizations). Consequently, to show that the right hand side of \eqref{eq-second-moment-method} is $1-o(1)$, it suffices to show that for any $i\neq j\in I$, 
\begin{equation}\label{eq-two-point-estimate}
\overline{\mathbb P}[S(i)\ge \chi,S(j)\ge \chi]=(1+o(1))\overline{\mathbb P}[S(i)\ge \chi]\overline{\mathbb P}S(j)\ge \chi]. 
\end{equation}

To this end, we fix $i\neq j\in I$ and seek a decoupling of $S(i)$ and $S(j)$. We define
\[
\widehat{S}(i):=\sum_{I_{m}^{(t)}=1}\frac{x_{m,a_i}^{(t+1)}(x_{m,b_i}^{(t)}-x_{m,c_i}^{(t)})}{d_{a_i}+x_{m,a_i}^{(t+1)}\cdot S_{m,a_i,-a_j}}\,,\quad\widehat{S}(j):=\sum_{I_{m}^{(t)}=1}\frac{x_{m,a_j}^{(t+1)}(x_{m,b_j}^{(t)}-x_{m,c_j}^{(t)})}{d_{a_j}+x_{m,a_j}^{(t+1)}\cdot S_{m,a_j,-a_i}}\,
\]
%\OS{added := in both places} \SH{looks good}
and write $\Delta(i):=S(i)-\widehat{S}(i)$ and $\Delta(j):=S(j)-\widehat{S}(j)$. Intuitively, since $\Delta(i)$ and $\Delta(j)$ are small by Property-(c) of good realizations,  $\widehat{S}(i)$ and $\widehat{S}(j)$ are close to $S(i)$ and $S(j)$, respectively. Moreover, under $\mathcal F_{-\{a_i,a_j\}}$, $\widehat{S}(i)$ and $\widehat{S}(j)$ are conditionally independent (recall \eqref{eq-F_{-I}}), which establishes \eqref{eq-two-point-estimate}.
 
More precisely, let $\operatorname{e}_n=(\log n)^{-1}$ ($\operatorname{e}_n$ can be any sub-polynomial $o(1)$ term). We have that
\begin{align*}
&\overline{\mathbb P}[S(i)\ge \chi,S(j)\ge \chi]\ge \overline{\mathbb P}[\widehat{S}(i)\ge \chi+\operatorname{e}_n,\widehat{S}(j)\ge \chi+\operatorname{e}_n]-\overline{\mathbb P}[\max\{|\Delta(i)|,|\Delta(j)|\}\ge \operatorname{e}_n]\,, \\
&\overline{\mathbb P}[S(i)\ge \chi,S(j)\ge \chi]\le \overline{\mathbb P}[\widehat{S}(i)\ge \chi-\operatorname{e}_n,\widehat{S}(j)\ge \chi-\operatorname{e}_n]+\overline{\mathbb P}[\max\{|\Delta(i)|,|\Delta(j)|\}\ge \operatorname{e}_n]\,.
\end{align*}
Note that
\begin{align*}
\Delta(i)=&\ \sum_{I_{m}^{(t)}=1}\frac{x_{m,a_j}^{(t)}(x_{m,b_i}^{(t)}-x_{m,c_i}^{(t)})}{(d_{a_i}+x_{m,a_i}^{(t+1)}\cdot S_{m,a_i}^{(t)})(d_{a_i}+x_{m,a_i}^{(t+1)}\cdot S_{m,a_i,-a_j}^{(t)})}\\
=&\ \sum_{m,t}I_m^t\frac{x_{m,a_j}^{(t)}(x_{m,b_i}^{(t)}-x_{m,c_i}^{(t)})}{(d_{a_i}+x_{m,a_i}^{(t+1)}\cdot S_{m,a_i,-\mathcal I}^{(t)})^2}+\widetilde{O}\left((np)^2\times \frac{n^{\gamma+3\delta/20}}{(np)^3}\right)\,.
\end{align*}
Since the last error term is $\widetilde{O}(n^{-\delta/2})$,
it follows from \eqref{eq-good-realization-3} that 
\begin{equation}\label{eq-error}
\overline{\mathbb P}[\max\{|\Delta(i)|,|\Delta(j)|\}\ge \operatorname{e}_n]\le \widetilde{O}(n^{-\delta/5})\,.
\end{equation}

Additionally, we have
\begin{align}
\nonumber&\ \overline{\mathbb{P}}[\widehat{S}(i)\ge \chi\pm \operatorname{e}_n,\widehat{S}(j)\ge \chi\pm \operatorname{e}_n]\\
\nonumber=&\ \overline{\mathbb{E}}\left[\mathbb{P}\big[\widehat{S}(i)\ge \chi\pm \operatorname{e}_n,\widehat{S}(j)\ge \chi\pm \operatorname{e}_n\mid \mathcal F_{-\{a_i,a_j\}}\big]\right]\\
\nonumber=&\ \overline{\mathbb E}\left[\mathbb{P}\big[\widehat{S}(i)\ge \chi\pm \operatorname{e}_n\mid \mathcal F_{-\{a_i,a_j\}}\big]\mathbb{P}[\widehat{S}(j)\ge \chi\pm \operatorname{e}_n\mid \mathcal F_{-\{a_i,a_j\}}]\right]\\
=&\ \overline{\mathbb E}\left[\mathbb{P}\big[S(i)\ge \chi\pm 2\operatorname{e}_n\mid \mathcal F_{-\{a_i,a_j\}}\big]\mathbb{P}[S(j)\ge \chi\pm 2\operatorname{e}_n\mid \mathcal F_{-\{a_i,a_j\}}]\right]+\widetilde{O}(n^{-\delta/10})\,,\label{eq-decoupling}
\end{align}
where the second equality follows from conditional independence, and the last from \eqref{eq-error}. 

By Lemma~\ref{lem-quantitative-estimate}, we have by applying Lemma~\ref{lem-quantitative-estimate}, 
\begin{align*}
 \mathbb{P}[S(i)\ge \chi\pm 2\operatorname{e}_n\mid \mathcal F_{-\{a_i,a_j\}}]=&\ (1+o(1))\Phi\big((\chi\pm2\operatorname{e}_n)/\sigma_i\big)\\(\text{using }\sqrt{2\pi}\Phi(x)\sim x^{-1}e^{-x^2/2} )=&\ (1+o(1))\Phi(\chi/\sigma_i)\\
=&\ (1+o(1))\overline{P}[S(i)\ge \chi]\,,
\end{align*}
and a similar estimate holds for $i$ replaced with $j$.
Combining this with \eqref{eq-decoupling} (recall that $\Phi(\chi/\sigma_i)\gg n^{-\delta/20}$ due to our choice of $\gamma$), we conclude \eqref{eq-two-point-estimate}. This shows that conditioned on any good realization, it holds w.h.p. that $\max_{1\le i\le n^\gamma}\mathcal S_{a_i,b_i,c_i}\ge 2\gamma\sqrt c\log n$. Combining with Proposition~\ref{prop-good-realization-typical} and Proposition~\ref{prop-ovelrine-S(i)}, we conclude that Item-(iii) holds with probability at least $0.95-o(1)$, completing the proof.

\section*{Simulation}
All simulations in Section~\ref{subsec-algo-exp} were performed using MATLAB 2024b (Mathworks, Natick, MA).

\appendix
\section{On admissibility}\label{appendix-admissible}

In this section, we show that a graph $G$ drawn from $\mathcal G(n,p)$ is admissible w.h.p.. For the reader's convenience, we now recall the definition of admissibility.

\begin{definition}\label{def-admissible*}
We say a directed graph $G^*$ on $[n]$ is admissible if the following hold:\\
 \noindent (i) If $d$ is either the in-degree or out-degree of any vertex in $G^*$, then
\begin{equation}\label{eq-degree*}
|d-np|\le \sqrt{10np\log n}\,.
\end{equation}
Additionally, for any $i\neq j\in [n]$, if $np^2\ge (\log n)^4$, then 
\begin{equation}\label{eq-overlap-1*}
\big||\operatorname{N}_i\cap \operatorname{N}_j|-np^2\big|\le \sqrt{10np^2\log n}\,,
\end{equation}
and if $np^2\le (\log n)^4$, then 
\begin{equation}\label{eq-overlap-2*}
|\operatorname{N}_i\cap \operatorname{N}_j|\le 4(\log n)^4\,.
\end{equation}
 \noindent (ii) For any $i\in [n]$, let $\{X_{t,i}\}_{t\ge 0}$ be a random walk on $G^*$ starting at $i$, and for $t\ge 1$, let $\pi_{t,i}$ denote the law of $X_{t,i}$ (which is a probability distribution on $[n]$). Then, as $t\to\infty$, $\pi_{t,i}$ converges to the (unique) stationary distribution $\pi$ on $[n]$, which satisfies $\pi(i)=1/n+o(1/n)$ for each $i\in [n]$. Moreover, for any integer $k\ge 1$, 
there exists a constant $c=c(k,\delta)$ such that 
\[
\operatorname{TV}(\pi_{t,i},\pi)\leq n^{-k}\,,\ \forall\ t\geq c_{k,\delta},i\in [n]\,.
\] 
 \noindent (iii) There exists $C^*>0$ such that the consensus time of the voter model on $G^*$,  defined as 
 \begin{equation}\label{eq-T-coal*}
T_{\operatorname{cons}}:=\min\{t\ge 0:x_i^{(t)}=x_j^{(t)},\forall i,j\in [n]\}\,,
\end{equation}
 satisfies $\mathbb{E}[T_{\operatorname{cons}}]\le C^*n$ (where the randomness is over the voter model dynamics).
 
 \noindent (iv) For all partitions $[n]=X\sqcup Y$ of the vertex set, the number of triples $i,j,j'$ of elements of $[n]$ such that $j\in X$, $j'\in Y$, and both $(i,j)$ and $(i,j')$ belong to $\vec{E}(G^*)$ is at least $np^2|X|\cdot|Y|/10^{12}$.
\end{definition}

Item-(i) is essentially a standard lemma regarding the regularity of random graphs.

\begin{lemma}\label{lem:edges-between-sets}
For $G^*\sim\mathcal G(n,p)$, w.h.p. we have that any vertex in $G^*$ has in-degree and out-degrees $d$ satisfying
\begin{equation}\label{eq-degree**}
|d-np|\le \sqrt{10np\log n}\,.
\end{equation}
Additionally, for any $i\neq j\in [n]$, if $np^2\ge (\log n)^4$, then 
\begin{equation}\label{eq-overlap-1**}
\big||\operatorname{N}_i\cap \operatorname{N}_j|-np^2\big|\le \sqrt{10np^2\log n}\,,
\end{equation}
and if $np^2\le (\log n)^4$, then 
\begin{equation}\label{eq-overlap-2**}
|\operatorname{N}_i\cap \operatorname{N}_j|\le 4(\log n)^4\,.
\end{equation}
\end{lemma}
\begin{proof}
Note that for $G^*\sim \mathcal G(n,p)$, any out-degree or in-degree $d$ satisfies $d\sim \mathbf{B}(n-1,p)$, and for any distinct $i, j\in [n]$, $|\operatorname{N}_i\cap\operatorname{N}_j|\sim \mathbf{B}(n-2,p^2)$. From Chernoff's bound (see Lemma~\ref{lem-Chernoff-bound})
%\OS{Maybe we could make it so that C.1 is the first section of the appendix instead? I think it's fine either way}\HD{The concentration of weighted sum is only used in Appendix C, I would suggest leaving the first two appendices as clean as possible} \OS{sounds good}
we have each of \eqref{eq-degree**}, \eqref{eq-overlap-1**}, and \eqref{eq-overlap-2**} fails with probability at most $n^{-3}$. The desired results then follow from the union bound.
\end{proof}

We next prove the existence of the stationary measure for the random walk on $G^*$.
 
\begin{lemma}\label{thm:stationary_exists}
    If $G^*\sim \mathcal G(n,p)$, then w.h.p. the random walk on $G^*$ is ergodic and thus possesses a unique stationary measure, which we denote by $\pi$.
\end{lemma}

\begin{proof}
First, we claim that w.h.p. $G^*$ is strongly connected. If this is not the case, then there exists a non-trivial partition of $[n]$ into sets $A$ and $B$ such that there are no edges from $A$ to $B$. For any partition of $[n]$ into parts $A$ and $B$, the probability that there is no edge from $A$ to $B$ is $(1-p)^{|A|\cdot|B|}$. By taking the union bound, we see the probability that $G^*$ is not strongly connected is at most
\[
\sum_{k=1}^{\lfloor n/2\rfloor }\binom{n}{k}(1-p)^{k(n-k)}\le \sum_{k=1}^{\lfloor n/2\rfloor }\exp(k\log n-pkn/2)=o(1)\,,
\]
verifying the claim. Additionally, it is straightforward to see that w.h.p. $G^*$ contains a directed cycle of length $3$ as well as a directed cycle of length $4$. The result follows by combining these two facts. 
\end{proof}

For a positive integer $\ell$ and $i,j\in [n]$, let $\mathscr P_{i,j;\ell}$ denote the set of directed (not necessarily self-avoiding) paths $\mathcal P=(i_0,i_1,\cdots,i_\ell)$ in $G^*$ from $i_0=i$ to $i_\ell=j$ of length $\ell$. In order to bound the mixing time of random walks on $G^*$, we will make use of a concentration result for $\mathscr P_{i,j;\ell}$. This result, which was first obtained in~\cite[Section~6.1]{cooper2011stationary} in order to study stationary distributions of random directed graphs, is a simple consequence of a concentration inequality from~\cite{kim2000concentration}.

\begin{lemma}\label{lem-path-count}
    Let $\ell\geq\lceil5/\delta\rceil$ be a fixed positive integer and suppose that $G^*\sim\mathcal G(n,p)$ (recall that $np\geq n^{\delta}$). Then, for any two vertices $i,j\in[n]$ we have that \[\Pb\left[|\mathscr P_{i,j,\ell}-n^{\ell-1}p^\ell|\geq\widetilde{O}_\ell(n^{\ell-1}p^\ell/n^{\delta/2})\right]\leq O(n^{-3})\,.\] (The subscript next to the big-$O$ indicates that the hidden constant depends on $\ell$.) 
\end{lemma}

Recall that for $i\in[n]$, we let $\{X_{t,i}\}_{t\ge 0}$ denote the random walk on $G^*$ starting at $i$, and for each $t\ge 1$, we let $\pi_{t,i}$ denote the law of $X_{t,i}$.
 
\begin{theorem}\label{thm:mixing}
     Given any two positive parameters $\delta$ and $k$, there exists a constant $c=c(k,\delta)$ such that if $G^*\sim\mathcal G(n,p)$ and $np\ge n^{\delta}$, then w.h.p.
     \begin{equation}\label{eq-mixing}
     \operatorname{TV}(\pi_{t,i},\pi)\leq n^{-k}\,,\quad \forall\ i\in [n]\,,t\geq c_{k,\delta}\,,
    \end{equation}  
    where $\pi$ denotes the stationary measure provided by the previous lemma.
\end{theorem}

\begin{proof}
    Pick an integer $\ell=O(1)$ such that $\ell\geq\lceil5/\delta\rceil$. By Lemma~\ref{lem-path-count} and the union bound, w.h.p. $G^*$ satisfies
    \begin{equation}\label{eq-path-concentration}
    |\mathscr P_{i,j;\ell}-n^{\ell-1}p^\ell|\le \widetilde{O}_\ell(n^{\ell-1}p^\ell/n^{\delta/2})\,
    \end{equation}
    for all $i,j\in[n]$.
    For each path $\mathcal P=(i_0,\cdots,i_\ell)\in \mathscr P_{i,j;\ell}$ define $\Xi(\mathcal P):=\prod_{s=0}^{\ell-1}d_{i_s}^{-1}$, and let $\Xi(\mathscr P_{i,j;\ell}):=\sum_{\mathcal P\in \mathscr P_{i,j;\ell}}\Xi(\mathcal P)$.
    Recall that, w.h.p., all vertices of $G$ have out-degree $d_i$ lying in the interval $[np-\sqrt {10np\log n},np+\sqrt {10np\log n}]$. This implies that for all $\mathcal P\in\mathscr P_{i,j;\ell}$,
    \begin{equation}\label{eq-Xi(P)-concentration}
    \Xi(\mathcal P)\in[(np+\sqrt {10np\log n})^{-\ell},(np-\sqrt {10np\log n})^{-\ell}]\,.
    \end{equation}
    Combining \eqref{eq-Xi(P)-concentration} with \eqref{eq-path-concentration} yields that
    \[
    \Xi(\mathscr P_{i,j;\ell})=\frac{1}{n}+\widetilde{O}_\ell\left(\frac{1}{n^{1+\delta/2}}\right)\,.
    \]
    A moment of thought reveals that $\pi_{i,\ell}(j)$ is precisely $\Xi(\mathcal P_{i,j;\ell})$. Hence, for large enough $n$,
    \[
    \operatorname{TV}(\pi_{i,\ell},\operatorname{Uni})\le \widetilde{O}_\ell(n^{-\delta/2})\le \frac{1}{2}n^{-\delta/3}\,,
    \]
    where $\operatorname{Uni}$ is the uniform distribution on $[n]$. By the triangle inequality, we have
    \[
    \operatorname{TV}(\pi_{i,\ell},\pi_{j,\ell})\le n^{-\delta/3}\,,\quad\forall i,j\in [n]\,.
    \]
    Therefore, by a standard coupling argument, for any $k\in \mathbb{N}$ and $c=c(k,\delta)=\ell\cdot \lceil 3k/\delta\rceil$, it holds that for any $t\ge c$, 
    \[
    \operatorname{TV}(\pi_{i,c},\pi_{j,c})\le n^{-\delta/3 \cdot \lfloor t/\ell\rfloor}\le n^{-\delta/3\cdot 3k/\delta}=n^{-k}\,,\quad \forall i,j\in [n]\,.
    \]
    This implies \eqref{eq-mixing} and completes the proof.
\end{proof}

As a byproduct of the above proof, we get the following lemma.

\begin{lemma}\label{thm:stationary}
    If $G^*\sim\mathcal G(n,p)$, then w.h.p. the stationary measure $\pi$ satisfies $\pi(i)=1/n+o(1/n)$. In particular, $\operatorname{TV}(\operatorname{Uni},\pi)=o(1),$ where $\operatorname{Uni}$ denotes the uniform distribution on $[n]$.
\end{lemma}

Combining Lemma~\ref{thm:stationary_exists}, Theorem~\ref{thm:mixing}, and Lemma~\ref{thm:stationary}, we conclude that Property-(ii) of admissibility holds w.h.p. as well.

Additionally, as a consequence of Theorem~\ref{thm:mixing} and Lemma~\ref{thm:stationary}, we derive a lower bound on $p_{i,j}^{(t)}$ (defined as in \eqref{eq-p_ij^t}) for any $i,j\in [n]$ and $t\ge 1$. Let $t_*\triangleq c(2,\delta)$ and $t_s\triangleq st_*$ for $s=1,2,\cdots$. For two independent random walks $\{X_{t}\}_{t\ge 0}$ and $\{Y_{t}\}_{t\ge 0}$ starting at $i$ and $j$, respectively, we have
\begin{align*}
    p_{i,j}^{(t)}=\mathbb{P}[\exists t'\le t:X_{t'}=Y_{t'}]\ge \mathbb{P}[\exists s\le \lfloor t/t_*\rfloor,X_{t_s}=Y_{t_s}]\ge 1-\left(1-\frac{1}{2n}\right)^{\lfloor t/t_*\rfloor}\,.
\end{align*}
Here, the last inequality follows from the fact that for any $s\ge 1$, the event $X_{t_s}=Y_{t_s}$ holds with conditional probability at least $1-\frac{1}{2n}$, given any realization of $X_{t_{s-1}}$ and $Y_{t_{s-1}}$.
Thus, for any $t\ge n$, $\min_{i,j}p_{i,j}^{(t)}\ge 1-\exp(-\frac{t}{2t_*n})$. From this tail estimate, we conclude that $m(G^*)$, the expected meeting time of two independent random walks starting at two uniformly and independently chosen vertices of $G^*$, is upper-bounded by a constant multiple of $n$. Specifically, let $C^*=C^*(\delta)>0$ be a constant such that $m(G^*)\le (C^*-1)n/4$ holds w.h.p. for $G^*\sim \mathcal G(n,p)$. 
%\OS{I have added the /3 so that we can use the same $C^*$ everywhere, even if it looks a bit artificial} \HD{Thanks!}
    
For Property-(iii), we utilize a result about the expected coalescence times of coalescing random walks on fast-mixing graphs. Recall the backward random walk paths $\mathcal P_i^t$ for $i\in [n]$ and $t\ge 1$. We define
\begin{equation}\label{eq-T-coal}
T_{\operatorname{coal}}:=\min\{t\ge 1:\mathcal P_i^t\text{ coalesces with }\mathcal P_j^t,\forall i,j\in [n]\}\,.
\end{equation}
\begin{comment}
Another consequence of the above observation is that the consensus time of the dynamics, defined as 
\begin{equation}\label{eq-T-stab}
T_{\operatorname{cons}}:=\min\{t\ge 0:x_i^{(t)}=x_j^{(t)},\forall i,j\in [n]\}\,,
\end{equation}
is upper-bounded by the coalescing time
\begin{equation}\label{eq-T-coal}
T_{\operatorname{coal}}:=\min\{t\ge 1:\mathcal P_i^t\text{ coalesces with }\mathcal P_j^t,\forall i,j\in [n]\}\,.
\end{equation}
As an quick application of this inequality, by applying the union bound, we can easily deduce the following tail bound on $T_{\operatorname{cons}}$: for any $t\ge 0$, $
\mathbb{P}[T_{\operatorname{cons}}\ge t]$ is upper-bounded by 
\begin{equation}\label{eq-max-prob}
n^2\cdot \max_{i\neq j\in [n]}\mathbb{P}[\text{two independent random walks on $G^*$ starting at }i,j\text{ do not meet in }t\text{ steps}]\,.
\end{equation}

We shall see later that the mixing time on $G^*$ is typically $O(1)$ and hence \eqref{eq-max-prob} decays like $\exp(-\Theta(t/n))$ for large $t$. The above tail bound then suggests that $T_{\operatorname{cons}}=O(n\log n)$ w.h.p.. Strikingly, a much stronger result in \cite{Oli13} reveals that $T_{\operatorname{cons}}$ is typically $\Theta(n)$ (we will return to this point in Section~\ref{subsec-admissible-graph}). Nevertheless, the above tail estimate will still be useful in later proof. 
\end{comment}
The following result essentially comes from~\cite{Oli13}. Recall the definition of $m(G^*)$ as above. 

\begin{theorem}\label{thm:coal}
For any graph $G^*$ on $n$ vertices with mixing time $O(1)$ (i.e., there exists some $t=O(1)$ with $\operatorname{TV}(\pi_{t,i},\pi)\le \frac14,\forall i\in [n]$), 
    the coalescence time $T_{\operatorname{coal}}$ defined in \eqref{eq-T-coal*} satisfies 
    \[
    \mathbb{E}[T_{\operatorname{coal}}]\le n+4m(G^*)\le C^*n\,.
    \]
    %where $Z_2,Z_3,\cdots$ are independent exponential random variables with $\mathbb{E}[Z_i]=$
\end{theorem}
\begin{proof}
The result follows in the same spirit as \cite[Theorem 1.2]{Oli13}. However, since that result applies only to continuous-time Markov chains, we sketch the proof here for completeness.

Consider coalescing random walks starting from every vertex of $G^*$. For $k \le n$, let $T_k$ denote the minimal time $t$ such that at time $t$, there are at most $k$ non-coalesced random walks. The crux of the argument is to show that for any $2 \le k \le n$,
\begin{equation}\label{eq-crux}
\mathbb{E}[T_{k-1} - T_k] \le 1 + \frac{2m(G)}{\binom{k}{2}}\,.
\end{equation}
Using Theorem~\ref{thm:mixing} and Lemma~\ref{thm:stationary}, \eqref{eq-crux} can be established by arguments similar to those in \cite[Lemma 3.2]{Oli13}. Given \eqref{eq-crux}, it follows that
\[
\mathbb{E}[T_{\operatorname{coal}}] = \mathbb{E}\Bigg[\sum_{k=2}^n (T_{k-1} - T_k)\Bigg] \le n + 4m(G^*) \le C^* n,
\]
thus completing the proof.
\end{proof}

%\HD{This theorem is only stated for non-directed graph. I'm not sure if this would be problem. }\OS{Right... I'll take a look at the paper. Update: In the version of the paper that I'm looking at (the one from annals of probability) this seems to be stated for general markov chains in Theorem 1.2., so it should apply to our setting directly}\HD{That's great, thanks for clarifying!}

By the duality to coalescing random walks, we observe that $T_{\operatorname{cons}}\le T_{\operatorname{coal}}$. 
Combining Theorems~\ref{thm:mixing} and \ref{thm:coal}, we obtain that Property-(iii) of admissibility holds w.h.p. as well. 

Finally we address Property-(iv) of admissibility. For a subset $X\subseteq[n]$ and a positive integer $k$, let $\operatorname{N}_{\operatorname{out},k}(X)\subseteq[n]$ denote the set of vertices that can be reached from at least $k$ different vertices of $X$ by traversing a single directed edge. Similarly, define $\operatorname{N}_{\operatorname{in},k}(X)\subseteq[n]$ as the set of vertices from which at least $k$ elements of $X$ can be reached by traversing a single directed edge. For notational simplicity, we write  $\operatorname{N}_{\operatorname{out}}(X)$ and $\operatorname{N}_{\operatorname{in}}(X)$ in place of $\operatorname{N}_{\operatorname{out},1}(X)$ and $\operatorname{N}_{\operatorname{in},1}(X)$, respectively. To show that Property-(iv) holds w.h.p., we require the following lemma, which essentially states that every set of vertices expands very well in $G^*$. 

\begin{lemma}\label{lem-expansion}
    Let $G^*\sim\mathcal G(n,p)$. Then, w.h.p., for every set $X\subseteq [n]$ it holds that \[|\operatorname{N}_{\operatorname{out}}(X)|,|\operatorname{N}_{\operatorname{in}}(X)|\geq \min\{np\cdot|X|/20,n/20\}\,\] and \[|\operatorname{N}_{\operatorname{out},k}(X)|,|\operatorname{N}_{\operatorname{in},k}(X)|\geq 2n/3\,,\] where $k=\lfloor p\min\{|X|,n/1000\}/100\rfloor$.
\end{lemma}

\begin{proof}
    Consider a set $X\subseteq [n]$. Without loss of generality, we may assume that $|X|\leq p^{-1}$. For each vertex $i\in[n]$, let $X_i$ denote the indicator random variable which takes the value $1$ if $i\in\operatorname{N}_{\operatorname{out}}(X)$ and $0$ otherwise. Observe that $|\operatorname{N}_{\operatorname{out}}(X)|=\sum_{i\in[n]\backslash X}X_i$, where the sum consists of i.i.d. Bernoulli random variables. Furthermore, one can verify that the mean of each of these Bernoulli variables is at least $p|X|/10$, implying that the expected value of the sum is no less than $p|X|(n-|X|)/10\geq p|X|n/15$ (for sufficiently large $n$). The claim about $|\operatorname{N}_{\operatorname{out}}(X)|$ now follows by using Chernoff's bound (see Lemma~\ref{lem-Chernoff-bound}) and then union bounding over all sets $X$ of size at most $p^{-1}$ (it suffices to use the fact that for every positive integer $m$ there are at most $n^m$ subsets of $[n]$ of size $m$). We can repeat the same argument to lower-bound $|\operatorname{N}_{\operatorname{in}}(X)|$.

    The proof of the second part of the statement is very similar. Let $X\subseteq [n]$ and set $k=\lfloor p\min\{|X|,n/1000\}/100\rfloor$. We can restrict our attention to the case where $|X|\leq n/100$. For $i\in[n]$, let $X_{i,k}$ be the random variable that takes the value $1$ if $i\in\operatorname{N}_{\operatorname{out},k}$ and $0$ otherwise. Notice that $|\operatorname{N}_{\operatorname{out},k}(X)|=\sum_{i\in[n]\backslash X}X_{i,k}$ where the summands are i.i.d. Bernoulli random variables. By the choice of $k$, each of these Bernoulli random variables has mean at least $9/10$, leading to $\mathbb{E}|\operatorname{N}_{\operatorname{out},k}(X)| \geq9n/10-|X|\geq89n/100$. Applying Chernoff's bound again and taking a union bound over all sets $X$ of size at most $n/100$, we obtain the desired result. For this last step, we use the estimate \[\sum_{i=1}^{n/1000}\binom{n}{i}\leq \frac{n}{1000}\binom{n}{n/1000}<\frac{n}{1000}\frac{n^n}{(n/1000)^{n/1000}(999n/1000)^{(999n/1000)}}<1.007^n\,,\] which holds for all sufficiently large $n$. The argument for $\operatorname{N}_{\operatorname{in},k}$ is identical.

    We conclude that, w.h.p., every subset $X\subseteq[n]$ satisfies both properties, as desired.
\end{proof}

We are now ready to show that Property-(iv) occurs w.h.p. as well.

\begin{theorem}
    Let $G^*\sim\mathcal G(n,p)$. Then, w.h.p., for every partition $[n]=X\sqcup Y$, the number of triples $i,j,j'$ of elements of $[n]$ such that $j, j'\in\operatorname N_i$ and $j\in X,j'\in Y$ is at least $np^2|X|\cdot|Y|/10^{12}.$ 
\end{theorem}

\begin{proof}
    Suppose $G^*$ satisfies the property stated in Lemma~\ref{lem-expansion} for every $X\subseteq [n]$. Consider a partition $[n]=X\sqcup Y$ and assume without loss of generality that $X$ is non-empty and $|X|\leq n/2$. At a high level, the proof consists of applying the above lemma twice: first to $X$, and then to either $\operatorname{N}_{\operatorname{in}}(X)$ or $\operatorname{N}_{\operatorname{in},k}(X)$, depending on the size of $X$. We provide the details below, assuming $n$ is sufficiently large. The argument is split into three cases.
    
    First, suppose that $|X|\leq 100p^{-1}$ and $np^2|X|\leq 2000$. By our assumption about $G^*$, we have $|\operatorname{N}_{\operatorname{in}}(X)|\geq np|X|/20$. Applying to $\operatorname{N}_{\operatorname{in}}(X)$ the bound given by the first part of the lemma, we get \[|\operatorname{N}_{\operatorname{out}}(\operatorname{N}_{\operatorname{in}}(X))|\geq (np)^2|X|/40000\,\] (note that the RHS is at most $n/20$ by our hypothesis). Note now that the number of triples $i,j,j'$ satisfying the requirements in the statement is at least \[|\operatorname N_{\operatorname{out}}(\operatorname{N}_{\operatorname{in}}(X))\backslash X|\geq(np)^2|X|/50000\geq np^2|X|\cdot|Y|/50000\,,\] as desired.

    Next, assume that that $|X|\leq 100p^{-1}$ and $np^2|X|> 2000$. As in the first case, $|\operatorname{N}_{\operatorname{in}}(X)|\geq np|X|/20$. Note that $np|X|/20\geq100p^{-1}$, and so $\lfloor p|\operatorname{N}_{\operatorname{in}}(X)|/100\rfloor\geq p\operatorname{N}_{\operatorname{in}}(X)/200$. Write $k'=\lfloor p\min\{|\operatorname{N}_{\operatorname{in}}(X)|,n/1000\}/100\rfloor$ and observe that $k'\geq p|\operatorname{N}_{\operatorname{in}}(X)|/(2\cdot 10^5)$. Now, applying the second part of the statement of Lemma~\ref{lem-expansion} to $\operatorname{N}_{\operatorname{in}}(X)$, we deduce that \[|\operatorname{N}_{\operatorname{out},k'}(\operatorname{N}_{\operatorname{in}}(X))|\geq 2n/3\implies|\operatorname{N}_{\operatorname{out},np/(2\cdot10^5)}(\operatorname{N}_{\operatorname{in}}(X))\backslash X|\geq n/6\,.\] In other word, there is a subset of $Y$ of size at least $n/6$ such that each of its elements can be reached from at least $np/(2\cdot 10^5)$ distinct elements of $\operatorname{N}_{\operatorname{in}}(X)$ by traversing a single directed edge. By definition, each element of $\operatorname{N}_{\operatorname{in}}(X)$ has an edge pointing towards $X$, so the number of triples $i,j,j'$ with the desired properties is at least \[\frac{n}{6}\cdot\frac{np}{2\cdot 10^5}\geq \frac{n^2p}{10^7}\geq \frac{np^2|X|\cdot|Y|}{10^9}\,.\]

    Lastly, suppose that $|X|> 100p^{-1}$. This implies that $\lfloor p|X|/100\rfloor\geq p|X|/200$. Write $k=\lfloor p\min\{|X|,n/1000\}/100\rfloor$ and note that $k\geq p|X|/(2\cdot10^5)$. By the assumption on $G^*$, we know that $|\operatorname{N}_{\operatorname{in},k}(X)|\geq 2n/3$. Again by the bound from Lemma~\ref{lem-expansion}, $|\operatorname{N}_{\operatorname{out},np/(2\cdot 10^5)}(\operatorname{N}_{\operatorname{in},k}(X))|\geq2n/3$ and thus \[|\operatorname{N}_{\operatorname{out},np/(2\cdot 10^5)}(\operatorname{N}_{\operatorname{in},k}(X))\backslash X|\geq n/6\,.\] This can be restated as saying that there is a subset of $Y$ of size at least $n/6$ each of whose elements can be reached from at least $np/(2\cdot 10^5)$ distinct elements of $\operatorname{N}_{\operatorname{in},k}(X)$ by traversing a single directed edge. Each element of $\operatorname{N}_{\operatorname{in},k}(X)$ has at least $k$ edges pointing towards $X$, so the number of triples with the required properties is at least \[\frac{n}{6}\cdot\frac{np}{2\cdot10^5}\cdot k\geq\frac{1}{12\cdot 10^5}\cdot n^2p\cdot \frac{p|X|}{2\cdot10^5}\geq \frac{1}{10^{12}}np^2|X|\cdot|Y|\,.\]
    
    We have shown that every graph $G^*$ for which every $X\subseteq [n]$ has the properties mentioned in the statement of Lemma~\ref{lem-expansion} satisfies the desired conclusion. Since a graph drawn from $\mathcal G(n,p)$ will be of this kind w.h.p., this finishes the proof.
\end{proof}

This completes the proof that $G^*\sim \mathcal G(n,p)$ is w.h.p. admissible.
%We incorporate the desirable properties of random directed graphs drawn from $\mathcal G(n,p)$ we have so far identified into the following definition of admissible graphs. 

%\begin{definition}\label{def-admissible}\
%We say a directed graph $G^*$ on $[n]$ with edge set $\vec{E}$ is admissible, if $G^*$ satisfies the following properties:\\
% \noindent (i) Any vertex in $G^*$ has inner and outer degree $d$ satisfying \eqref{eq-degree}, and the neighborhood intersection of any $i\neq j\in [n]$ satisfies \eqref{eq-overlap-1} and \eqref{eq-overlap-2}.\\%lying in $[np-\sqrt{10np\log n},np+\sqrt{10np\log n}]$. \\
 %\noindent (ii) The graph $G^*$ has a stationary measure $\pi$. Furthermore, for any $k\in \mathbb{N}$, the random walk on $G$ achieves $n^{-k}$-mixing after $c(k,\delta)$ steps, as defined in Theorem~\ref{thm:mixing} above. \\
 %\noindent (iii) For every $i\in[n]$, $\pi(i)=1/n+o(1/n).$\\
 %\noindent (iv) The coalescence time $T_{\operatorname{coal}}$ of $G^*$ satisfies $\mathbb{E}[T_{\operatorname{coal}}]\le 3C^*n$. 
%\end{definition}

%We have shown that $G^*\sim \mathcal G(n,p)$ is admissible w.h.p.. Henceforth unless specified, we will assume the graph $G^*$ is a deterministic admissible graph.

\section{Deferred proof from Section~\ref{sec-positive}}\label{appendix-positive}
Here we provide the proof of Lemmas~\ref{lem-martingale-concentration}-\ref{lem-effective-observations-lower-bound}.

\begin{proof}[Proof of Lemma~\ref{lem-martingale-concentration}]
	Fix $i,j\in [n]$. It is clear that $|\mathcal M_{m,i\rightarrow j}^{(t_s)}|\le 2,\forall 1\le m\le M, 1\le s\le T_*$. By Azuma's inequality, we have
    \[
	\mathbb{P}\left[|\mathcal M_{i\rightarrow j}|\ge \frac{50c_0MT_*}{np}\right]\le 2\exp\left(-\frac{(50c_0MT_*/np)^2}{8M T_*}\right)< 2\exp\left(-\frac{300c_0^2MT_*}{(np)^2}\right)\,.
	\] 
    Since $$MT_*\ge \frac{M\cdot\min\{T,n\}}{2t^*}\ge \frac{Cn^2p^2\log n}{2t^*}\,,$$
and $t_*=O(1)$ depends only on $\delta$, we have for $C=C(\delta)$ large enough the above probability is at most $2\exp(-10\log n)=2n^{-10}$. The result then follows from the union bound.
\end{proof}

\begin{proof}[Proof of Lemma~\ref{lem-TV-distance}]
Let $\mathcal B$ be the event that two independent random walks starting at $u$ and $j$ meet before or on time $t_*$. We claim that $\mathbb{P}[\mathcal B]\le \frac{20c_0}{np}$. Define 
\begin{equation}\label{eq-p(G^*)}
 p(G^*):=\max_{u\neq j}\frac{|\operatorname{N}_u\cap \operatorname{N}_j|}{d_ud_j}
\end{equation}
 as the maximum probability that two random walks starting at different locations meet at time $1$. From the union bound, we get $\mathbb P[\mathcal B]\le t_*\cdot p(G^*)$. When $p=o(1)$, we obtain from Property-(i) of admissibility that $$p(G^*)=
        \widetilde{O}\left(\frac{1}{(np)^2}\right)=o\left(\frac 1{np}\right)\,\quad\text{if\, }p\le n^{-1/2}(\log n)^2\,,$$
        and
        $$
        p(G^*)\le \frac{2np^2}{(np)^2}\le \frac{2}{n}\,\quad\text{if\, }p\ge n^{-1/2}(\log n)^2\,.
$$  
Hence, we have $p(G^*)=o(\frac{1}{np})$ for any $p=o(1)$. Since $t_*=O(1)$, $t_*\cdot p(G^*)=o\left(\frac 1 {np}\right)\le \frac{4c_0}{np}$. For $p=\Omega(1)$ with $p\le c_0$, we can pick $\delta=0.99$ in assumption \eqref{eq-p-assumption}, and thus $t_*=c(2,\delta)$ can be picked to be no more than $10$ (see the proof of Theorem~\ref{thm:mixing}). This yields $t_*\cdot p(G^*)\le \frac{20}{n}\le \frac{20c_0}{np}$. Hence, in either case, we have $\mathbb P[\mathcal B]\le \frac{20c_0}{np}$, verifying the claim.  

Since $\mu_{u,j}=\pi_{t_*,u}\otimes\pi_{t_*,j}\mid \mathcal B^c$ by definition, we have 
\[
\operatorname{TV}(\mu_{u,j},\pi_{t_*,u}\otimes\pi_{t_*,j})\le 2\mathbb{P}[\mathcal B]\le \frac{40c_0}{np}\,.
\]
Additionally, by Property-(ii) of admissibility and our choice of $t_*$, 
\[
\operatorname{TV}(\pi_{t_*,u},\pi)\le n^{-2},\  \operatorname{TV}(\pi_{t_*,j},\pi)\le n^{-2} \ \Rightarrow \ \operatorname{TV}(\pi_{t_*,u}\otimes \pi_{t_*,j},\pi^{\otimes 2})\le 2n^{-2}\,.
\]
The result now follows from the triangle inequality.
\end{proof}

\begin{proof}[Proof of Lemma~\ref{lem-concentration-p}]
Write $\delta_t=\max_{(u,v),(u',v'),u\neq v,u'\neq v'}|p_{u,v}^{(t)}-p_{u',v'}^{(t)}|$. We first claim that $\delta_t\le t\delta_1$. For any two pairs $(u,v)$ and $(u',v')$, we have
	\[
	p_{u,v}^{(t)}=\frac{1}{d_ud_v}\sum_{i\sim u,j\sim v}p_{i,j}^{(t-1)},\quad p_{u',v'}^{(t)}=\frac{1}{d_{u'}d_{v'}}\sum_{i'\sim u',j'\sim v'}p_{i',j'}^{(t-1)}\,.
	\]
	Note that $p_{i,j}^{(t-1)}=1$ for $i=j$, so $|p_{u,v}^{(t)}-p_{u',v'}^{(t)}|$ is bounded by $\delta_{t-1}$ plus
	\[
	\Bigg|\frac{1}{d_ud_v}\sum_{i\sim u,i\sim v}1-\frac{1}{d_{u'}d_{v'}}\sum_{i'\sim u',i'\sim v'}1\Bigg|\le \delta_1\, ,
	\]
	and the claim follows by induction. 

    Therefore, since $t_*=O(1)$, it remains to show that $\delta_1=o\left(\frac{1}{np}\right)$. Recall the definition of $p(G^*)$ in \eqref{eq-p(G^*)}. Clearly, we have $\delta_1\le p(G^*)$, which we have shown to be $o\left(\frac{1}{np}\right)$ provided that $p=o(1)$. When $p=\Omega(1)$, by \eqref{eq-degree} and \eqref{eq-overlap-1} in Property-(i) of admissibility, we get $\delta_1=\widetilde{O}(n^{-3/2})=o(n^{-1})$. Hence, the result always holds. 
\end{proof}

\begin{proof}[Proof of Lemma~\ref{lem-effective-observations-lower-bound}]
%\OS{Changed some of the indexing in this proof to make it consistent with the notation in the proof of $\mathcal G_{\operatorname{RW}}$. Reminer: double check}
We first show that for each $1\le m\le M$, $\widetilde{T}_m$ stochastically dominates a uniform distribution on $\{0,1,\cdots,L\}$, where $L:=\min\{T,\frac{n}{32}-1\}$.

Fix $1\le m\le M$ and we omit the subscript $m$. Consider the filtration $$\widetilde{\mathcal F}_t:=\sigma\big(x_{i}^{(t')}:i\in [n],t'\leq t\big)\,.$$
Recall that $W^{(t)}=\mathbb{E}_{p\sim \pi}[x_{p}^{(t)}]$. Since $\pi$ is the stationary distribution of the random walk on $G^*$, we have that $\{W^{(t)}\}_{t\ge 0}$ is a martingale with respect to $\widetilde{\mathcal F}_t$. For any $1\le t\le \min\{T,n\}$, we have by definition that
\[
\mathbb{P}[\widetilde{T}\le t]=\mathbb{P}\left[\max_{1\le r\le t}|W^{(r)}|\ge \frac12\right]\,.
\]
By Chebyshev's inequality and Doob's maximal inequality for $p=2$, 
\[
\mathbb{P}\left[\max_{1\le r\le t}|W^{(t)}|\ge \frac12\right]\le 16\mathbb{E}\big[(W^{(t)})^2\big]=16\left(\mathbb{E}\big[(W^{(0)})^2\big]+\sum_{r=0}^{t-1}\mathbb{E}\big[(W^{(r+1)}-W^{(r)})^2\big]\right)\,,
\]
where the equality follows from the martingale property. Observe now that
\[
\mathbb{E}\big[(W^{(0)})^2\big]=\sum_{i\in [n]}\pi(i)^2\le \frac{2}{n}\,,
\]
where the inequality is a consequence of Property-(iii) of admissibility. 
We claim that for every $r\ge 1$, it is the case that
\[
\mathbb{E}\big[(W^{(r+1)}-W^{(r)})^2\big]\le \frac{2}n\,.
\]
To see this, note that after conditioning on $\widetilde{\mathcal F}_{r}$, $W^{(r+1)}=\sum_{i\in [n]}\pi(i)x_{i}^{(r)}$, where $x_i^{(r)}$ ($i\in [n]$) are conditionally independent. Therefore, conditioned on any realization of $\widetilde{\mathcal F}_{r-1}$.
\[
\mathbb{E}\big[(W^{(r+1)}-W^{(r)})^2\mid\widetilde{\mathcal F}_{r-1}\big]=\sum_{i\in [n]}\pi(i)^2\operatorname{Var}[x_i^{(r+1)}\mid \widetilde{\mathcal F}_{r}]\le \sum_{i\in [n]}\pi(i)^2\le \frac{2}{n}\,,
\]
This verifies our claim, and thus we obtain
\[
\mathbb{P}[\widetilde{T}\le t]\le \frac{32(t+1)}{n}\,.
\]

Since $\widetilde{T}\le \min\{T,n\}$ holds deterministically, this proves that $\widetilde{T}$ stochastically dominates a random variable $V$ whose distribution is given by
\[
\mathbb{P}[V=t]=\frac{32}{n}\,,\forall\ 0\le t\le \min\left\{T-1,\frac n{32}-1\right\}\,,
\]
and $\mathbb{P}[V=T]=1-\frac{32T}{n}$ for $T<\frac{n}{32}$. 
The random variable $V$ further dominates the uniform distribution on $\{0,1,\cdots,L\}$, which gives us the desired domination relation. 

Using this domination relation, it suffices to show that there exists a universal constant $\zeta>0$ such that for any $M\ge 1$ and $U_1,\cdots,U_M\stackrel{\text{i.i.d.}}{\sim}\operatorname{Uni}(\{0,1,\cdots,L\})$, the following holds with probability at least $0.95$:
\[
U_1+\cdots+U_M\ge \zeta\cdot M\cdot \min\{T,n\}\,.
\]
Note that $\mathbb{E}[U_m]=\frac{L}{2}\ge \frac{1}{100}\min\{T,n\}$ and $\operatorname{Var}({U}_m)\le \min\{T,n\}^2$. If $M\ge 10^6$, Chebyshev's inequality yields that
\[
\mathbb{P}\left[U_1+\cdots+U_M\le \frac{M\cdot \min\{T,n\}}{200}\right]\le \frac{M\cdot \min\{T,n\}^2}{\left({M\cdot \min\{T,n\}}/{200}\right)^2}=\frac{40000}{M}\le 0.04\,.
\]
Hence, with probability at least $0.96$, 
\[
U_1+\cdots+U_M\ge \frac{M\cdot \min\{T,n\}}{200}\,.
\]
If $M\le 10^6$, then $L\to \infty$ as $n\to \infty$. Hence, we have $U_1\ge \frac{L}{100}$ with probability at least $0.96$. This implies
\[
U_1+\cdots+U_M\ge U_1\ge \frac{L}{100}\ge \frac{1}{10^{10}}\cdot M\cdot\min\{T,n\}\,.
\]
Therefore, if we pick $\zeta=10^{-10}$, we get that for any $M\ge 1$ the inequality
\[
U_1+\cdots+U_M\ge \zeta\cdot M\cdot\min\{T,n\}\,
\]
holds with probability at least $0.96$. This completes the proof.
\end{proof}

\section{Deferred proof from Section~\ref{sec-negative}}\label{appendix-negative}

In this section, we provide the proofs that were deferred in Section~\ref{sec-negative}. We start by stating two lemmas on the concentration of the sum of independent variables, which will be frequently used.

The first result is a well-known Chernoff bound for the binomial distribution. 
\begin{lemma}\label{lem-Chernoff-bound}
    For any $N\in \mathbb{N},p\in (0,1)$, and $\delta>0$, denote $\mu=Np$. Then, for $X\sim \mathbf{B}(N,p)$, %(here $\mathbf{B}(N,p)$ denotes a binomial distribution with parameters $N,p$),
    \[
    \mathbb{P}[X\ge (1+\delta)\mu]\le \exp\big(-[(1+\delta)\log (1+\delta)-\delta]\mu\big)\,,
    \]
    and
    \[
    \mathbb{P}[X\le (1-\delta)\mu]\le \exp\left(-\frac{\delta^2}{2}\mu\right)\,.
    \]
\end{lemma}
The proof of Lemma~\ref{lem-Chernoff-bound} can be found in \cite[Theorem 4.4 and Theorem 4.5]{MU05}. Beyond the binomial case, we will also need to deal with weighted sum of independent Bernoulli variables in the later proofs. The following lemma gives a concentration bound for such weighted sums. 

\begin{lemma}\label{lem-weighted-Bernoulli-concentration}
    Let $\lambda_1,\cdots,\lambda_k\ge 0$ with $\max_{1\le i\le k}\lambda_i=\lambda^*$, and $p_1,\cdots,p_k\in [0,1]$. Consider the random variable
    \[
    X:=\lambda_1X_1+\cdots+\lambda_kX_k,\quad X_i\sim \mathbf{B}(1,p_i)\text{ are independent}, 1\le i\le k\,.
    \]
    For any $\theta$ such that $0<\theta\le (10\lambda^*)^{-1}$ and $M>0$, it holds that
    \[
    \mathbb{P}\big[|X-\mathbb{E}[X]|\ge M\big]\le 2\exp(\theta^2\lambda^*\mathbb{E}[X]-\theta M)\,.
    \]
\end{lemma}
\begin{proof}
    Note that
    \[
    \mathbb{E}[e^{\theta X}]=\prod_{i=1}^k(p_ie^{\theta\lambda_i}+1-p_i)=\prod_{i=1}^k\big(1+(e^{\theta\lambda_i}-1)p_i\big)\le \exp\left(\sum_{i=1}^k(e^{\theta\lambda_i}-1)p_i\right)\,.
    \]
    Using the fact that $e^x-1\le x+x^2$ for any $0\le x\le 1/10$ and our assumption on $\theta$, we see $\mathbb{E}[e^{\theta X}]$ is further upper-bounded by 
    \[
\exp\left(\theta\sum_{i=1}^k\lambda_ip_i+\theta^2\sum_{i=1}^k\lambda_i^2p_i\right)\le \exp\left(\theta\mathbb{E}[X]+\theta^2\lambda^*\mathbb{E}[X]\right)\,.
    \]
    Therefore, by Chebyshev's inequality, we have
    \[
    \mathbb{P}[X\ge \mathbb{E}[X]+M]\le \exp(-\theta\mathbb{E}[X]-\theta M)\cdot \mathbb{E}[e^{\theta X}]\le \exp(\theta^2\lambda^*\mathbb{E}[X]-\theta M)\,.
    \]
    Similarly, using the fact that $e^{-x}-1\le -x+x^2$ for any $0\le x\le 1/10$ and another Chebyshev's inequality, we have the same estimation for the lower tail, and thus the result follows.
\end{proof}

\addtocontents{toc}{\protect\setcounter{tocdepth}{1}} % Lower TOC depth
\subsection{Proof of key lemmas}\label{appendix-proof-of-key-lemmas}

This subsection provides the proof for key Lemmas~\ref{lem-MI}, \ref{lem-FGTL}, and \ref{lem-RW}, which address important properties related to the evolution of voter dynamics.

\subsubsection{Proof of Lemma~\ref{lem-MI}}\label{appendix-lem-MI}
\begin{proof}
Since $S_m^{(0)}\sim 2\mathbf{B}(n,\frac{1}{2})-n$ for each $1\le m\le M$, the first condition $|S_m^{(0)}|\le n/\log n,\ \forall\ 1\le m\le M$ holds w.o.p.. In what follows, we focus on \eqref{eq-moderate-increment}.

First, we show that w.o.p. $T_m\le n(\log n)^2$ for all $1\le m\le M$. Recall that $T_m\le T_{m,\operatorname{cons}}\le T_{m,\operatorname{coal}}$ by duality to coalescing random walks (see \eqref{eq-T-coal}). An application of the union bound yields that for each $1\le m\le M$,
$\mathbb{P}[T_{m,\operatorname{cons}}\ge n(\log n)^2]$ is upper-bounded by 
\begin{equation*}%\label{eq-max-prob}
\sum_{i\neq j}\mathbb{P}[\text{two independent random walks on $G^*$ starting at }i,j\text{ do not meet in }n(\log n)^2\text{ steps}]\,.
\end{equation*}
However, as argued in Section~\ref{appendix-admissible}, we have that each probability term in the above sum is upper-bounded by $(1-{(2n)}^{-1})^{\Theta(n(\log n)^2)}$, which is super-polynomially small. This shows that $T_m\le n(\log n)^2$ happens w.o.p. for each $m$, thus verifying the claim. 

%In light of this bound on $T_m$, it suffices to show that \eqref{eq-moderate-increment} happens w.o.p. for any $1\le m\le M$ and $0\le t\le n(\log n)^2$. To see this, we fix $m,t$ and omit $m$ in the subscript for simplicity. We condition on any realization of $x_{i}^{(t)},i\in [n]$ and write $\Delta=n-|S^{(t)}|$. We first assume $\Delta\ge (\log n)^2$, and prove that with overwhelming conditional probability $|S^{(t+1)}-S^{(t)}|\le \xi_n\Delta$. Let $\operatorname{-1}^{(t)}$ denote the set of vertices with labels $-1$ at time $t$. Without loss of generality, we assume that $S^{(t)}=n-\Delta$, and so $|\operatorname{-1}^{(t)}|=\Delta/2$. Let
In light of this bound on $T_m$, it suffices to show that \eqref{eq-moderate-increment} happens w.o.p. for any $1\le m\le M$ and $0\le t\le n(\log n)^2$. To see this, we fix $m$ and $t$, omitting $m$ from the subscript for simplicity. We condition on any realization of $x_{i}^{(t)},i\in [n]$, and write $\Delta=n-|S^{(t)}|$. We first assume $\Delta\ge (\log n)^2$, and prove that with overwhelming conditional probability, we have $|S^{(t+1)}-S^{(t)}|\le 2(\log n)^2\Delta^{1/2}$. Let $\operatorname{-1}^{(t)}$ denote the set of vertices with labels $-1$ at time $t$. Without loss of generality, we assume that $S^{(t)}=n-\Delta$, so that $|\operatorname{-1}^{(t)}|=\Delta/2$. Let
\[
p_i^{(t+1)}:=\frac{|N_i\cap \operatorname{-1}^{(t)}|}{d_i}\,,\ \forall i\in [n]\,.
\]
Then, conditioned on $x_i^{(t)},i\in [n]$, the size of $\operatorname{-1}^{(t+1)}$ is distributed as the sum of $n$ independent Bernoulli variables with parameters $p_1^{(t+1)},\cdots,p_n^{(t+1)}$, respectively. We observe that
\[
\sum_{i=1}^{n}p_i^{(t+1)}=\frac{1+o(1)}{np}\cdot \sum_{i=1}^n |N_i\cap \operatorname{-1}^{(t-1)}|=\frac{1+o(1)}{np}\cdot \sum_{i\in \operatorname{-1}^{(t)}}d_i^{\operatorname{in}}=(1+o(1))|\operatorname{-1}^{(t-1)}|\,,
\]
where we have used the degree concentration in property-(i) of admissibility.
%Hence, the expected value of $|\operatorname{-1}^{(t)}|$ is $(1/2+o(1))\Delta$. By Lemma~\ref{lem-weighted-Bernoulli-concentration} with $\lambda^*=1,\theta=\xi_n/5$ and $M=\xi_n\Delta$, the probability that this sum deviates from its expectation by more than $\xi_n\Delta$ is at most $2\exp(-\Omega(\xi_n^2\Delta))$. Since $(\log n)^{-1/2}\ll \xi_n\ll 1$, this deviation probability is super-polynomially small. This proves that w.o.p. $|S^{(t+1)}-S^{(t)}|=o(\Delta)$ provided that $\Delta\ge (\log n)^2$.
Hence, the expected value of $|\operatorname{-1}^{(t)}|$ is $(1/2+o(1))\Delta$. Applying Lemma~\ref{lem-weighted-Bernoulli-concentration} with $\lambda^*=1,\theta=\Delta^{-1/2}$, and $M=(\log n)^2\Delta^{1/2}$, we get that the probability that this sum deviates from its expectation by more than $(\log n)^2\Delta^{1/2}$ is at most $2\exp(-\Omega((\log n)^{2}))$. This deviation probability is super-polynomially small. This shows that w.o.p. $|S^{(t+1)}-S^{(t)}|=o(\Delta)$ provided that $\Delta\ge (\log n)^2$. 

We are left with the case where $\Delta\le (\log n)^2$. We still have that $|\operatorname{-1}^{(t+1)}|$ is a sum of Bernoulli variables with mean no more than $(1+o(1))(\log n)^2$. Again, the Chernoff bound tells us that the probability that this exceeds $2(\log n)^2$ is super-polynomially small. The proof is now completed.
\end{proof}

\subsubsection{Proof of Lemma~\ref{lem-FGTL}}\label{appendix-lem-FGTL}

\begin{proof}
The proof proceeds in the same spirit as that of Lemma~\ref{lem-MI}, but requires more nuanced arguments.
We show that for each fixed $m,t$, \eqref{eq-from-local-to-global-event} holds w.o.p.. By the same reasoning as before, this implies the desired result. In what follows, we fix $m$ and $t$, and omit the subscript of $m$.

Recall Property-(iii) of admissibility. We pick $L=c(\delta,2)$ as in Theorem~\ref{thm:mixing} (i.e., the random walk on $G^*$ achieves $n^{-2}$-mixing after $L$ steps). We first consider the case $t>L$. We condition on the realizations of $x_{i}^{(t-L)}$ for $i\in [n]$ and let $\Delta:=n-|S^{(t-L)}|$. We begin by further assuming that $\Delta\ge \frac{1}{2}p^{-1}(\log n)^2$. We will prove that with overwhelming conditional probability, 
\begin{equation}\label{eq-di-Sit}
\frac{p\Delta}{2}\le d_i-|S_i^{(t)}|\le \frac{3p\Delta}{2}\,,\quad\forall i\in [n]\,.
\end{equation}

Without loss of generality, we assume that $\Delta=n-S^{(t-L)}$ so that $|\operatorname{-1}^{(t-L)}|=\frac12 \Delta$. As before, denote by $\pi_{i,l},1\le l\le L$ the distribution of the location after $l$ steps of a random walk on $G^*$ starting at $i$. For any fixed $i\in [n]$, we claim that for each $1\le l\le L+1$, it holds with overwhelming conditional probability that
\begin{equation}\label{eq-L-step}
\left(\frac12-\frac{L+2-l}{4(L+2)}\right)\frac{\Delta}{n}\le\mathbb{E}_{j\sim \pi_{i,l}}[\mathbf{1}\{x_j^{(t+1-l)}=-1\}]\le \left(\frac{1}{2}+\frac{L+2-l}{4(L+2)}\right)\frac{\Delta}{n}\,.
\end{equation}

We prove this claim inductively, starting from $l=L+1$ and going down to $l=1$. For the case $l=L+1$, since $\operatorname{TV}(\pi_{i,L},\pi)\le n^{-2}$ and $\pi(i)=1/n+o(1/n)$, the result follows from our assumption that $|\operatorname{-1}^{(t-L)}|=\frac12 \Delta$. Now, for any $1\le l\le L$, suppose that \eqref{eq-L-step} occurs w.o.p. for $l+1$. We condition on a realization of all $x_{i}^{(t-l)},i\in [n]$ and assume that \eqref{eq-L-step} holds for $l+1$. Then the conditional distribution of $\mathbb{E}_{j\sim \pi_{i,l}}[\mathbf{1}\{x_j^{(t-l+1)}=-1\}]$ is a weighted sum of independent Bernoulli variables, and the mean of this sum equals $\mathbb{E}_{j\sim \pi_{i,l+1}}[\mathbf{1}\{x_j^{(t-l)}=-1\}]$, which lies in the interval
\[
\left[\left(\frac12-\frac{L+1-l}{4(L+2)}\right)\frac{\Delta}{n},\left(\frac12+\frac{L+1-l}{4(L+2)}\right)\frac{\Delta}{n}\right]
\]
by our assumption. It is easy to see that the maximal weight in this sum, $\max_{j\in [n]}\pi_{i,l}[j]$, is at most $O((np)^{-1})$. Therefore, using Lemma~\ref{lem-weighted-Bernoulli-concentration} with parameters $\lambda^*=\Theta((np)^{-1})$, $\theta=\Theta(np)$, and $M=\frac{\Delta}{4(L+2)n}$, we conclude that the probability that $\mathbb{E}_{j\sim \pi_{i,l}}[\mathbf{1}\{x_j^{(t-l+1)}=-1\}]$ deviates from its expectation by at least $\frac{\Delta}{4(L+2)n}$ is upper-bounded by
$\exp(-\Theta(p\Delta))$. This probability is super-polynomially small by our assumption on the size of $\Delta$. Therefore, we conclude that \eqref{eq-L-step} holds w.o.p. for $l$. This completes the induction and verifies the claim.

Taking $l=1$ and noting that $\pi_{i,1}$ is the uniform distribution on $\operatorname{N}_i$, we get that for each $i\in [n]$, w.o.p.
\[
\left(\frac14+\frac{1}{4(L+1)}\right)\frac{\Delta}{n}\le \frac{|\operatorname{N}_i\cap \operatorname{-1}^{(t)}|}{d_i}\le \left(\frac34-\frac{1}{4(L+2)}\right)\frac{\Delta}{n}\,.
\]
Using the union bound, we see that w.o.p. the above holds for all $i\in [n]$.
Since $d_i=(1+o(1))np$ for all $i\in [n]$, we conclude that \eqref{eq-di-Sit} holds w.o.p.. Furthermore, conditioned on the event that $\mathcal G_{\operatorname{MI}}$ occurs, this implies that $\frac{\Delta}{2}\le n-|S^{(t)}|\le 2\Delta$. Thus, under \eqref{eq-di-Sit} and $\mathcal G_{\operatorname{MI}}$, \eqref{eq-from-local-to-global-event} holds. Hence, \eqref{eq-from-local-to-global-event} holds w.o.p..

We now assume that $\Delta\le \frac12 p^{-1}(\log n)^2$. Then, under $\mathcal G_{\operatorname{MI}}$, we have $n-|S^{(t)}|<p^{-1}(\log n)^2$, and thus the first two inequalities in \eqref{eq-from-local-to-global-event} are trivial. For the last inequality, we can once again show inductively that for each $1\le l\le L+1$, w.o.p. $\mathbb{E}_{j\sim \pi_{i,l}}[\mathbf{1}\{x_j^{(t-l+1)}=-1\}]\le \left(\frac{1}{2}+\frac{L+1-l}{2L}\right)\frac{(\log n)^2}{n}$. Taking $l=1$, this reduces to the last inequality in \eqref{eq-from-local-to-global-event}, thus completing the analysis of the case where $t>L$.

Finally, for the case $t\le L$, by assuming $\mathcal G_{\operatorname{MI}}$ holds, we get that w.o.p. $n-|S^{(t)}|=(1-o(1))n$. Moreover, using ideas similar to the ones above, it is readily seen that w.o.p. $d_i-|S_i^{(t)}|=(1-o(1))np$. Thus, \eqref{eq-from-local-to-global-event} holds w.o.p. too. This concludes the proof.
\end{proof}

\subsubsection{Proof of Lemma~\ref{lem-RW}}\label{appendix-lem-RW}

\begin{proof}
We fix $1\le m\le M$ and prove that all the properties happen w.o.p. for $m$. In the following, we omit the subscript $m$.
Recall that $\pi$ denotes the stationary measure of $G^*$. We consider the stochastic process $\{W^{(t)}\}_{t\ge 0}$ that closely follows $\{S^{(t)}\}_{t\ge 0}$ given by \[W^{(t)}=\mathbb E_{\pi}\big[x_{i}^{(t)}\big]=\sum_{i\in[n]}\pi(i)x_{i}^{(t)}.\] 
As argued in Section~\ref{sec-positive}, $W^{(t)}$ is a martingale with respect to the filtration \[\widetilde{\mathcal F}_t=\sigma\big(x_{i}^{(t')}:i\in [n],t'\leq t\big)\,.\] 

Next, we define 
\begin{align*}
    \operatorname{1}^{(t)}:=\{i\in[n]\ \mid\ x_{i}^{(t)}=1\}\,,\quad
    \operatorname{-1}^{(t)}:=\{i\in[n]\ \mid\ x_{i}^{(t)}=-1\}\,.
\end{align*} 
Since $G^*$ is admissible, we have that \[W^{(t)}=\pi(\operatorname{1}^{(t)})-\pi(\operatorname{-1}^{(t)})=1-2\pi(\operatorname{-1}^{(t)})=1-(2+o(1))\frac{|\operatorname{-1}^{(t)}|}{n}\ .\]  Fix $\Delta\in[0,n]$ and set $\Delta'=\Delta/n$. If $n-\Delta\leq S^{(t)}$, then $|\operatorname{-1}^{(t)}|\leq \Delta/2$ and $W^{(t)}\geq1-(1-o(1))\Delta'$. Analogously, if $n-\Delta\leq -S^{(t)}$, then $W^{(t)}\leq -1+(1-o(1))\Delta'$. Hence, for part (i), it suffices to show that there are no pairs $(m,t)$ with $t\leq \min\{T,n/(\log n)^{10}\}$ and $|W^{(t)}|\leq 1/3$, while for part (ii) we shall upper-bound the number of $t\le T_{\operatorname{cons}}$ with $1-|W^{(t)}|\leq 3\Delta'/2$ (for $\Delta'\ge 1/(\log n)^{10}$).

Let us first deal with part (i). Let $B$ denote the event that there exists $t$ such that $|S^{(t)}|\leq n/2$ and the inequality $|S^{(t+1)}-S^{(t)}|\leq 2(\log n)^2\sqrt{n}$ does not hold. From Lemma~\ref{lem-MI}, we see that $\Pb [B]$ is super-polynomially small. By the arguments above, if $B$ does not occur, then $|W^{(t+1)}-W^{(t)}|\leq O((\log n)^2n^{-1/2})$ whenever $|W^{(t)}|\leq 1/3$. We now introduce an auxiliary martingale $\{Z^{(t)}\}_{t\ge 0}$ defined by setting $Z^{(0)}=W^{(0)}$, and for all $t\ge 0$,
$$
Z^{(t+1)}=\begin{cases}
  W^{(t+1)}, & \mbox{if } |Z^{(t)}|< 1/3 \\
  Z^{(t)}, & \mbox{if }  |Z^{(t)}|\geq 1/3.
\end{cases}
$$
 In other words, the new martingale $Z$ is a copy of $W$ up to the point where $|Z^{(t)}|$ exceeds $1/3$; afterwards, it becomes constant. 

A standard martingale concentration inequality (see Theorem 8.3 in~\cite{chung2006concentration}) yields that for $ t=\lfloor n/(\log n)^{10}\rfloor$, 
\begin{align*}
\Pb[|Z^{(t)}|\geq1/3]\leq&\ 2\exp\left(-\frac{1/9}{2\sum_{i=1}^{t}{O}((\log n)^4/n)}\right)+\Pb[B]\\
\leq&\ 2\exp(-\Omega((\log n)^6))+\Pb[B]\,,
\end{align*}
which is super-polynomially small. Moreover, $|Z^{(t)}|<1/3$ implies that $|W^{(t')}|\le 1/3$ for all $t'\leq t$. Thus, the result follows by applying the union bound over all $m$ with $1\leq m\leq M$.

Now we move on to the proof of (ii). Fix $\Delta\leq n/(\log n)^{10}$. We will bound the number of times $t$ for which $W^{(t)}\geq 1-3\Delta'/2$. The other case $W^{(t)}\le -1+3\Delta'/2$ can be handled symmetrically.
As hinted earlier, our strategy will consist of exploiting the fact that once $W^{(t)}$ is relatively close to $1$, there is a good chance that the process $W^{(t+1)}, W^{(t+2)},\cdots$ will reach $1$ without ever going too far below $|W^{(t)}|$. To make this precise, we first need a lower bound on the expected squared change of the martingale at step $t+1$. More precisely, we claim that if we denote $1-|W^{(t)}|$ by $q$, then 
\begin{equation}\label{eq-change-lower-bound}
    \mathbb E\left[(W^{(t+1)}-W^{(t)})^2\mid \widetilde{\mathcal F}_t\right]\geq\Omega({q/n}) \, . 
\end{equation}

As we observed in the proof of Lemma~\ref{lem-effective-observations-lower-bound}, \[\mathbb{E}\big[(W^{(t+1)}-W^{(t)})^2\mid\widetilde{\mathcal F}_{t}\big]=\sum_{i\in [n]}\pi(i)^2\operatorname{Var}[x_{i}^{(t+1)}\mid \widetilde{\mathcal F}_{t}]\,.\] We now carefully analyze the terms $\operatorname{Var}[x_{i}^{(t+1)}\mid \widetilde{\mathcal F}_{t}]$. Recall that $S_{i}^{(t)}=\sum_{j\in\operatorname{N_i}}x_{j}^{(t)}$. A straightforward computation yields \[\operatorname{Var}[x_{i}^{(t)}\mid \widetilde{\mathcal F}_{t}]= \frac{(d_i-S_{i}^{(t)})(d_i+S_{i}^{(t)})}{4d_i^2}=\frac{|P_i^{(t)}|}{d_i^2}\, ,\] where $P_i^{(t)}$ is the set of pairs of indices $(j,j')$ with $j,j'\in [n]$ such that $j,j'\in \operatorname{N}_i$, $x_{j}^{(t)}=-1$, and $x_{j'}^{(t)}=1$. Since $G^*$ is admissible, we have $d_i\leq 2np$. We aim to lower-bound the quantity $(2np)^{-2}\sum_{i\in[n]}|P_i^{(t)}|$. Note that $\sum_{i\in[n]}|P_i^{(t)}|$ counts the number of triples $(i,j,j')\in[n]^3$ with $j,j'\in \operatorname{N}_i$, $x_{j}^{(t)}=-1$, and $x_{j'}^{(t)}=1$. By Property-(iv) of admissibility, this quantity is at least $p^2n\cdot |\operatorname{1}^{(t)}|\cdot|\operatorname{-1}^{(t)}|/10^{12}$. Hence, we can assume that $\sum_{i\in[n]}|P_i^{(t)}|(2np)^{-2}\geq \min\{|\operatorname{1}^{(t)}|,|\operatorname{-1}^{(t)}\}|/(4\cdot 10^{12}).$ The claim follows from a simple calculation, using again the fact that $\pi(i)=1/n+o(1/n)$ for all $i\in[n]$.

We now set the ground for a multi-scale analysis of the evolution of the martingale $\{W^{(t)}\}$. Consider the real numbers $k_s=1-2^s(\log n)^5/n$ for $0\leq s\leq \log_2\lfloor(n/(\log n)^5)\rfloor=s_*$. For convenience, write $k_{-1}=1$ and $k_{s^*+1}=-1$. Partition $[-1,1]$ into subintervals $I_0=[k_0,1]$ and $I_s=[k_{s},k_{s-1})$ for $s=1,\cdots, s_*+1$. 

We say that $t$ is a \textit{crossing time} if $W^{(t-1)}$ and $W^{(t)}$ lie in different subintervals of the partition defined above. Denote by $\eta_1<\eta_2<\cdots<\eta_r$ all the crossing times (note that there are finitely many almost surely). For each crossing time $\eta_i$, define $s(\eta_i)$ as follows: If $W^{(\eta_i-1)}<W^{(\eta_i)}$, then $s(\eta_i)$ is the index of the subinterval containing $W^{(\eta_i)}$; otherwise, let $s(\eta_i)$ be the index of the subinterval containing $W^{(\eta_i-1)}$. The advantage of defining $s(\eta_i)$ in this way is that, as we will see later, the $\mathcal G_{\operatorname{MI}}$ property implies that $W^{(\eta_i)}$ must be very close to $k_{s(\eta_i)}$. 

Note that the martingale $W_{t}$ could spend several steps fluctuating around some neighborhood of $k_{s}$, thus producing multiple contiguous crossing times where the martingale alternates between between the intervals $I_{s-1}$ and $I_s$. This motivates the following definition. A crossing time $\eta_i$ is called \textit{novel} if $s(\eta_i)\neq s(\eta_{i-1})$. Let $\tau_1<\tau_2<\cdots<\tau_{r'}$ denote the novel crossing times.  Finally, let $\tau_{r'+1}$ denote the first time the martingale hits either $1$ or $-1$ and set $s(\tau_{r'+1})=-1$ or $s(\tau_{r'+1})=s^*+1$, accordingly. See Figure~\ref{fig:random_Walk} for an illustration.

\begin{figure}
    \centering
    \includegraphics[width=0.9\linewidth]{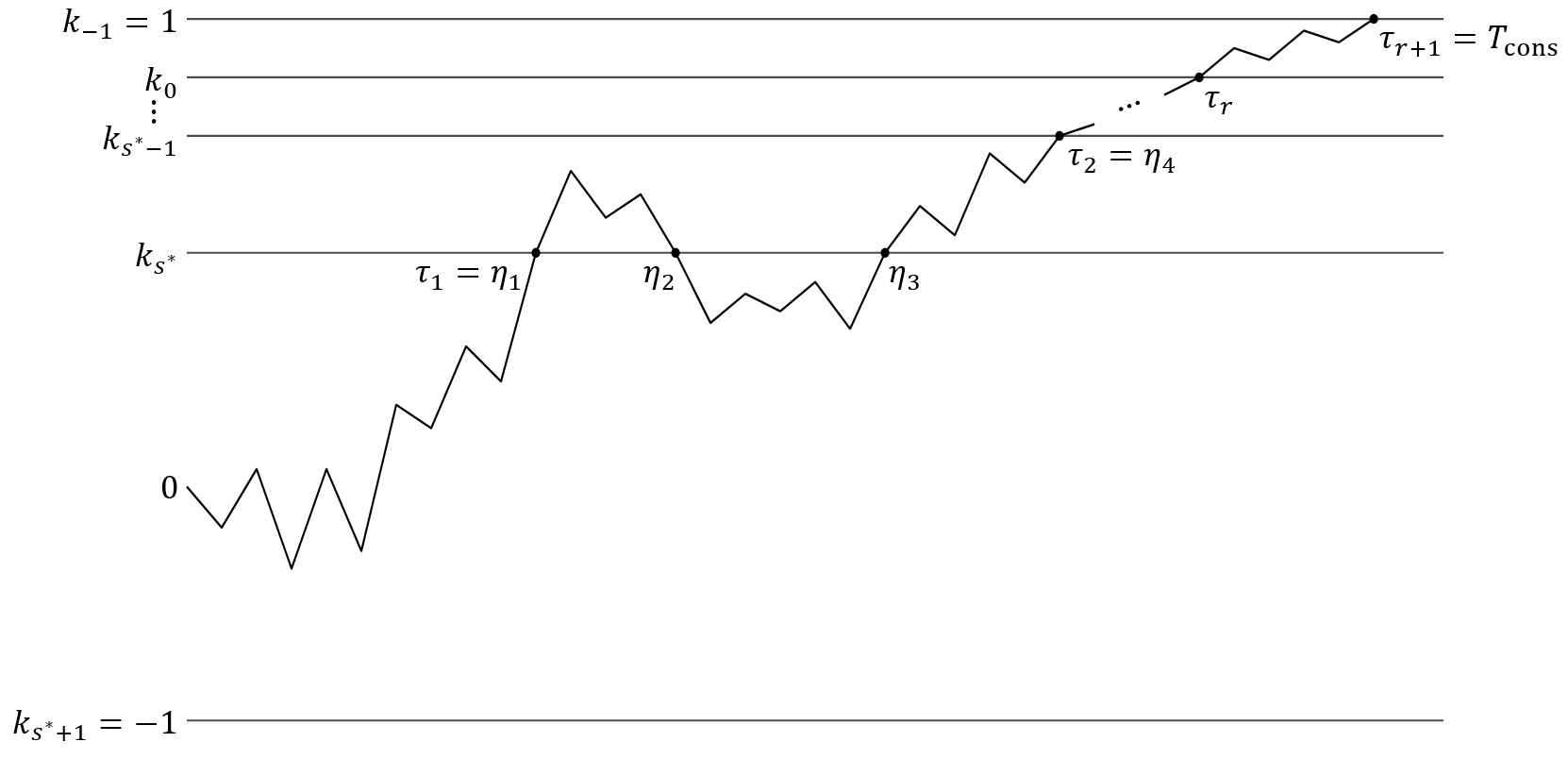}
    \caption{$W^{(t)}$ viewed as a random walk}
    \label{fig:random_Walk}
\end{figure}

By Lemma~\ref{lem-MI} and the fact that $|W^{(t)}|$ is uniformly bounded by $1$, we can condition on the event that $\mathcal G_{\operatorname{MI}}$ holds throughout the rest of the proof (strictly speaking, this might add some ``noise" that causes the martingale property to break; however, the effect is negligible and can be ignored for the remainder of our argument). Given that $\mathcal G_{\operatorname{MI}}$ holds, we have that $s(\tau_i)$ and $s(\tau_{i+1})$ differ by at most one for every $1\leq i\leq r'$. Furthermore, the fact that $\pi(i)=1/n+o(1/n)$, combined with property $\mathcal G_{\operatorname{MI}}$, implies that \[\left|W^{(\tau_i)}-k_{s(\tau_i)}\right|\leq 2(\log n)^2(n-|S^{(\tau_i-1)}|)^{1/2}\] holds for all $1\leq i\leq r'+1$. In particular, this means that $W^{\tau_i}=k_{s(\tau_i)}+o(1-k_{s(\tau_i)})$. This allows us to think of the sequence $\{W^{(\tau_i)}\}_{i=1}^{r'+1}$ as a random walk which (approximately) takes values in the set $\{k_{-1},k_0,\cdots,k_{s^*+1}\}$. Our next goal is to better understand the transition probabilities of this random walk.

For $i\leq r'$, a direct application of the Martingale Stopping Theorem yields \[\mathbb E\left[W^{(\tau_{i+1})}\mid \widetilde{\mathcal F}_{\tau_i}\right]=W^{(\tau_i)}\, .\] Combining this with the above observation, we arrive at
\begin{align}\label{eq-MI-implication}
    \mathbb E\left[k_{s(\tau_{i+1})}\mid \widetilde{\mathcal F}_{t}\right]=k_{s({\tau_i})}+o(1-k_{s({\tau_i})})\, .
\end{align}  
It is convenient to rewrite this as 
\begin{align}
    \mathbb E\left[1-k_{s(\tau_{i+1})}\mid \widetilde{\mathcal F}_{t}\right]=(1+o(1))(1-k_{s({\tau_i})})\, .
\end{align}
Observe now that, so long as $1\leq s(\tau_i)\leq s^*-1$, we have $1-k_s=2(1-k_{s-1})=(1-k_{s+1})/2$. Thus, after conditioning on the event that $1\leq s(\tau_i)\leq s^*-1$, a straightforward computation yields that
\begin{align*}
    \Pb\left[s(\tau_{i+1})=s(\tau_{i})+1\right]={2}/{3}+o(1)\,,\quad
    \Pb\left[s(\tau_{i+1})=s(\tau_{i})-1\right]={1}/{3}+o(1)\,.
\end{align*}
If $s(\tau_i)=0,$ then we have that $k_0=(1-k_{1})/2$ and $k_{-1}=0$, so the same argument yields
\begin{align*}
\Pb\left[s(\tau_{i+1})=-1\right]=1/2+o(1)\,,\quad
\Pb\left[s(\tau_{i+1})=1\right]=1/2+o(1)\,.
\end{align*}
Also note that in the latter case, if $s(\tau_{i+1})=-1$, then $r'=i$ and the process ends.

For $0\leq s\leq s^*-1$, let $\vartheta (s)$ denote the number of indices $i$ such that $s(\tau_i)=s$. By a standard argument using the Green function of a biased random walk on $\mathbb Z$, each $\vartheta (s)$ is stochastically dominated by a geometric random variable with mean $5$. In particular, this implies that $\Pb\left[\vartheta(s)\ge (\log n)^2\right]$ is super-polynomially small. Let $s(\Delta'):=\lceil\log_2((3/2)\Delta'\cdot n/(\log n)^2)\rceil$. In this way, $k_{s(\Delta')}$ is the largest $k_s$ which is at most $1-3\Delta'/2$ (since $\Delta\leq n/(\log n)^{10}$, we may assume that $s(\Delta')< s^*$), and a direct application of the union bound yields that
\begin{equation}\label{eq-transient}
    \Pb\left[\sum_{s=0}^{s(\Delta')}\vartheta(s)\geq \left(s(\Delta')+1\right)(\log n)^2\right]\,
\end{equation} 
is again super-polynomially small.

Observe now that for any $t$ with $W^{(t)}\geq 1-3\Delta'/2$, it must be the case that $\tau_i\leq t\leq \tau_{i+1}$ for some $i$ such that $\tau_{i}$ satisfies $0\leq s(\tau_i)\leq s(\Delta')$. By~\eqref{eq-transient}, w.o.p. the total number of novel crossing times $\tau_i$  satisfying $0\leq s(\tau_i)\leq s(\Delta')$ is less than $\left(s(\Delta')+1\right)(\log n)^2$. Thus, all that remains is to control the distances $\tau_{i+1}-\tau_i$ between contiguous novel crossing times satisfying  $0\leq s(\tau_i)\leq s(\Delta')$. 

We claim that
\begin{align}\label{eq-probability-large-gap}
    \Pb\left[\tau_{i+1}-\tau_i\geq \max\{\Delta(\log n)^2,(\log n)^7\}\mid \widetilde{\mathcal F}_{\tau_i},0\leq s(\tau_i)\leq s(\Delta')\right]\, 
\end{align}
is super-polynomially small. Towards the goal of proving this statement, we will first show that
\begin{equation}\label{eq-expected-time}
    \mathbb E\left[\tau_{i+1}-t\mid \widetilde{\mathcal F}_{t},\tau_{i}\leq t<\tau_{i+1},0\leq s(\tau_i)\leq s
    (\Delta')\right]\leq O\left(\max\{\Delta,(\log n)^5\}\right)\, .
\end{equation}
Whenever $\tau_{i}\leq t'<\tau_{i+1}$, we have that $k_{s(\tau_i)-1}\leq W^{(t')}\leq k_{s(\tau_{i})+1}$, and thus $1-W^{(t')}=\Theta(1-k_{s(\tau_i)})$. Hence, after conditioning on $\widetilde{\mathcal F}_{t}$ for $\tau_{i}\leq t<\tau_{i+1}$ and $0\leq s(\tau_i)\leq s(\Delta')$, repeated applications of~\eqref{eq-change-lower-bound} yield
\begin{align*}
\mathbb E\left[(W^{(\tau_{i+1})})^2\right]= & \ \mathbb E[(W^{(t)})^2]+\sum_{t'=\tau_i}^{\tau_{i+1}-1}\mathbb E[(W^{(t'+1)}-W^{(t')})^2]\\
\geq & \ \mathbb E[(W^{(t)})^2]+\mathbb E[t-\tau_i]\Omega(1-k_{s(\tau_i)})/n,
\end{align*}
where the martingale property is used in the first line. On the other hand, a standard computation using the Martingale Stopping Theorem and~\eqref{eq-MI-implication} gives \[\mathbb E[(W^{(\tau_{i+1})})^2]-\mathbb E[(W^{(t)})^2]=\Theta\left((1-k_{s(\tau_i)})^2\right).\] Putting both inequalities together, we arrive at \[ \mathbb E\left[\tau_{i+1}-t\mid \widetilde{\mathcal F}_{t},\tau_{i}\leq t<\tau_{i+1},0\leq s(\tau_i)\leq s(\Delta')\right]\cdot\Omega(1-k_{s(\tau_i)})/n\leq\Theta\left((1-k_{s(\tau_i)})^2\right)\, ,\] which implies~\eqref{eq-expected-time} after noting that $(1-k_{s(\tau_i)})n=O(\max\{\Delta,(\log n)^5\})$ whenever $0\leq s(\tau_i)\leq s(\Delta')$ holds.

Using Markov's inequality, we deduce from~\eqref{eq-expected-time} that there exists a constant $C'$ with \[\Pb\left[\tau_{i+1}-t\geq C'\max\{\Delta,(\log n)^5\}\mid \widetilde{\mathcal F}_{t},\tau_{i}\leq t<\tau_{i+1},0\leq s(\tau_i)\leq s^*-1\right]\leq 1/2\,.\] 
Successively applying this bound to $t=\tau_{i}+\ell C'\max\{\Delta,(\log n)^5\},\ell=0,1,2,\cdots$ for as long as $t<\tau_{i+1},$ we get 
%\SH{here, all the $C'$'s are typo, right? (just wanted to confirm before changing)}
\begin{align*}
    \Pb\left[\tau_{i+1}-\tau_i\geq K\cdot C'\max\{\Delta,(\log n)^5\}\mid \widetilde{\mathcal F}_{\tau_i},0\leq s(\tau_i)\leq s(\Delta')\right]\leq 2^{-K}\, ,
\end{align*}
and~\eqref{eq-probability-large-gap} follows.

Via the union bound,~\eqref{eq-transient} and~\eqref{eq-probability-large-gap} imply that
\[\Pb\left[\#\{0\le t\le T_{\operatorname{cons}}:W^{(t)}\geq1-3\Delta'/2\}\geq(s(\Delta')+1)(\log n)^2\cdot \max\{\Delta(\log n)^2,(\log n)^7\}\right]\,\] is super-polynomially small. Since $s(\Delta')+1\leq \log n$, using the union bound again, along with the fact that w.o.p. $M\leq \widetilde O(n^2)$, we conclude that $\#\{t\le T_{\operatorname{cons}}:W^{(t)}\geq1-3\Delta'/2\}$ is upper-bounded by $\max\{\Delta(\log n)^6,(\log n)^{10}\}$ w.o.p.. Similarly, we can upper-bound w.o.p. the number of times $t$ for which $W^{(t)}\leq -1+3\Delta'/2$ in exact same way. This concludes the proof.
\end{proof}

\subsection{Proof of technical lemmas}\label{appendix-technical-lemmas}

This subsection provides the proof of technical lemmas~\ref{lem-stochastic-domination}, \ref{lem-Sgood-Sbad}, and \ref{lem-quantitative-estimate}.

\subsubsection{Proof of Lemma~\ref{lem-stochastic-domination}}\label{appendix-lem-stochastic-domination}

\begin{proof}
As we hinted at earlier, it suffices to show that for any $1 \le m \le M$ and $0 \le t \le T - 1$, conditioned on any realization of $\{x_{m,i}^{(t')}: i \in [n],\, t' < t\} \cup \{x_{m,a}^{(t+1)}\}$ that does not violate $\mathcal{G}_{\operatorname{MI}} \cap \mathcal{G}_{\operatorname{FGTL}}$, both $\mathcal G_{\operatorname{MI}}\cap \mathcal G_{\operatorname{FGTL}}$ and the second case in \eqref{eq-3-cases} happen simultaneously with conditional probability at most $p_{m,t}$. The details can be found in Appendix~\ref{appendix-lem-stochastic-domination}. 

The case $t = 0$ is straightforward, as $p_{m,0} = 1$ under $\mathcal{G}_{\operatorname{MI}}$. To prove the general case $t > 0$, we note that, given $x_{m,i}^{(t-1)}$ for $i \in [n]$, it holds that  
\[
\mathbb{P}[x_{i,b}^{(t)} = \pm 1] = \frac{d_b \pm S_{m,b}^{(t-1)}}{2d_b}\,, \quad \mathbb{P}[x_{i,c}^{(t)} = \pm 1] = \frac{d_c \pm S_{m,c}^{(t-1)}}{2d_c}\,,
\]
and that $x_{m,b}^{(t)}$ and $x_{m,c}^{(t)}$ are conditionally independent. Therefore, under this conditioning, we obtain  
\[
\mathbb{P}[x_{m,b}^{(t)} \neq x_{m,c}^{(t)}] \le \frac{d_b - |S_{m,b}^{(t-1)}|}{2d_b} + \frac{d_c - |S_{m,c}^{(t-1)}|}{2d_c}\,.
\]
Since $\{x_{m,i}^{(t-1)} : i \in [n]\}$ does not violate $\mathcal{G}_{\operatorname{FGTL}}$, we get that  
\[
\max\{d_b - |S_{m,b}^{(t-1)}|,\, d_c - |S_{m,c}^{(t-1)}|\} \le 10\max\{(\log n)^2,\, p(n - |S_m^{(t-1)}|)\}\,,
\]
and thus the above probability is upper-bounded by (since $d_b, d_c \ge \frac{1}{2}np$)  
\[
\frac{40\max\{p(n - |S_m^{(t-1)}|),\, (\log n)^2\}}{np}\,.
\]

Now, we further reveal all the random variables $\{x_{m,i}^{(t)} : i \in [n]\}$. If $\mathcal{G}_{\operatorname{MI}} \cap \mathcal{G}_{\operatorname{FGTL}}$ is not violated, we have  
\[
\max\{p(n - |S_n^{(t-1)}|),\, (\log n)^2\} = (1 + o(1)) \max\{p(n - |S_n^{(t)}|),\, (\log n)^2\}
\]
and  
\[
d_a - |S_{m,a}^{(t)}| \le 10 \max\{p(n - |S_m^{(t)}|),\, (\log n)^2\} \le 11 \max\{p(n - |S_m^{(t-1)}|),\, (\log n)^2\}\,.
\]
In this case, the conditional probability of the event $\{x_{m,a}^{(t+1)} S_{m,a}^{(t)} < 0\}$ equals  
\[
\frac{d_a - |S_{m,a}^{(t)}|}{2d_a} \le \frac{11 \max\{p(n - |S_m^{(t-1)}|),\, (\log n)^2\}}{np}\,.
\]
Putting things together, we see that conditioned on $\{x_{m,i}^{(t')} : i \in [n],\, t' < t\} \cup \{x_{m,a}^{(t)}\}$, the event $\mathcal{G}_{\operatorname{MI}} \cap \mathcal{G}_{\operatorname{FGTL}}$ holds while the second case in~\eqref{eq-3-cases} happens with conditional probability at most  
\[
\max\left\{1,\ \frac{40 \max\{p(n - |S_m^{(t-1)}|),\, (\log n)^2\}}{np} \times \frac{11 \max\{p(n - |S_m^{(t-1)}|),\, (\log n)^2\}}{np} \right\} \le p_{m,t}\,.
\]
This concludes the proof.
\end{proof}

\subsubsection{Proof of Lemma~\ref{lem-Sgood-Sbad}}\label{appendix-lem-Sgood-Sbad}

\begin{proof}
For the first probability, recall that we have already shown that $\mathbb{E}[\mathcal{S}_{\operatorname{good}}] \le 10^5 c \log n$. Additionally, $\mathcal{S}_{\operatorname{good}}$ is a weighted sum of independent Bernoulli variables with maximum weight no more than $(\log n)^{-4}$ (by the definition of good pairs). Therefore, using Lemma~\ref{lem-weighted-Bernoulli-concentration} with $\lambda^* = (\log n)^{-4}$, $\theta = 1/(10\lambda^*)$, and $M = 10^5 c \log n$, we obtain that the event $\mathcal{S}_{\operatorname{good}} - \mathbb{E}[\mathcal{S}_{\operatorname{good}}] \ge 10^5 c \log n$ happens with probability at most  
\[
\exp\left(10^{-2} (\log n)^4 \mathbb{E}[\mathcal{S}_{\operatorname{good}}] - 10^4 c (\log n)^5\right) \le \exp\left(10^3 c (\log n)^5 - 10^4 c (\log n)^5\right)\,,
\]
which is super-polynomially small, as desired.
        
For the second probability, we only need to consider the case $T \ge n / (\log n)^{10}$. We have argued that $\mathcal{S}_{\operatorname{bad}}$ is stochastically dominated by $\mathbf{B}\big(\widetilde{O}(np),\, \widetilde{O}((np)^{-2})\big)$, which has mean $\widetilde{O}((np)^{-1}) \le n^{-\delta/2}$. A straightforward application of the Chernoff bound~\eqref{lem-Chernoff-bound} yields that $\mathbb{P}[\mathcal{S}_{\operatorname{bad}} \ge c \log n]$ is also super-polynomially small. This completes the proof.
\end{proof}

\subsubsection{Proof of Lemma~\ref{lem-quantitative-estimate}}\label{appendix-lem-quantitative-estimate} 

To prove Lemma~\ref{lem-quantitative-estimate}, we need the following quantitative central limit theorem (usually known as the Berry--Esseen bound) for sums of independent variables. The modern proof of this result mostly relies on Stein's method; see, e.g., \cite[Chapter 3]{CGS11}.

\begin{lemma}\label{lemma-quantitative-CLT}
   Let $X_1,\cdots,X_n$ be independent random variables with $\mathbb{E}[X_i]=0$ and $\mathbb{E}[|X_i|^3]<\infty$ for any $1\le i\le n$. Then it holds that for any $x\in \mathbb R$,  
   \[
   \left|\mathbb{P}[X_1+\cdots+X_n\ge x]-\Phi\Bigg(\frac{x}{\sqrt{\operatorname{Var}(X_1)+\cdots+\operatorname{Var}(X_n)}}\Bigg)\right|\le \frac{\sum_{i=1}^n\mathbb{E}[|X_i|^3]}{(\sum_{i=1}^n\operatorname{Var}(X_i))^{3/2}}\,,
   \]
   where $\Phi(t):=\frac{1}{\sqrt{2\pi}}\int_t^{\infty}e^{-u^2/2}\operatorname{d}\! u$ is the Gaussian tail. 
\end{lemma}

\begin{proof}[Proof of Lemma~\ref{lem-quantitative-estimate}]

Fix a good realization $\omega$ of $\mathcal{F}_{-\mathcal{I}}$, and let $\widetilde{\omega}$ be a further realization of $\mathcal{F}_{-\{a_i\}}$ (compatible with $\omega$). Recall that we denote $\widetilde{\mathbb{P}}=\mathbb{P}[\cdot \mid \widetilde{\omega}]$. Then, under $\widetilde{\mathbb{P}}$,
\[
S(i) = \sum_{I_m^{(t)} = 1} X_{m,a_i,b_i,c_i}^{(t)}
\]
is a sum of independent random variables with mean $0$ and variance $\widetilde{\sigma}_i^2 = \sigma_i^2 + o(1)$ (recall~\eqref{eq-deterministic-relation}).

Moreover, since $n - |S_m^{(t)}| \ge n^{1 - \delta/20}$ for $I_{m}^{(t)} = 1$, under $\mathcal{G}_{\operatorname{FGTL}}$ this implies that $d_{a_i} - |S_{m,a_i}^{(t)}| \ge \frac{np}{10} \cdot n^{-\delta/20}$. Thus,  
\[
\mathbb{E}\left[\left|(X_{m,a_i,b_i,c_i}^{(t)})^3\right|\right] \le \widetilde{O}\left((np)^{-3} \cdot n^{3\delta/20}\right)\,.
\]
Therefore, from Lemma~\ref{lemma-quantitative-CLT} we conclude that, for any $\widetilde{\omega}$,
\[
\left|\widetilde{\mathbb{P}}[S(i) \ge \chi] - \Phi\left({\chi}/{\widetilde{\sigma}_i}\right)\right| \le {\widetilde{\sigma}_i^{-3}} \sum_{m,t : I_{m}^{(t)} = 1} \widetilde{\mathbb{E}}\left[\left|(X_{m,a_i,b_i,c_i}^{(t)})^3\right|\right] \le \widetilde{O}\left((np)^{-1} \cdot n^{3\delta/20}\right) \le n^{-\delta/2}\,.
\]

Note that for $\sigma^2 \ge 10^{-3} \zeta c \log n - o(1)$, $\widetilde{\sigma}^2 = \sigma^2 - o(1)$, $\chi = 2\gamma \sqrt{c} \log n$, and any $x \in [\chi - 1, \chi + 1]$, using the asymptotics $\Phi(x)\sim x^{-1}e^{-x^2/2}/\sqrt{2\pi}$, we have
\[
|\Phi(x / \sigma) - \Phi(x / \widetilde{\sigma})| \le |x/\sigma-x/\widetilde{\sigma}|(e^{-x^2/2\sigma^2}+e^{-x^2/2\widetilde{\sigma}^2})= o(\Phi(x / \sigma))\,.
\]
This proves that $\widetilde{\mathbb P}[S(i)\ge x]=(1+o(1))\Phi(x/\sigma_i)$ for any $x\in [\chi-1,\chi+1]$, verifying the first statement of Lemma~\ref{lem-quantitative-estimate}. The second statement then follows from the iterated expectation theorem, and the proof is completed.
\end{proof}

\subsection{Proof of Proposition~\ref{prop-good-realization-typical}}\label{appendix-prop-good-realization-typical}

In this subsection, we show Proposition~\ref{prop-good-realization-typical}, that a realization of $\mathcal F_{-\mathcal I}$ is good with probability at least $0.95-o(1)$. 

\subsubsection{Property-(a)}

We first prove that \eqref{eq-good-realization-1} happens w.h.p.. To do so, we show that 
\begin{equation}\label{eq-expectation}
\mathbb{E}\left[\sum_{m=1}^M\sum_{t=0}^{T_m-1}\overline{I}_m^{(t)}U_m^{(t)}(i)\right]\le \widetilde{O}(n^{-\delta/4})\,.
\end{equation}
Provided this is true, since $\gamma<\delta/10$, it follows from Markov's inequality and the union bound that \eqref{eq-good-realization-1} happens w.h.p.. 

To prove~\eqref{eq-expectation}, we claim that for any $1 \le m \le M$ and $1 \le t \le T$,  
\begin{equation}\label{eq-claim}
\mathbb{E}\left[\overline{I}_m^{(t)} U_m^{(t)}(i)\right] \le n^{-10} + \widetilde{O}((np)^{-2}) \mathbb{P}\left[n - |S_m^{(t-1)}| \le 2n^{1 - \delta/20}\right]\,.
\end{equation}
Given this claim, we conclude that the left-hand side of~\eqref{eq-expectation} is upper-bounded by  
\[
n^{-10} \mathbb{E}\left[\sum_{m=1}^M T_m\right] + \widetilde{O}((np)^{-2}) \mathbb{E}\left[\sum_{m=1}^M \sum_{t=1}^{T_m - 1} \mathbf{1}\left\{n - |S_m^{(t)}| \le 2n^{1 - \delta/20}\right\}\right]\,.
\]
Using the tail bound of $T_m$ developed in the proof of Lemma~\ref{lem-MI} (see Appendix~\ref{appendix-lem-MI}), we find that the first part is $\widetilde{O}(n^{-6})$. For the second part, note that if $T \le n / (\log n)^{10}$, then under $\mathcal{G}_{\operatorname{RW}}$ the sum inside the expectation is void (and thus equals $0$). Therefore, the entire expression is upper-bounded by $\widetilde{O}(n^{-6})$ (since $\mathcal{G}_{\operatorname{RW}}$ happens with overwhelming probability). If $T \ge n / (\log n)^{10}$, then under $\mathcal{G}_{\operatorname{RW}}$ the sum is upper-bounded by $\widetilde{O}(M \cdot n^{1 - \delta/20}) \le \widetilde{O}((np)^2\cdot n^{-\delta/20})$. Hence, the second part is at most $\widetilde{O}(n^{-\delta/4})$. This concludes the proof of~\eqref{eq-expectation}.

To verify the claim, recall that for each $1\le m\le M$ and $0\le t\le T_m-1$, 
\begin{align} \label{U_m,t_def}
U_m^{(t)}{(i)}=\begin{cases}
    \frac{4}{d_{a_i}^2-|S_{m,a_i}^{(t)}|^2}\,,\quad&\text{ if }|S_{m,a_i}^{(t)}|\neq d_{a_i}\text{ and }x_{m,b_i}^{(t)}\neq x_{m,c_i}^{(t)}\,,\\
    \frac{2}{d_{a_i}}\,,\quad&\text{ if $|S_{m,a_i}^{(t)}|= d_{a_i}$ and $x_{m,b_i}^{(t)}\neq x_{m,c_i}^{(t)}$}\,,\\
    0\,,\quad&\text{ otherwise}\,.
\end{cases}
\end{align}
Hence, $U_m^{(t)}{(i)}$ is uniformly upper-bounded by $\frac{4}{np}$ and, under $\mathcal G_{\operatorname{MI}}\cap \mathcal G_{\operatorname{FTGL}}$, we have that
\[
U_m^{(t)}{(i)}\le \frac{4\cdot \mathbf{1}\{x_{m,b}^{(t)}\neq x_{m,c}^{(t)}\}}{np\max\{1,p(n-|S_m^{(t-1)}|)\mathbf{1}\{p(n-|S_m^{(t-1)}|)\ge 10(\log n)^2\}}\,.
\]
Since $\mathcal G_{\operatorname{MI}}\cap \mathcal G_{\operatorname{FGTL}}$ happens w.o.p. and $\{I_{m}^{(t)}=1\}\cap \mathcal G_{\operatorname{MI}}$ implies $n-|S_m^{(t-1)}|\le 2n^{1-\delta/20}$, we have that $\mathbb{E}[I_m^{(t)}U_m^{(t)}(i)]$ is at most
\[ n^{-10}+\mathbb{E}\left[\mathbf{1}\{\mathcal G_{\operatorname{FGTL}}\}\mathbf{1}\{(n-|S_m^{(t-1)}|\le 2n^{1-\delta/20}\}\cdot \frac{4\cdot \mathbf{1}\{x_{m,b}^{(t)}\neq x_{m,c}^{(t)}\}}{np[1+p(n-|S_m^{(t-1)}|)/(\log n)^{10}]}\right]\,.
\]
Conditioned on any realization of $x_{m,i}^{(t-1)}$ that does not violate $\mathcal{G}_{\operatorname{FGTL}} \cap \{n - |S_m^{(t-1)}| \le 2n^{1 - \delta/20}\}$, we have shown—as in the proof of Lemma~\ref{lem-stochastic-domination} (see Appendix~\ref{appendix-lem-stochastic-domination})—that the event $x_{m,b}^{(t)} \neq x_{m,c}^{(t)}$ happens with conditional probability at most  
\[
\frac{20 \max\{p(n - |S_m^{(t-1)}|),\, (\log n)^2\}}{np}\,.
\]
This implies that the above expectation is upper-bounded by $\widetilde{O}((np)^{-2}) \mathbb{P}\left[n - |S_m^{(t-1)}| \le 2n^{1 - \delta/20}\right]$. Thus, the claim follows, and the proof is completed.

\subsubsection{Property-(b)} 
We next show that Property-(b) holds with probability at least $0.95 - o(1)$; that is, at least half of the indices $i \in [n^\gamma]$ satisfy~\eqref{eq-good-realization-2}. Recall the definition of $\widetilde{T}_m$ for $1 \le m \le M$. Lemma~\ref{lem-effective-observations-lower-bound} states that the event  
\[
\widetilde{\mathcal{G}} := \left\{ \widetilde{T}_1 + \cdots + \widetilde{T}_M \ge \zeta \cdot M \cdot \min\{T, n\} \right\}
\]  
happens with probability at least $0.95$. We claim that for any $1 \le i \le n^\gamma$,
\begin{equation}\label{eq-sum}
\mathbb{P}\left[\widetilde{\mathcal{G}} \cap \mathcal{G}_{\operatorname{FGTL}},\ \sum_{m=1}^M \sum_{t=0}^{T_m - 1} I_m^{(t)} U_m^{(t)}(i) \le 10^{-3} \zeta c \log n \right] = o(1)\,.
\end{equation}
Given this, the desired result follows from Markov's inequality.

To show the claim, we condition on any realization of  
\[
\mathcal{F}_{-\{b,c\}} := \left\{x_{m,i}^{(t)} : 1 \le m \le M,\ 0 \le t \le T,\ i \in [n] \setminus \{b, c\} \right\}
\]  
that does not violate $\widetilde{\mathcal{G}} \cap \mathcal{G}_{\operatorname{FGTL}}$. Note that for each $1 \le m \le M$ and $0 \le t \le \widetilde{T}_m$, we have $n - |S_m^{(t)}| \ge \frac{n}{2}$, and thus $\mathcal{G}_{\operatorname{FGTL}}$ yields that  
\[
\frac{np}{20} \le d_i - |S_{i,m}^{(t)}| \le 2np\,.
\]  
Hence, the probability that $x_{m,b_i}^{(t)} \neq x_{m,c_i}^{(t)}$ is at least $1/50$. Moreover, it holds deterministically that  
\[
U_m^{(t)} \ge \frac{\mathbf{1}\{x_{m,b}^{(t)} \neq x_{m,c}^{(t)}\}}{4(np)^2}\,.
\]  
Using these observations, we get that under such a conditioning, the sum in~\eqref{eq-sum} stochastically dominates a binomial variable $X \sim \mathbf{B}(\zeta c \log n, 10^{-2})$. By Chernoff’s bound (Lemma~\ref{lem-Chernoff-bound}), we obtain that  
\[
\mathbb{P}[X \le 10^{-3} \zeta c \log n] = o(1)\,,
\]  
and hence the result follows. This completes the proof.

\subsubsection{Property-(c)}
Now we show that Property-(c) holds with high probability.  
Our goal is to prove the following:
\begin{align*}
&\qquad\qquad\quad\mathbb{E}\Bigg[\Bigg(\sum_{m,t} I_{m}^{(t)} \frac{x_{m,a_j}^{(t)} (x_{m,b_i}^{(t)} - x_{m,c_i}^{(t)})}{(d_{a_i} + x_{m,a_i}^{(t+1)} S_{m,a_i,-\mathcal{I}}^{(t)})^2} \Bigg)^2 \Bigg] \\
= \sum_{\substack{m_1,t_1 \\ m_2,t_2}} & \mathbb{E}\Bigg[ I_{m_1}^{(t_1)} I_{m_2}^{(t_2)} \cdot \frac{x_{m_1,a_j}^{(t_1)} (x_{m_1,b_i}^{(t_1)} - x_{m_1,c_i}^{(t_1)})}{(d_{a_i} + x_{m_1,a_i}^{(t_1+1)} S_{m_1,a_i,-\mathcal{I}}^{(t_1)})^2} 
\cdot \frac{x_{m_2,a_j}^{(t_2)} (x_{m_2,b_i}^{(t_2)} - x_{m_2,c_i}^{(t_2)})}{(d_{a_i} + x_{m_2,a_i}^{(t_2+1)} S_{m_2,a_i,-\mathcal{I}}^{(t_2)})^2} \Bigg] \le \widetilde{O}(n^{-\delta/4})\,,
\end{align*}
The result follows from Markov's inequality.

We partition the sum into two parts. The first part consists of those terms $(m_1, t_1)$ and $ (m_2, t_2)$ such that $m_1 = m_2$ and $|t_1 - t_2| \le \log n$, and the second part consists of the remaining terms. We will bound the two parts separately. The intuition behind this is that the first part contains relatively few terms, so their total contribution is small, while each term in the second part exhibits certain cancellations, making it small as well.

We begin by noting that for any $m, t$, due to the presence of $I_{m_1}^{(t_1)}$ and $I_{m_2}^{(t_2)}$, it holds deterministically that  
\[
\left| I_{m}^{(t)} \cdot \frac{x_{m,a_j}^{(t)} (x_{m,b_i}^{(t)} - x_{m,c_i}^{(t)})}{(d_{a_i} + x_{m,a_i}^{(t+1)} S_{m,a_i,-\mathcal{I}}^{(t)})^2} \right| \le \widetilde{O}((np)^{-2} \cdot n^{\delta/10})\,.
\]
Thus, the first part of the sum is upper-bounded by  
\[
\widetilde{O}((np)^{-4} \cdot n^{\delta/10}) \times 2 \log n \times 
\begin{cases}
M \cdot T\,, & \text{if } T \le n\,, \\
\sum_{m=1}^M T_m\,, & \text{if } T > n\,.
\end{cases}
\]
For the case $T \le n$, we have $M \cdot T = c n^2 p^2 \log n = \widetilde{O}((np)^2)$; and for $T > n$, $M = \widetilde{O}(np^2)$, so $\mathbb{E}\left[\sum_{m=1}^M T_m\right] = \widetilde{O}((np)^2)$.  
Consequently, the expected value of the first part of the sum is at most $\widetilde{O}((np)^{-2} \cdot n^{\delta/5}) \le \widetilde{O}(n^{-\delta/4})$.

Next, We handle the second part. Fix $(m_1, t_1)$ and $ (m_2, t_2)$ such that $I_{m_1}^{(t_1)} = I_{m_2}^{(t_2)} = 1$ and either $m_1 \neq m_2$ or $m_1 = m_2$ but $|t_1 - t_2| > \log n$. We claim that  
\begin{equation}\label{eq-crossing-term-upper-bound}
\left| \mathbb{E}\left[ I_{m_1}^{(t_1)} I_{m_2}^{(t_2)} \cdot \frac{x_{m_1,a_j}^{(t_1)} (x_{m_1,b_i}^{(t_1)} - x_{m_1,c_i}^{(t_1)})}{(d_{a_i} + x_{m_1,a_i}^{(t_1+1)} S_{m_1,a_i,-\mathcal{I}}^{(t_1)})^2} \cdot \frac{x_{m_2,a_j}^{(t_2)} (x_{m_2,b_i}^{(t_2)} - x_{m_2,c_i}^{(t_2)})}{(d_{a_i} + x_{m_2,a_i}^{(t_2+1)} S_{m_2,a_i,-\mathcal{I}}^{(t_2)})^2} \right] \right| \le \widetilde{O}((np)^{-4} \cdot n^{-\delta/4})\,.
\end{equation}

Provided this claim holds, the desired result follows. For $T \le n$, there are at most $\widetilde{O}((np)^4)$ terms in the sum, so the expected value of the sum is $\widetilde{O}(n^{-\delta/4})$. For $T \ge n$, we have $M = \widetilde{O}(np^2)$. Then for each $1 \le m \le M$ and $t \le n (\log n)^2$, we use the bound from~\eqref{eq-crossing-term-upper-bound}. For $t \ge n (\log n)^2$, we use the tail bound on $T_m$ developed in the proof of Lemma~\ref{lem-MI} (see Appendix~\ref{appendix-lem-MI}). Altogether, we conclude that the expected value is $\widetilde{O}(n^{-\delta/4})$.

It remains to prove~\eqref{eq-crossing-term-upper-bound}. Fix any $1 \le m \le M$, and for any $0 \le t \le T_m$, let $\mathcal{F}_t$ be the $\sigma$-field generated by the random variables $\{x_{m,i}^{(t')} : t' \le t,\ i \in [n]\}$. For any $0 \le t_0 \le T$, define $\tau_0 := \max\{-1, t_0 - \log n\}$. We will prove the following stronger result: Conditioned on any realization of $\mathcal{F}_{\tau_0}$, it holds that  
\begin{equation}\label{eq-coupling-goal}
\mathbb{E}\left[\left| \mathbb{E}\left[ I_{m}^{(t_0)} \cdot \frac{x_{m,a_j}^{(t_0)} (x_{m,b_i}^{(t_0)} - x_{m,c_i}^{(t_0)})}{\left(d_{a_i} + x_{m,a_i}^{(t_0+1)} S_{m,a_i,-\mathcal{I}}^{(t_0)}\right)^2} \,\middle|\, \mathcal{F}_{\tau_0} \right] \right| \right] = \widetilde{O}((np)^{-2} \cdot n^{-9\delta/20})\,.
\end{equation}

To see why~\eqref{eq-coupling-goal} implies~\eqref{eq-crossing-term-upper-bound}, we first consider the case where $m_1 \ne m_2$. By independence, the left-hand side of~\eqref{eq-crossing-term-upper-bound} decomposes as  
\[
\left| \mathbb{E}\left[ I_{m_1}^{(t_1)} \cdot \frac{x_{m_1,a_j}^{(t_1)} (x_{m_1,b_i}^{(t_1)} - x_{m_1,c_i}^{(t_1)})}{\left(d_{a_i} + x_{m_1,a_i}^{(t_1+1)} S_{m_1,a_i,-\mathcal{I}}^{(t_1)}\right)^2} \right] \cdot \mathbb{E}\left[ I_{m_2}^{(t_2)} \cdot \frac{x_{m_2,a_j}^{(t_2)} (x_{m_2,b_i}^{(t_2)} - x_{m_2,c_i}^{(t_2)})}{\left(d_{a_i} + x_{m_2,a_i}^{(t_2+1)} S_{m_2,a_i,-\mathcal{I}}^{(t_2)}\right)^2} \right] \right|\,.
\]
By~\eqref{eq-coupling-goal}, both terms are of order $\widetilde{O}((np)^{-2} \cdot n^{-9\delta/20})$, and hence the result follows. For the case where $m_1 = m_2$ and $|t_1 - t_2| > \log n$, without loss of generality we assume $t_1 > t_2$, so $t_1 > \tau_1 := t_1 - \log n \ge t_2 + 1$. Since  
\[
\left| I_{m_2}^{(t_2)} \cdot \frac{x_{m_2,a_j}^{(t_2)} (x_{m_2,b_i}^{(t_2)} - x_{m_2,c_i}^{(t_2)})}{\left(d_{a_i} + x_{m_2,a_i}^{(t_2+1)} S_{m_2,a_i,-\mathcal{I}}^{(t_2)}\right)^2} \right| \le \widetilde{O}((np)^{-2} \cdot n^{\delta/10})
\]
holds deterministically, by applying the law of iterated expectations with respect to $\mathcal{F}_{\tau_1}$ and using the triangle inequality, the bound in~\eqref{eq-crossing-term-upper-bound} also follows from \eqref{eq-coupling-goal} in this case.

Now we focus on the proof of \eqref{eq-coupling-goal}. Henceforth, we fix the choices of $1\le m\le M,0\le t_0\le T,a_i,b_i,c_i,a_j\in [n]$, and the set $\mathcal I\subseteq [n]$. For simplicity, we omit the parameter $m$ in the subscript. Recall the duality between the dynamics and coalescing random walks: For each $1\le t\le T$ and each $i\in [n]$, we independently select an arrow $a_{i}^t$ that points from $i$ to a uniformly random vertex in the out-neighborhood of $i$ in $G^*$. Then, for any $t\le t'$, let $i=i_0,i_1,\cdots,i_{t}$ be the backward random walk path starting at $i$ following the arrows. We then have $x_i^{(t')}=x_{i_1}^{(t'-1)}=\cdots=x_{i_{t}}^{(t'-t)}$. 

We first deal with the simple case where $t_0<\log n$, $\mathcal F_{\tau_0}=\varnothing$, so there is no conditioning in \eqref{eq-coupling-goal}. 
Using similar concentration arguments as in the proof of Lemma~\ref{lem-FGTL} (see Appendix~\ref{appendix-lem-FGTL}), it can be shown that w.o.p. $|S_{a_i,-\mathcal I}^{(t_0)}|\le \widetilde{O}\big((np)^{1/2}\big)$. Thus, under this event, we have
\[
\left|\frac{1}{(d_{a_i}+x_{a_i}^{(t_0+1)}S_{a_i,-\mathcal I}^{(t_0)})^2}-\frac{1}{d_{a_i}^2}\right|\le \widetilde{O}\big((np)^{-5/2}\big)\le\widetilde{O}\big((np)^{-2}\cdot n^{-\delta/2}\big)\,.
\]
Hence, it remains to show that
\[
\frac{1}{d_{a_i}^2}\mathbb{E}[x_{a_j}^{(t_0)}(x_{b_i}^{(t_0)}-x_{c_i}^{(t_0)})]=\widetilde{O}\big((np)^{-2}\cdot n^{-9\delta/20}\big)\,.
\]
By the duality to coalescing random walks, we have $\mathbb{E}[x_{a_j}^{(t_0)}x_{b_i}^{(t_0)}]$, which represents the probability that two independent random walks starting at $a_j$ and $b_i$ intersect before time $t_0<\log n$, which is of order $\widetilde{O}((np)^{-1})$. Similarly, we have $\mathbb{E}[x_{a_j}^{(t_0)}x_{c_i}^{(t_0)}]=\widetilde{O}((np)^{-1})$. Since $d_{a_i}=\Theta(np)$, the above is $\widetilde{O}((np)^{-3})\le \widetilde{O}((np)^{-2}\cdot n^{-\delta})$, and the result follows.

Next, We consider the case where $t_0\ge \log n$. At this point, we fix a realization of $\mathcal F_{\tau_0}$. Let $\mathbb{P}_{\tau_0,t_0}$ denote the probability distribution of the random arrows $a_{i}^t,i\in [n],\tau_0<t\le t_0+1$. We note that $\mathbb{P}_{\tau_0,t_0}$ is a product distribution of uniformly random arrows that are independent of $\mathcal F_{\tau_0}$. Moreover, the two terms
\[
I^{(t_0)}\cdot\frac{x_{a_j}^{(t_0)}x_{b_i}^{(t_0)}}{(d_{a_i}+x_{a_i}^{(t_2+1)}S_{a_i,-\mathcal I}^{(t_0)})^2}\,,\quad I^{(t_0)}\cdot\frac{x_{a_j}^{(t_0)}x_{c_i}^{(t_0)}}{(d_{a_i}+x_{a_i}^{(t_0+1)}S_{a_i,-\mathcal I}^{(t_0)})^2}
\]
are both measurable with respect to the $\sigma$-field of $\mathbb{P}_{\tau_0,t_0}$ given the realization of $\mathcal F_{\tau_0}$. 

Recall Theorem~\ref{thm:mixing} and denote $L=c(1,\delta)$. This means that a random walk on $G^*$ mixes up to TV-error $O(n^{-1})$ after $L$ steps. We construct a coupling $\Gamma=(\{a_{i}^t\},\{\widetilde{a}_i^t\})$ of $\mathbb{P}_{\tau_0,t_0}$ with itself. By the mixing property, there exists a coupling $\gamma$ of the backward random walk paths $\{B_t\}_{\tau_0\le t\le t_0}$ and $\{C_t\}_{\tau_0\le t\le t_0}$ starting at $B_{t_0}=b_i$ and $C_{t_0}=c_i$ such that $\gamma[B_t=C_t,\forall\tau_0\le t\le t_0-L]\ge 1-O(n^{-1})$. We define $(\{a_i^t\},\{\widetilde{a}_i^t\})\sim \Gamma$ as follows: First, sample $\{a_i^t\}\sim \mathbb{P}_{\tau_0,t_0}$; let $\{B_t\}_{\tau_0\le t\le t_0}$ be the backward random walk starting at $B_{t_0}=b_i$, following the arrows $a_i^t$. Then, we define $\{C_t\}_{\tau_0\le t\le t_0}$ as a backward random walk starting at $C_{t_0}=c_i$, conditionally sampled from $\gamma$ given $\{B_t\}_{\tau_0\le t\le t_0}$. We then define $\{\widetilde{a}_i^t\}$ as follows: We set $\widetilde{a}_i^t=a_i^t$ for all $i\in [n]$ and $\tau_0<t\le t_0$. For each $\tau_0<t\le t_0$, we modify the arrows $a_{C_t}^t$ to point to $C_{t-1}$. See Figure~\ref{fig:coupling} for an illustration. It is straightforward to check that $\{\widetilde{a}_i^t\}$ has the same distribution as $\mathbb{P}_{\tau_0,t_0}$.

\begin{figure}
    \centering
    \includegraphics[width=0.8\linewidth]{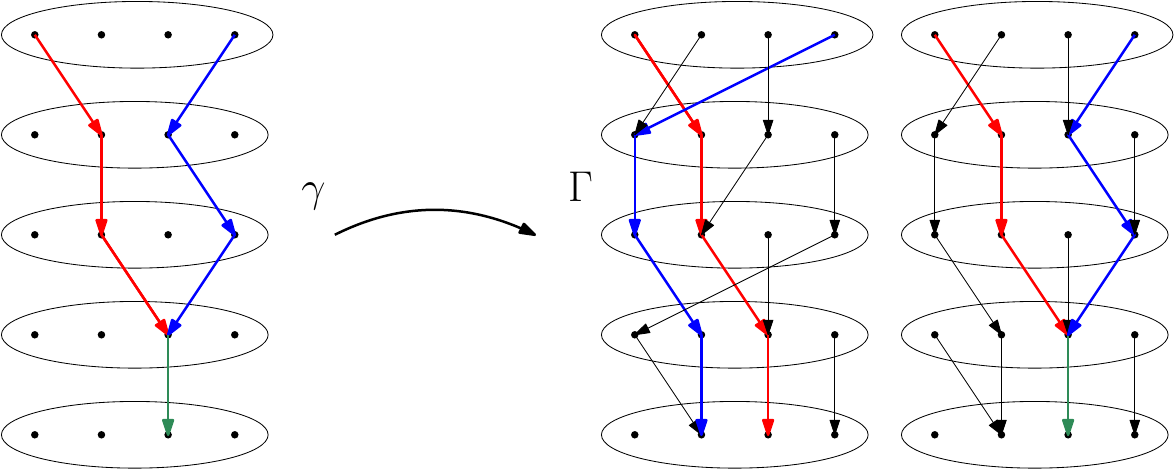}
    \caption{Obtaining the coupling $\Gamma$ from the coupling $\gamma$}
    \label{fig:coupling}
\end{figure}

For $i\in [n]$ and $\tau_0<t\le t_0$, let $x_i^{(t)}$ and $\widetilde{x}_i^{(t)}$ be the labels corresponding to the arrows $\{a_i^t\}$ and $\{\widetilde{a}_i^t\}$, respectively. Define $I^{(t)},\widetilde{I}^{(t)},S_{a_i,-\mathcal I}^{(t)},$ and $\widetilde{S}_{a_i,-\mathcal I}^{(t)}$ accordingly. We have that
\begin{align}
&\ \nonumber\qquad\qquad\qquad\Bigg|\mathbb{E}\Bigg[I_{m}^{(t_0)}\cdot\frac{x_{m,a_j}^{(t_0)}(x_{m,b_i}^{(t_0)}-x_{m,c_i}^{(t_0)})}{(d_{a_i}+x_{m,a_i}^{(t_0+1)}S_{m,a_i,-\mathcal I}^{(t_0)})^2}\mid \mathcal F_{\tau_0}\Bigg]\Bigg|\\
=&\ \Bigg|\mathbb{E}_{(\{a_i^t\},\{\widetilde{a}_i^t\})\sim \Gamma}\Bigg[I^{(t_0)}\cdot\frac{x_{a_j}^{(t_0)}x_{b_i}^{(t_0)}}{(d_{a_i}+x_{a_i}^{(t_0+1)}S_{a_i,-\mathcal I}^{(t_0)})^2}-\widetilde{I}^{(t_0)}\cdot\frac{\widetilde{x}_{a_j}^{(t_0)}\widetilde{x}_{b_i}^{(t_0)}}{(d_{a_i}+\widetilde{x}_{a_i}^{(t_0+1)}\widetilde{S}_{a_i,-\mathcal I}^{(t_0)})^2}\Bigg]\Bigg|\,.\label{eq-Gamma}
\end{align}

To bound \eqref{eq-Gamma}, we define three good events $\mathcal E_1,\mathcal E_2$, and $\mathcal E_3$, and show that all of them hold with probability $1-\widetilde{O}(n^{-\delta})$. Let $\mathcal E_1$ be the event that for any $i\in [n]$ and $t_0-L\le t\le t_0$, the number of arrows $a_j^t$ pointing to $i$ is at most $\log n$. It is straightforward to see that for any $i\in [n]$ and $t_0-L\le t\le t_0$, the number of arrows $a_j^t$ pointing to $i$ is stochastically dominated by $\mathbf{B}(2np,2/np)$. Thus, from Chernoff's bound (Lemma~\ref{lem-Chernoff-bound}) and the union bound, we conclude that $\Gamma[\mathcal E_1]\ge1-O(n^{-1})$. Let $\mathcal E_2$ be the event that $B_t=C_t,\forall \tau_0\le t\le t_0-L$, which happens with probability at least $1-O(n^{-1})$ by our coupling. Finally, let $\mathcal E_3:=\{x_{a_i}^{(t_0+1)}= \widetilde{x}_{a_j}^{(t_0+1)},x_{a_j}^{(t_0)}=\widetilde{x}_{a_j}^{(t_0)}\}$. We observe that under $\mathcal E_2$, for $t^*\in \{t_0,t_0+1\}$, $x_i^{(t^*)}\neq \widetilde{x}_i^{(t^*)}$ only happens when the backward random walk starting at $i$ at time $t^*$ following the arrows $a_i^t$ intersects with $C_t$ at some time $t_0-L\le t\le t_0$. It follows that $\Gamma[\mathcal E_2\cap \mathcal E_3]=1-O\big((np)^{-1}\big)=1-\widetilde{O}(n^{-\delta})$.

Now, we assume that $\mathcal E:=\mathcal E_1\cap \mathcal E_2\cap \mathcal E_3$ happens. Using the above observation, we conclude that under $\mathcal E$, there are at most $O((\log n)^L)=\widetilde{O}(1)$ indices $i\in [n]$ such that $x_i^{(t_0)}\neq \widetilde{x}_i^{(t_0)}$, and thus $|S_{a_i,-\mathcal I}^{(t_0)}-\widetilde{S}_{a_i,-\mathcal I}^{(t_0)}|\le \widetilde{O}(1)$. As a result, under $\mathcal E$, if $I^{(t_0)}=\widetilde{I}^{(t_0)}=1$, we have that
\begin{align*}
&\ \qquad\qquad\qquad\Bigg|\frac{x_{a_j}^{(t_0)}x_{b_i}^{(t_0)}}{(d_{a_i}+x_{a_i}^{(t_0+1)}S_{a_i,-\mathcal I}^{(t_0)})^2}-\frac{\widetilde{x}_{a_j}^{(t_0)}\widetilde{x}_{b_i}^{(t_0)}}{(d_{a_i}+\widetilde{x}_{a_i}^{(t_0+1)}\widetilde{S}_{a_i,-\mathcal I}^{(t_0)})^2}\Bigg|\\
=&\ \Bigg|x_{a_j}^{(t_0)}x_{b_i}^{(t_0)}\cdot \frac{x_{a_i}^{(t_0+1)}\big(S_{a_i,-\mathcal I}^{(t_0)}-\widetilde{S}_{a_i,-\mathcal I}^{(t_0)}\big)\big(2d_{a_i}+x_{a_i}^{(t_0+1)}(S_{a_i,-\mathcal I}^{(t_0)}+\widetilde{S}_{a_i,-\mathcal I}^{(t_0)})\big)}{(d_{a_i}+x_{a_i}^{(t_0+1)}S_{a_i,-\mathcal I}^{(t_0)})^2(d_{a_i}+x_{a_i}^{(t_0+1)}\widetilde{S}_{a_i,-\mathcal I}^{(t_0)})^2}\Bigg|=\widetilde{O}\big((np)^{-3}\cdot n^{\frac{\delta}{5}}\big)\,.
\end{align*}
Additionally, if $I^{(t_0)}\neq \widetilde{I}^{(t_0)}$, then (recalling that $I^{(t_0)}$ is the indicator of $\{n-|S_{-\mathcal I}^{(t_0)}|\ge n^{1-\delta/20}\}$ and similarly for $\widetilde{I}^{(t_0)}$) 
$$\big|n-|\mathcal S_{-\mathcal I}^{(t_0)}|-n^{1-\delta/20}\big|=\widetilde{O}(1)\,.$$  

In conclusion, we obtain that the right hand side of \eqref{eq-Gamma} is upper-bounded by 
\begin{align*}
&\widetilde{O}\big((np)^{-2}\big)\Gamma[\mathcal E^c]+\widetilde{O}\big((np)^{-3} n^{\frac{\delta}{5}}\big)+\widetilde{O}\big((np)^{-2}\big)\mathbb{E}\Big[\Gamma\Big[\big\{\big|n-|\mathcal S_{-\mathcal I}^{(t_0)}|-n^{1-\frac{\delta}{20}}\big|=\widetilde{O}(1)\big\}\mid \mathcal F_{\tau_0}\Big]\Big]\,.
\end{align*} 
It remains to show that the last expectation term, which equals (using the iterative expectation theorem) $$\mathbb{P}\Big[\big\{\big|n-|\mathcal S_{-\mathcal I}^{(t_0)}|-n^{1-\delta/20}\big|=\widetilde{O}(1)\big\}\Big]\,,$$ is of order $\widetilde{O}(n^{-9\delta/20})$. 

By losing a super-polynomially small probability, we may assume that $\mathcal G_{\operatorname{MI}}\cap \mathcal G_{\operatorname{FGTL}}$ holds. We further condition on the labels $x_i^{(t_0-1)},i\in [n]$. According to $\mathcal G_{\operatorname{MI}}$, $n-|S_{a_i,-\mathcal I}^{(t_0)}|$ cannot be $\widetilde{O}(1)$-close to $n^{1-\delta/20}$ unless $n-|S_{a_i}^{(t_0-1)}|\in [\frac{1}{2}n^{1-\delta/20} n,2n^{1-\delta/20}]$. If this is the case, from $\mathcal G_{\operatorname{FGTL}}$ we know that for any $i\in [n]$,
\[
p_i^{(t_0)}=\frac{N_i\cap \operatorname{-1}^{(t_0-1)}}{d_i}\text{ satisfies }\min\{p_i^{(t_0)},1-p_i^{(t_0)}\}\ge \frac{1}{100n^{\delta/20}}\,.
\]
Note that conditioned on $x_i^{(t_0-1)},i\in [n]$, $\mathcal S_{-\mathcal I}^{(t_0)}$ is an independent sum given by $$
S_{-\mathcal I}^{(t_0)}=\sum_{i\in [n]\setminus \mathcal I}x_i^{(t_0)}\,,\quad
x_{i}^{(t_0)}\sim 2\mathbf{B}(p_i^{(t_0)})-1,i\in [n]\setminus\mathcal I\,.$$
By the local central limit theorem for sums of independent random variables (see, e.g. \cite{MALLER1978101}), since the variance of $\mathcal S_{-\mathcal I}^{(t_0)}$ is of order at least $\Theta(n^{1-\delta/20})$, it follows that
\[
\sup_{k\in \mathbb{N}}\mathbb{P}[S_{-\mathcal I}^{(t_0)}=k]\le \widetilde{O}(n^{-1/2+\delta/40})\,.
\]
Consequently, we obtain that $\mathbb{P}\big[\big|n-|\mathcal S_{-\mathcal I}^{(t_0)}|-n^{1-\delta/20}\big|=\widetilde{O}(1)\big]=\widetilde{O}(n^{-1/2+\delta/40})\le n^{-9\delta/20}$. This proves \eqref{eq-coupling-goal} for the case $t\ge L$ and concludes the proof. 

\subsection{Proof of Proposition~\ref{prop-ovelrine-S(i)}}\label{appendix-overline-S(i)-small}
Finally we prove Proposition~\ref{prop-ovelrine-S(i)}.

\begin{proof}
Fix a good realization $\omega$ of $\mathcal F_{-\mathcal I}$ and an index $i \in [n^\gamma]$. Our goal is to upper-bound
\[
\overline{S}(i) = \sum_{m=1}^{M}\sum^{T_m-1}_{t=0}\overline{I}_{m}^{(t)}X_{m,a_i,b_i,c_i}^{(t)} = \sum_{\overline{I}_m^{(t)}=1} X_{m,a_i, b_i, c_i}^{(t)} 
\]
by showing that $\mathbb{P}[\overline{S}(i)\ge \gamma\sqrt{c}\log n]\le \widetilde{O}(n^{-\delta/10})$. Since $\gamma<\delta/10$, the result then follows from the union bound.
Roughly speaking, the tail bound is obtained by applying martingale concentration inequalities, but we need to handle several technicalities. 

We will keep using the observation that if we further condition on a realization $\widetilde{\omega}$ of $\mathcal F_{-\{a_i\}}$ that is compatible with $\omega$, then all random variables $X_{m,a_i,b_i,c_i}$ become conditionally independent, where randomness comes only from $x_{m,a_i}^{(t+1)}\sim 2\mathbf{B}\big(1,\frac{d_{a_i}+S_{m,a_i}^{(t)}}{2d_{a_i}}\big)-1,0\le t\le T_m-1$. Moreover, recall that (denoting by $\widetilde{\mathbb{P}}$ the conditional measure given $\widetilde{\omega}$)
$$
\widetilde{\mathbb{E}}[X_{m,t}]=0\,, \,\quad
\operatorname{Var}_{\widetilde{P}}[X_{m,a_i,b_i,c_i}^{(t)}]=U_{m}^{(t)}(i)\,.
$$

%For each $1 \leq m \leq M$ and $0\leq t\leq T_m-1$, %write
%$$
%X_{m,t}=X_{m,a_i, b_i, c_i}^{(t)}\overline{I}_m^{(t)}=\frac{x_{m,a_i}^{(t+1)}(x_{m,b_i}^{(t)}-x_{m,c_i}^{(t)})}{d_{a_i}+x_{m,a_i}^{(t+1)}S_{m,a_i}^{(t)}}\overline{I}_m^{(t)}.
%$$
%Observe that We have a natural bound for $X_{m,t}$ as follows.
%$$
%\left|\frac{x_{m,a_i}^{(t+1)} (x_{m,b_i}^{(t)}-x_{m,c_i}^{(t)})}{d_{a_i}+x_{m,a_i}^{(t+1)}S_{m,a_i}^{(t)}}\right| \leq \frac{ |x_{m,b_i}^{(t)}-x_{m,c_i}^{(t)}|}{d_{a_i}-|S_{m,a_i}^{(t)}|} \leq \frac{2}{2}=1.
%$$

Define $\mathcal{Q}_{\text{good}}$ as the set of pairs $(m,t)$ with $\overline{I}_m^{(t)}=1$ such that $d_{a_i}-|S_{m,a_i}^{(t)}|\geq \log n$. Let $\mathcal{Q}_{\text{bad}}$ be the remaining pairs. We write
\[
\overline{S}(i)=\sum_{(m,t)\in \mathcal Q_{\operatorname{good}}}X_{m,a_i,b_i,c_i}^{(t)}+\sum_{(m,t)\in \mathcal Q_{\operatorname{bad}}}X_{m,a_i,b_i,c_i}^{(t)}\triangleq\overline{\mathcal S}_{\operatorname{good}}+\overline{\mathcal S}_{\operatorname{bad}}\,.
\]
In what follows, we show separately that
\[
\mathbb{P}\left[\overline{\mathcal S}_{\operatorname{good}}\ge \frac{\gamma\sqrt c\log n}{2}\right]\le \widetilde{O}(n^{-\delta/10})\,,\quad \mathbb{P}\left[\overline{\mathcal S}_{\operatorname{bad}}\ge \frac{\gamma\sqrt c\log n}{2}\right]\le \widetilde{O}(n^{-\delta/10})\,.
\]

We first handle $\overline{\mathcal S}_{\operatorname{good}}$. We call a realization $\widetilde{\omega}$ \emph{nice} if 
\[
\sum_{m=1}^M\sum_{t=0}^{T_m-1}\overline{I}_m^{(t)}U_m^{(t)}(i)\le \gamma^2c\log n\,.
\]
It follows from Property-(a) of good realizations (see \eqref{eq-good-realization-1}) that $\widetilde{\omega}$ is nice with probability at least $1-n^{-\delta/10}$. Additionally, for any nice realization $\widetilde{\omega}$, we have
\[
\operatorname{Var}_{\widetilde{\mathbb P}}[\overline{\mathcal S}_{\operatorname{good}}]=\sum_{(m,t)\in \mathcal Q_{\operatorname{good}}}U_m^{(t)}\le \gamma^2c\log n\,.
\]
Under $\widetilde{\mathbb P}$, since $\overline{\mathcal S}_{\operatorname{good}}$ is a sum of independent variables of mean zero, it is automatically a martingale. Moreover, by the definition of $\mathcal Q_{\operatorname{good}}$, we have that for every $(m,t)\in \mathcal Q_{\operatorname{good}}$,
$$
|X_{m,a_i,b_i,c_i}^{(t)}| \leq \frac{ |x_{m,b_i}^{(t)}-x_{m,c_i}^{(t)}|}{d_{a_i}-|S_{m,a_i}^{(t)}|} \leq \frac{2}{\log n}.
$$
Therefore, applying the concentration inequality for martingales with bounded difference (see, e.g. \cite[Theorem 6]{chung2006concentration}), we get
\begin{align*}
\widetilde{\mathbb{P}}\left[\overline{\mathcal S}_{\operatorname{good}}\geq \frac{\gamma\sqrt{c}}{2}\log n\right] &\leq  \exp\left(-\frac{\gamma^2c(\log n)^2/4}{2\sum_{(m,t)\in\mathcal{Q}_{\text{good}}}\text{var}(X_{m,a_i,b_i,c_i})+3\gamma\sqrt{c}}\right)\\
&\leq \exp\left(-\frac{\gamma^2c (\log n)^2}{8\gamma^2c\log n+12\gamma\sqrt{c}}\right)\le n^{-0.11}\le n^{-\delta/10}\,.
\end{align*}

Now we deal with $\overline{\mathcal S}_{\operatorname{bad}}$. We note that for each $(m,t)\in \mathcal Q_{\operatorname{bad}}$, $|X_{m,a_i,b_i,c_i}^{(t)}|$ equals
\[
\left|\frac{x_{m,a_i}^{(t+1)} (x_{m,b_i}^{(t)}-x_{m,c_i}^{(t)})}{d_{a_i}+x_{m,a_i}^{(t+1)}S_{m,a_i}^{(t)}}\right|=\begin{cases}
    \frac{2}{d_{a_i}+|S_{m,a_i}^{(t)}|}\leq \frac{2}{d_{a_i}}\,,&\text{ if }x_{m,a_i}^{(t+1)}S_{m,a_i}^{(t)}\ge 0\text{ and }x_{m,b_i}^{(t)}\neq x_{m,c}^{(t)}\,,\\
    \frac{2}{d_{a_i}-|S_{m,a_i}^{(t)}|} \leq 1\,,\quad&\text{ if $x_{m,a}^{(t+1)}S_{m,a_i}^{(t)}<0$ and $x_{m,b_i}^{(t)}\neq x_{m,c_i}^{(t)}$}\,,\\ 
    0\,,\quad&\text{ otherwise}\,.
\end{cases}
\]
We first show that, with probability at least $1-O(n^{-\delta})$, the second case never happens. To establish this, we employ the union bound, which requires us to first upper-bound $|\mathcal{Q}_{\text{bad}}|$. By Lemmas~\ref{lem-FGTL} and \ref{lem-RW}, we may assume that $\mathcal G_{\operatorname{FGTL}}\cap \mathcal G_{\operatorname{RW}}$ holds, losing only a super-polynomially small probability. If $(m,t)\in \mathcal{Q}_{\text{bad}}$, then for some $i\in[n^\gamma]$ we have $d_{a_i}-|S_{m,a_i}^{(t)}|\leq \log n$. By $\mathcal G_{\operatorname{FGTL}}$, this implies that $n-|S_m^{(t)}|\leq \widetilde{O}(p^{-1})$. Furthermore, by Item-(i) of $\mathcal G_{\operatorname{RW}}$, this only happens when $T\ge n/(\log n)^{10}$, in which case we have $M=\widetilde{O}(np^2)$. Then, by Item-(ii) of $\mathcal G_{\operatorname{RW}}$, the number of such pairs is at most $M\times \widetilde{O}(p^{-1})=\widetilde{O}(np)$. Thus, we conclude that whenever $\mathcal G_{\operatorname{FTGL}}\cap \mathcal G_{\operatorname{RW}}$ holds, it must follow that $|\mathcal Q_{\operatorname{bad}}|\le \widetilde{O}(np)$.

Now, let us upper-bound the probability that the second case happens. The argument in the proof of Lemma~\ref{lem-stochastic-domination}  (see Appendix~\ref{appendix-lem-stochastic-domination}) yields that the event $x_{m,b_i}^{(t)}\neq x_{m,c_i}^{(t)}$ occurs with probability at most 
$$
\frac{40}{np}\cdot \max\{p(n-|S_{m}^{(t-1)}|),(\log n)^2\} = \widetilde{O}\left(\frac{1}{np}\right).
$$
Moreover, if we further condition on $S_{m,a_i}^{(t)}$, then $x_{m,a_i}^{(t+1)}S_{m,a_i}^{(t)}<0$ occurs with conditional probability 
$$
\frac{d_{a_i}-|S_{m,a_i}^{(t)}|}{2d_{a_i}}\le \frac{2(d_{a_i}-|S_{m,a_i}^{(t)}|)}{np} =  \widetilde{O}\left(\frac{1}{np}\right).
$$
Combining these bounds, the expected number of pairs $(m,t)\in \mathcal Q_{\operatorname{bad}}$ that satisfy the second case is at most
$$
\widetilde{O}(np) \times  \widetilde{O}\left(\frac{1}{np}\right)\times  \widetilde{O}\left(\frac{1}{np}\right) =  \widetilde{O}\left(\frac{1}{np}\right)= \widetilde{O}\left(n^{-\delta}\right)\,,
$$
as claimed.

Given that $\mathcal G_{\operatorname{FTGL}}\cap \mathcal G_{\operatorname{RW}}$ holds and the second case never happens, since $d_{a_i}\ge np/2$, we conclude that $\overline{\mathcal S}_{\operatorname{bad}}$ is stochastically dominated by $$\frac{4}{np}\mathbf{B}(\widetilde{O}(np),\widetilde{O}({(np)^{-2}}))\,.$$
Applying Chernoff's bound, we obtain that the probability that the inequality $\overline{\mathcal S}_{\operatorname{bad}}\ge \frac{\gamma\sqrt{c}}{2}\log n$ holds is super-polynomially small. This yields the desired tail estimate for $\overline{\mathcal S}_{\operatorname{bad}}$, thereby concluding the proof.
\end{proof}

\bibliographystyle{alpha}
\small
\bibliography{ref}

\end{document}